\newif\ifspringer
\springertrue
\springerfalse

\ifspringer
\documentclass[graybox]{svmult}
\else
\documentclass[a4paper,11pt]{article}
\usepackage{theorem}         
\fi

\usepackage{amsmath}
\usepackage{amsfonts}
\usepackage{amssymb}         
\usepackage{amsopn}
\usepackage{graphicx}        
\usepackage{mathrsfs}		
\usepackage{enumerate}
\usepackage{todonotes}
\usepackage{mathptmx}       
\usepackage{makeidx}         
\usepackage{graphicx,epsfig,color}
\graphicspath{{img}}
\usepackage{multicol}        
\usepackage{multirow}
\usepackage[bottom]{footmisc}
\usepackage{url}
\usepackage{amsfonts}
\usepackage{wrapfig}
\usepackage{amsmath}
\usepackage{algorithm}
\usepackage{algorithmic}
\usepackage{soul}
\usepackage{tikz}
\usetikzlibrary{graphs,quotes}

\ifspringer
\else
\usepackage{tocbasic}
\DeclareTOCStyleEntry[beforeskip=-.3em]{tocline}{section}
\DeclareTOCStyleEntry[beforeskip=-.5em]{tocline}{subsection}
\DeclareTOCStyleEntry[beforeskip=-.7em]{tocline}{subsubsection}
\fi

\newcommand{\transpose}[1]{{#1}^\top}
\newcommand{\ran}{\mathsf{ran}}
\newcommand{\dom}{\mathsf{dom}}
\newcommand{\trace}[1]{\mathsf{tr}(#1)}
\newcommand{\diag}[1]{\mathsf{diag}(#1)}
\newcommand{\DD}{\mathbf{D}_n}
\newcommand{\dDD}{\mathbf{D}_n^\ast}
\newcommand{\nadj}[1]{\widehat{\mathsf{adj}}(#1)}
\newcommand{\gray}[1]{{\color{gray} #1}}

\ifspringer
\makeindex             
\else
\newtheorem{result}{\ }[section]
\theoremstyle{changebreak}                

\newtheorem{proposition}[result]{Proposition}

\newtheorem{example}[result]{Example}

\newenvironment{proof}
 {{\sl Proof.}\hspace*{1 ex}}%
 {{\nopagebreak\hspace*{\fill}$\Box$\par\vspace{12pt}}}

\fi

\begin{document}

\ifspringer
\title*{Distance geometry with and without the graph\thanks{This survey was partly supported by the ANR-24-CE23-1621-03 project ``EVARISTE''.}}

\titlerunning{Distance geometry with and without the graph}
\author{Leo Liberti \and Carlile Lavor}
\authorrunning{Liberti, Lavor}
\institute{Leo Liberti\at LIX CNRS Ecole Polytechnique, Institut Polytechnique de Paris, 91128 Palaiseau, France \\ \email{liberti@lix.polytechnique.fr} \and Carlile Lavor\at IMECC, University of Campinas, 13081-970, Campinas, Brazil \\ \email{clavor@unicamp.br}}

%
\maketitle
\else

\thispagestyle{empty}
\begin{center} 
{\LARGE Distance geometry with and without the graph}
\par \bigskip
{\sc Leo Liberti${}^1$, Carlile Lavor${}^2$}
\begin{minipage}{15cm}
\begin{flushleft}
\par \bigskip
 {\small
   \begin{enumerate}
   \item {\it LIX CNRS, \'Ecole Polytechnique, Institut Polytechnique de Paris, F-91128 Palaiseau, France} \\ Email:\url{liberti@lix.polytechnique.fr}
   \item {\it IMECC, University of Campinas, 13081-970, Campinas, Brazil} \\ Email:\url{clavor@unicamp.br}
   \end{enumerate}
 }
\end{flushleft}
\end{minipage}
\par \medskip \today
\end{center}
\par \bigskip

\fi


\ifspringer
  \abstract{%
\else
  \begin{abstract}
\fi
    We survey theoretical, algorithmic, and computational results at the intersection of distance geometry problems and mathematical programming, both with and without adjacencies as part of the input. While mathematical programming can solve large-scale distance geometry problems with adjacencies, they are severely challenged in the absence thereof.
\ifspringer
  }
\else
  \end{abstract} 
\fi

\vspace*{1cm}

\ifspringer\else
\renewcommand{\baselinestretch}{0.75}\normalsize
\tableofcontents
\renewcommand{\baselinestretch}{1.0}\normalsize

\section*{Preface}
\fi

\noindent This is one of many surveys we have co-authored about Distance Geometry (DG) \cite{mdgpsurvey,dmdgpejor,dgp-sirev,dg-4or,six,idgpsurvey,dgds} over the years: the reader might then wonder what we could possibly survey that we have not already surveyed seven times in the last fifteen years.

In \cite{mdgpsurvey} we reviewed many continuous solution methods for the DG problem (which we define below) and a few discrete ones, among which ABBIE \cite{Hen95} and Branch-and-Prune \cite{lln5}, the most prominent application of DG being that of proteins. In \cite{dmdgpejor} we surveyed several variants of the DG problem related to the Branch-and-Prune algorithm, the only DG application being (again) proteins. In \cite{dgp-sirev} we surveyed what we thought was most of the field of DG: many problem variants, applications, theoretical results, and methods, among which the Branch-and-Prune and proteins. When \cite{vetterli} was published, however, we discovered we had neglected a large part of DG field. The survey \cite{dg-4or}, which was co-authored by only one of us, discusses methods for assigned and unassigned DG problems, focusing on rigidity, coordinates computation, nanostructures, and proteins. The survey \cite{six} is completely different, as it focuses on six deep theorems about DG in the history of mathematics, from Heron of Alexandria to Kurt G\"odel. The paper \cite{idgpsurvey} reviews a very specific aspect of the Branch-and-Prune algorithm, namely the discretization of the search space, but it also proposes new ideas, and so it is only partly a survey. Finally, \cite{dgds}, which was written by only one of us, surveys the interfaces between DG and other fields of applied mathematics and theoretical computer science: mathematical programming, linear algebra, computational complexity, natural language processing, artificial intelligence, machine learning, data science, and statistics.

Like \cite{dg-4or}, this survey focuses on the interplay between assigned and unassigned DG problems, and considers proteins as applications. Unlike \cite{dg-4or}, however, this survey explores mathematical programming formulations for both problems, discusses the relationship between ``unassigned'' and ``assigned'' versions of the problem, and presents a comparative computational benchmark for most of the proposed formulations. We hope that this unconventional computational angle will make mathematical programming formulations come to life, instead of remaining purely theoretical constructs.

The reader might also wonder why we are so keen on writing surveys. We, as well as many of our co-authors, have had the good fortune to be invited to give talks in our careers. After delivering the talk, we were often asked (by whomever had invited us) to write a survey. This happened with \cite{mdgpsurvey,dmdgpejor,six}. We were also approached by some Editors-in-Chief of journals that routinely commission surveys: this happened with \cite{dgp-sirev,dg-4or,dgds}. The current survey was invited by Meera Sitharam and Tony Nixon, who co-organized several thematic programs at the Fields Institute. And the odd one out \cite{idgpsurvey}, as we already explained, is not quite just a survey. We could therefore answer ``it's not \textit{us}, it's \textit{them}!''. But the truth is that writing surveys teaches us a lot: by re-organizing the material in a way that makes sense for us (and, we hope, for the reader too) we see the whole mountain, where previously we had only glimpsed at some peaks. So it is not quite them, it is really just us. 

\section{Introduction}
\label{s:intro}
DG problems come in two broad forms: assigned and unassigned. Their input is always a set of distances, but the word ``distance'' is ambiguous. On the one hand, ``one meter'', or ``two-point-six {\AA}ngstr\"om'' are distances. On the other hand, ``there are 542 km between Toronto and Montreal'' is also a distance. These two concepts are related, but their formalization is very different. For the first, a positive scalar suffices. For the second, one also needs two distinct points: more precisely, we assign the scalar ($542$) to an edge $\{\mathsf{Toronto}, \mathsf{Montreal}\}$. Since a set of edges defines a graph, DG problems can be defined on a graph (distance scalars are assigned to edges) or without a graph (distance scalars are ``unassigned'').

The output of DG problems, however, is always the same: we ask for a set of points in a Euclidean space such that each given distance appears exactly once as the Euclidean distance between two of the points. The points correspond to vertices of the graph in the assigned form; in the unassigned form, the number of points $n$ is given as part of the input. This set of points is called \textit{realization}.

\label{unassignedform} Take the unassigned form. Once a realization is given that is compatible with the distance scalars, it is possible to construct a graph, since each scalar is now the length of a segment between two points. This tells us that the assigned form is easier than the unassigned one, since the graph is already given as part of the input. In the unassigned form we have to compute an assignment of distance scalars to edges as well as constructing a realization compatible with this assignment. This plays out in problem formulations as well: from each formulation of the assigned form, we can derive a more complicated formulation of the unassigned form by adding some assignment variables. 

\subsection{Problem definitions}
\label{s:probdefn}
We now give the formal decision version of the two problem forms, with and without the graph. For an integer $p$ we let $[p]=\{1,\ldots,p\}$. We recall that an assignment $\alpha:S\to T$ is an injective but not necessarily surjective function. 
\begin{itemize}
\item \textsc{Distance Geometry Problem} (DGP). Given an integer $K>0$ and a simple undirected edge-weighted graph $G=(V,E,d)$, determine whether there exists a realization $x:V\to\mathbb{R}^K$ such that
  \begin{equation}
    \forall \{u,v\}\in E \quad \|x_u-x_v\|_2^2 = d_{uv}^2, \label{dgp}
  \end{equation}
  where $d_{uv}$ is the weight of the edge $\{u,v\}$ and $x_v\in\mathbb{R}^K$ is the position vector of vertex $v$.
\item \textsc{Unassigned Distance Geometry Problem} (UDGP). Given two positive integers $K,n$ and a list $L=(\delta_\ell\;|\;\ell\le m)$ of positive scalars, determine whether there exists an assignment $\alpha:[m]\to [n]\times [n]$ and a realization $x:[n]\to\mathbb{R}^K$ such that:
  \begin{equation}
    \forall \ell\le m \quad \|x_{\alpha_{\ell_1}}-x_{\alpha_{\ell_2}}\|_2^2 = \delta_{\ell}^2, \label{udgp}
  \end{equation}
  where $\alpha(\ell)=(\alpha_{\ell_1},\alpha_{\ell_2})$ for all $\ell\le m$.
\end{itemize}
We remark that the UDGP is invariant to the order of $L$. We write Eq.~\eqref{dgp}-\eqref{udgp} in squared form because they avoid the square root, which is computationally problematic. We also note that, if an assignment $\alpha$ is given, any UDGP instance becomes a DGP one, since the assignment allows the construction of a simple undirected graph edge-weighted by $L$.

As stated, the DGP and UDGP are decision problems, susceptible of a single bit (YES/NO) output, with a certificate for YES instances. This certificate, for YES instances, is a realization $x$ for the DGP, and a couple (realization $x$, assignment $\alpha$) for the UDGP. DGP instances may be NO for two reasons: the given input is incompatible with any possible realization in $K$ dimensions, or the given input is incompatible with realizations in \textit{any} dimensions (or with the metric axioms). UDGP instances may be NO also because every assignment leads to an infeasible DGP instance.

We do not consider other distances than Euclidean in this survey, because our chosen applications (molecules) are based on Euclidean distances (but see \cite{oneinfnorm-lncs}): accordingly, we dispense with the problem name ``Euclidean DGP'' (EDGP), since there is no confusion. Nor do we consider imprecise versions of DGP and UDGP explicitly, where $d$ and $\delta$ are real intervals instead of scalars \cite{mdgpsurvey,idgpsurvey,zoo}, mainly because the solution methods we discuss --- all based on solving formulations with an off-the-shelf solver --- already cater for this type of imprecision, also known as ``experimental error'', by turning the strict feasibility enforced by Eq.~\eqref{dgp} into optimization by means of a penalty function. 

In the following we shall encode ``realizations'' of the DGP or the UDGP as $n\times K$ real matrices: thus, a realization can be precise, imprecise, approximate, random, or even wrong. In short, realizations will simply be whatever solution of the DGP/UDGP we were able to find. 

\subsection{Complexity}
\label{s:complexity}
The DGP (and, respectively, the UDGP) is the inverse problem of the following trivial problem: given a realization $x\in\mathbb{R}^{nK}$ compute a subset of distances adjacent to pairs of points (and, respectively, compute a subset of distance values). While the direct problems are trivial, the inverse problems are \textbf{NP}-hard. The DGP is \textbf{NP}-hard (weakly, by reduction from \textsc{partition} to the case $K=1$, and strongly, by reduction from \textsc{3sat} for any fixed $K$ \cite{saxe79}). The UDGP is \textbf{NP}-hard even when $m=n(n-1)/2$ \cite[Thm.~4.2]{skiena}, i.e.~when all possibile distances are given, yielding a complete graph with any assignment $\alpha$: by contrast, the DGP is tractable on complete graphs \cite{sippl,alencar2}. The fact that \textbf{NP}-hardness of the DGP is proved on sparse graphs, while that of the UDGP is proved on instances corresponding to complete graphs, may be indicative of the respective application settings: DGPs are often solved on (sparse) disk graphs, while UDGPs on complete (or almost complete) lists of distance scalars\footnote{Even more-than-complete lists have been considered in solving UDGPs: in \cite{liga}, some experiments have been conducted on a redundant list of distance scalars, i.e.~having length exceeding $m$.}.

\subsection{Applications}
\label{s:apps}
The typical DGP applications are: synchronization of clocks in sensor network protocols \cite{singer4} ($K=1$), localization of sensor networks \cite{biswasacm} ($K=2$), reconstruction of the shape of proteins from distance data \cite{CH88,bipbip} ($K=3)$, control of underwater autonomous vehicles \cite{bahr} ($K=3$), transformation of words and sentences into vectors in natural language processing \cite{dg4wv,dgpwords} (any $K$).

The typical UDGP applications are: the ``partial digest'' methodology in DNA sequencing\footnote{Here is a summary description of the partial digest problem. Suppose one wants to determine thes sites of a DNA strand at which a certain sequence $\sigma$ over $\mbox{A},\mbox{C},\mbox{G},\mbox{T}$ occurs. One employs a specific enzyme that breaks the strand at sites where $\sigma$ appears, and measures the lengths of the pieces. One then solves a UDGP problem in one dimension in order to retrieve the sites at which the DNA strand was broken: those are the sites containing $\sigma$.} \cite{skiena2} ($K=1$), the reconstruction of crystal structure from interatomic distances obtained from X-ray crystallography \cite{patterson} ($K\in\{2,3\}$), the reconstruction of the shape of small inorganic molecules, such as e.g.~nanostructures, from distance data \cite{liga,dg-4or} ($K=3$), and the reconstruction of the shape of proteins from distance data \cite{berger} ($K=3$).

In these lists of DGP and UDGP applications, we note immediately the preponderance of molecular applications, and the fact that the same application (shape of proteins) occurs in both. Molecular applications are frequent in DG because ``looking'' at a molecule, given its nano-scale, means irradiating it and actually look at the diffraction of the radiation as it passes through the molecule: either as a crystal, as in X-ray crystallography, or in solution, as in nuclear Overhauser effect spectroscopy (NOESY), or frozen in a thin layer of ice, as in Cryo-EM. The result of the analysis, the radiation spectrum, is represented as a multivariate function $f$ in $1,2,3$ or even more dimensions. Some of the peaks of $f$ indicate that two or more atoms of a certain type are at some given distance value (not all peaks are meaningful, however \cite{patterson}).

To clarify, let us take a spectrum function $f(a,b)$ of two dimensions $a,b$. Supposing that $a,b$ denoted atom IDs (e.g.~$\mathrm{H}^9, \mathrm{C}^{17}$), one could simply read a peak in $f(a,b)$ as a distance between two well-defined atoms, with the distance value proportional to the integral of $f$ over a neighbourhood including the peak. Instead, peaks depend on other entities that are correlated with distance values in more complicated ways.
\begin{itemize}
\item In the case of X-ray experiments on nanostructures, one can only obtain a probability of finding a certain distance between two atoms (this probability function is called \textit{pair distribution function}, or sometimes simply PDF, in the relevant literature). Nonetheless, this yields unassigned but relatively precise and complete distance value measurements in nanostructures, which is why nanostructures are a typical UDGP application.
\item 
In the case of NOESY experiments on proteins in solution, distance values are implied by resonance effects at some given frequencies specific to each atomic nucleus. These frequencies are known as the ``chemical shifts'' of an atom: a peak in chemical shift space may denote the closeness of two atoms, and their distance value can be estimated. The assignment of the values to atom pairs is also estimated, see Sect.~\ref{s:proteins} below. These estimations yield considerable errors \cite{berger}, which is why the determination of the shape of proteins is an application of both the DGP and the UDGP \cite{wuthrich,wuthrich_nobel}. 
\end{itemize}

\subsubsection{The case of proteins}
\label{s:proteins}
Chemical shifts depend on the atom type and its environment. Unfortunately, many atoms in a protein have very close chemical shifts, so that, for practical purposes, these atoms share the same chemical shift. This means that a single peak $p=(a,b)$ of intensity $f(a,b)$ is assigned ambiguously to all of the atoms in the set $A^p_a$ with the same chemical shift $a$, and all of the atoms in the set $B^p_b$ with the same chemical shift $b$ at the peak $p$. The intensity also depends on the number of bonds in $A^p_a\times B^p_b$ that actually occur in the protein with distance $d$ proportional to $p$, which further complicates matters.

Luckily, though, proteins have a periodic backbone with a simple known structure. By a mixture of experimental and algorithmic procedures, this structure makes it possible to derive an assignment of distance values to edges: this justifies the appearance of proteins in the DGP applications list. These procedures are imperfect, however, and a systematic assignment error still occurs \cite{berger}: this requires a re-assignment phase, which justifies the appearance of proteins in the UDGP applications list.

\ifspringer
\subsection{Contents of this survey}
The rest of this survey is organized as follows. Sect.~\ref{s:theory} introduces some basic theoretical results. In Sect.~\ref{s:mp} we give a short introduction to mathematical programming. In Sect.~\ref{s:form} we look at DGP formulations: for each we consider relaxations, approximations, and the derived UDGP formulations. In Sect.~\ref{s:matform} we look at matrix formulations and their resulting approximations. In Sect.~\ref{s:results} we perform extensive computational evaluations of all of the DGP and UDGP formulations. Sect.~\ref{s:concl} concludes the survey.
\fi

\section{Basic results}
\label{s:theory}
First, we make the statement about the the relative difficulty betwen UDGP and DGP (paragraph ``Take the unassigned form [\dots]'' in Sect.~\ref{s:intro} on p.~\pageref{unassignedform}) more precise. For a function $f$, $\dom f$ is the domain and $\ran f$ the range of $f$. Note that the sequence $\delta$ can be seen as a function $[m]\to\mathbb{R}_+$.
\begin{proposition}
  \label{prop:udgp2dgp}
  Let $\mathscr{U}=(K,n,L)$ be an instance of the UDGP, and $\alpha:[m]\to[n]\times [n]$ be an assignment. Consider the graph $G_\alpha=(V,E,d)$ where $V=[n]$, $E=\ran\alpha$, $d=\delta\circ\alpha^{-1}$ and the corresponding DGP instance $\mathscr{D}_\alpha=(K,G_\alpha)$. If $\mathscr{U}$ is YES, then $\mathscr{D}_\alpha$ is YES. If $\mathscr{U}$ is NO, $\mathscr{D}_\alpha$ is NO for all possible $\alpha$. 
\end{proposition}
\begin{proof}
  Assume $\mathscr{U}$ is YES. The graph $G_\alpha$ has $[n]$ as vertex set $V$, and edges $\{u,v\}$ for $u<v\le n$ whenever there is $\ell\le m$ with $\alpha(\ell)=(u,v)$. Moreover, each edge $\{u,v\}$ is weighted by $\delta_\ell$ such that $\alpha(\ell)=\{u,v\}$. By Eq.~\eqref{udgp} and the definition of $G_\alpha$, $x$ also satisfies Eq.~\eqref{dgp}, i.e.~$\mathscr{D}_\alpha$ is YES.  Assume now that $\mathscr{U}$ is NO, and suppose, to aim at a contradiction, that there is $\alpha$ such that $\mathscr{D}_\alpha$ is YES: then it has a realization $x$ for $G_\alpha$, which, by Eq.~\eqref{dgp} and by definition of $G_\alpha$, is also a realization of $\mathscr{U}$, against the assumption. 
\end{proof}
We note that Prop.~\ref{prop:udgp2dgp} introduces a mapping from UDGP to DGP instances, namely $(\mathscr{U},\alpha)\to\mathscr{D}_\alpha$. We also emphasize the notation $G_\alpha$ in order to refer to the the graph corresponding to the assignment $\alpha$ from a UDGP instance, as in Prop.~\ref{prop:udgp2dgp}. 

\begin{example}
  Consider the UDGP with input $(K=2,n=3,L=(3,1,1))$. Since $K=2$ and $n=3$ we want a triangle in the plane. But since the distance scalars $3,1,1$ violate the triangular inequality, the instance is infeasible. Observe that $L$ is a \textit{list} rather than a \textit{set} of scalars, since we may need to specify scalars with a multiplicity greater than one (this is often the case in nanostructures). Morover, by Prop.~\ref{prop:udgp2dgp}, no reconstructed graph $G_\alpha$ yields a feasible DGP instance. 
  \hfill $\blacksquare$ 
\end{example}

\begin{example}
  Consider the UDGP with input $(K=2,n=3,L=(\delta_1,\delta_2,\delta_3))$, which results in a triangle in the plane, with side lengths $\delta_1,\delta_2,\delta_3$ and an underlying graph
  \[T_2=(\{1,2,3\},\{\{1,2\},\{2,3\},\{1,3\}\}),\]
which is the complete graph on $3$ vertices. The realization does not change if the point labels change, so all of the possible assignments are valid:
  \begin{itemize}
  \setlength{\parskip}{-0.12em}
  \item $\alpha^1: \delta_1\to \{1,2\}, \delta_2\to\{2,3\}, \delta_3\to\{1,3\}$
  \item $\alpha^2: \delta_1\to \{1,2\}, \delta_2\to\{1,3\}, \delta_3\to\{2,3\}$ 
  \item $\alpha^3: \delta_1\to \{1,3\}, \delta_2\to\{2,3\}, \delta_3\to\{1,2\}$
  \item $\alpha^4: \delta_1\to \{1,3\}, \delta_2\to\{1,2\}, \delta_3\to\{2,3\}$
  \item $\alpha^5: \delta_1\to \{2,3\}, \delta_2\to\{1,2\}, \delta_3\to\{1,3\}$
  \item $\alpha^6: \delta_1\to \{2,3\}, \delta_2\to\{1,3\}, \delta_3\to\{1,2\}$,
  \end{itemize}
  even though all are realized by the same triangle, but with permuted vertex labels. In general, if an assignment $\alpha$ is computed for a UDGP instance, leading to a graph $G_{\alpha}$, all isomorphic versions of $G_{\alpha}$ are admissible reconstructions. If all possible distances are given for the number $n$ of points (as e.g.~in this example $m=|L|=3=n(n-1)/2$), all $n!$ vertex permutations lead to feasible assignments, i.e.~to isomorphic copies of the $3$-clique graph. \hfill $\blacksquare$ 
\end{example}

An inverse mapping from DGP to UDGP instances is constructed as follows: let $\mathscr{D}=(K,G=(V,E,d))$ be a DGP instance. Let $<$ be any total order on $E$ inducing the edge list $(e_1,\ldots,e_m)$. Then we can define $\delta_\ell=d_{e_\ell}$ for all $\ell\le m$. This yields a UDGP instance $\mathscr{U}_<=(K,|V|,\delta)$ which is YES if $\mathscr{D}$ is YES. By contrast, $\mathscr{U}_<$ may be YES even though $\mathscr{D}$ is NO, as shown in Example \ref{eg:DnoUyes}.

\begin{example}
\label{eg:DnoUyes}
Consider the graph $G$ given by the following weight function $d_{12}=3, d_{23}=4, d_{13}=5, d_{14}=2, d_{24}=2$, shown in Fig.~\ref{fig:eg3} (top left) with the correct realization:
\begin{figure}[!ht]
  \begin{center}
    \begin{minipage}{0.4\textwidth}
      \tikz [x=3cm, y=3cm, circle, scale=0.6]
      \graph[no placement, nodes={draw,circle}] {
        1[x=0,y=0] --["5"] 3[x=-1.33,y=1] ;
        1 --["3"] 2[x=0,y=1] --["2"] 4[x=0.4,y=0.5] --["2"] 1 ;
        2 --["4"] 3 
      };
      \\
      \tikz [x=3cm, y=3cm, circle, scale=0.6]
      \graph[no placement, nodes={draw,circle}] {
        1[x=0,y=0] --["5"] 3[x=-1.33,y=1] ;
        1 --["3"] 2[x=0,y=1] ;
        2 --["4"] 3 ;
        1 --["2"] 4[x=-1.33,y=0] --["2"] 3 
      };
    \end{minipage}
    \hfill
    \begin{minipage}{0.59\textwidth}
      \includegraphics[width=1.0\textwidth]{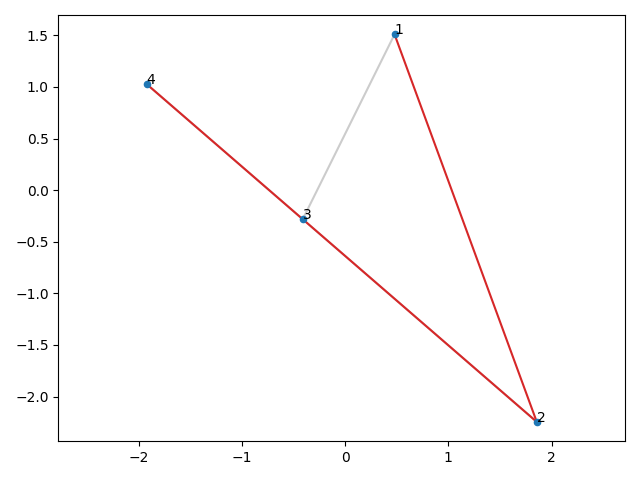}
    \end{minipage}
  \end{center}
  \caption{The same UDGP instance gives rise to (at least) three DGP instances (two YES, one NO).}
  \label{fig:eg3}
\end{figure}
Evidently, the DGP with $K=2$ on $G$ is YES. Now consider instead the graph $H$ defined by $d_{12}=3, d_{23}=4, d_{13}=5, d_{14}=2, d_{34}=2$, where the distance value $2$ previously assigned to $\{2,4\}$ is now assigned to $\{1,4\}$, shown in Fig.~\ref{fig:eg3} (bottom left) with an incorrect realization. It is obvious that the DGP instance defined on $H$ is NO, since there is no triangular realization for the subgraph $H[1,3,4]$ defined on distance values $5,2,2$: they do not satisfy triangular inequalities. Nonetheless, both DGPs are obtained from the same UDGP instance $\mathscr{U}=(2,4,(2,2,3,4,5))$ by means of different assignments:
\begin{eqnarray*}
  \alpha_G &:&1\to\{1,4\},2\to\{2,4\},3\to\{1,2\},4\to\{2,3\},5\to\{1,3\} \\
  \alpha_H &:&1\to\{1,4\},2\to\{3,4\},3\to\{1,2\},4\to\{2,3\},5\to\{1,3\}.
\end{eqnarray*}
Another possible solution of $\mathscr{U}$, found by a solver deployed on an MP formulation, is the assignment $\alpha':1\to\{1,3\},2\to\{3,4\},3\to\{2,3\},4\to\{1,2\},5\to\{2,4\}$ leading to the realization in Fig.~\ref{fig:eg3} (right), where the triangle on $2,3,4$ is flat. 
  \hfill $\blacksquare$
\end{example}

\section{A mathematical programming primer}
\label{s:mp}
All of the methods we shall discuss in this survey are formulation-based. We are going to formulate the DGP and UDGP in many different ways by means of Mathematical Programming (MP), a declarative formal language for describing and solving optimization problems \cite{williams}. The general form of an MP formulation is:
\begin{equation}
  \min_x \{ f(x) \;|\; \forall i\le p\;g_i(x)\le 0\land x\in X\},
\end{equation}
where:
\begin{itemize}
\item $x$ is an array of \textit{decision variables};
\item the functions $f(x)$ and $g_i(x)$ for each $i\le p$ are represented by a mathematical expression based on a formal grammar with the usual arithmetic operators, elementary functions, and brackets;
\item $f(x)$ is the \textit{objective function} to be minimized;
\item $\forall i\le p\;g_i(x)\le 0$ are \textit{explicit constraints};
\item $X$ is a set of \textit{implicit constraints} that may be hard or inconvenient to represent, but for which there exist convenient computational methods.
\end{itemize}
We warn that sometimes, to save space, MP formulations present some explicit constraints implicitly, when the explicit form has already been discussed previously.

Once a problem is represented by a MP formulation, one may look at the type of implicit constraints (e.g.~non-negative orthant, integer lattice), variables (e.g.~continuous, integer, binary, mixed, matrix) and terms (e.g.~linear, quadratic, polynomial, general nonlinear) involved in objective and constraints, and choose an off-the-shelf piece of software called \textit{solver} that caters for the formulation properties. When the solver is deployed on the formulation, given sufficient time and assuming the formulation conforms to the theoretical assumptions of the solver, it will provide one or more solutions to the problem, or report an error. Implicit constraints are handled by specific parts of the algorithm implemented by the chosen solver.

The input of an MP formulation are the mathematical expressions $f,g_i$ (for $i\le p$) and possibly the choice of solver given by $X$. The output is given by the values of the decision variables after the solution process, or the type of error returned. 

In general, MP formulations may be feasible or infeasible, bounded or unbounded. An appropriate solver may be able to prove feasibility/infeasibility, boundednes/unboundedness, and also provide a solution as an output (values of the decision variables after the solver terminates). 

Solvers are implemented algorithms for solving a certain subclass of MP formulations. The different existing solvers yield a cover of MP by subclasses, which provides a taxonomy for MP. For the purposes of this survey, the taxonomy we make use of is: Linear Programming (LP), Semidefinite Programming (SDP), Mixed-Integer LP (MILP), Quadratically Constrained Quadratic Programming (QCQP), convex QCQP (cQCQP), Nonlinear Programming (NLP), convex NLP (cNLP), Mixed-Integer QCQP (MIQCQP), Mixed-Integer NLP (MINLP).

Solvers may be local or global: a local solver requires a starting point (initial variable values) and reaches a close-by local optimum. A global solver gives some kind of guarantee of global optimality. Usually, local solvers for purely continuous problems are significantly faster than their global counterparts. If integer variables are involved, however, most problems become \textbf{NP}-hard, and local solvers may be as slow as global ones (in fact, with integer variables, a ``local solver'' often consists in running a global solver with a finite amount of time). 

In the taxonomy above, LP is the only MP subclass in \textbf{P}, and the only subclass for which feasibility and boundedness (or their converse) can be proved by the solver with a certificate that can be verified by a Turing Machine (TM). In particular, LP SDP and cQCQP are all convex MPs, and a local solver identifies a local optimum which, by convexity, is also global. The cQCQP and SDP subclasses, therefore have similar characteristics to the LP subclass, but within an $\epsilon>0$ precision limit. MILP, QCQP, MIQCQP, NLP, MINLP are all \textbf{NP}-hard; moreover, MIQCQP and MINLP are undecidable \cite{undecminlp}.

The situation with cNLP is more complicated: many cNLPs are ``tame'', in the sense that a local NLP solver will usually identify a local optimum which, by convexity, is also global (the ``usually'' caters for the possibility of the cNLP not conforming to constraint qualification conditions which are rarely verifiable by a TM, and hence one can just hope nothing bad will happen). The cNLP class, however, also includes copositive reformulations of the Motzkin-Straus formulation \cite{motzkinstraus} of \textsc{max clique}, a famously \textbf{NP}-hard problem, showing that cNLP is also \textbf{NP}-hard by inclusion \cite{burer2}. In practice, no local NLP solver can deal with the implicit copositivity constraint. 

An important feature of MP formulations is that they can be \textit{reformulated}, i.e.~symbolically changed, so that some aspect of the formulation remains invariant. The most common invariants are (a) the set of global optima, (b) at least one global optimum, (c) the globally optimal value, (d) a guarantee that the reformulation will yield a bound in the optimization direction of the original formulation, (e) a general approximation guarantee \cite{refmathprog}. The reformulation used most frequently is \textit{linearization} \cite{arschapter}, which consists in replacing a mathematical expression $f(x)$ with a new variable $\phi$, then adjoin the \textit{defining constraint} $\phi=f(x)$ to the formulation. The resulting reformulation is \textit{exact}, i.e.~all global optima of the original formulation are mapped to global optima of the second. Linearizations are usually employed as a starting point for further reformulations, usually of the \textit{relaxation} type, i.e.~a reformulation yielding a guaranteed bound in the optimization direction. 

\section{MP formulations for the DGP and the UDGP}
\label{s:form}
All DGP formulations are either equivalent to, or derived from, Eq.~\eqref{dgp}, and similarly for UDGP formulations. By turning the decision problem Eq.~\eqref{dgp} into an optimization problem, most of our formulations (with the exception of \textsf{push-and-pull} and variants, see Sect.~\ref{s:pushpull}) relax feasibility into optimality: specifically, the DGP instance is YES if and only if the globally optimal objective function value is zero. Moreover, if the instance is NO because of slightly imprecise distance data, any realization will yield an optimal objective function value that is ``reasonably'' small. This is why, in this survey, we chose to dispense from representing experimental measurement errors explicitly by intervals: one may more simply cater to these errors by accepting realizations with objective values that are slightly larger than zero, but within a given $\epsilon>0$. 

With most DGP formulations, we can derive a corresponding UDGP formulation by means of the mapping $\mathrm{DGP}\to\mathrm{UDGP}$ given by $\mathscr{D}\to\mathscr{U}_<$ mentioned in Sect.~\ref{s:theory} just before Example \ref{eg:DnoUyes}. This usually involves changing a few indices and sets in the formulation, and adding the component relative to the assignment $\alpha$.

All congruences (translations, rotations, reflections) applied to a valid realization produce another valid realization: thus, all MP formulations of YES instances have in fact an uncountable number of global optima. We found that it sometimes help solvers to fix at least the translations (fixing rotations and reflections at the formulation level is harder). This can be achieved with a set of \textit{centroid constraints}:
\begin{equation}
  \forall k\le K \quad \sum\limits_{u\in V} x_u = 0.
  \label{centroid}
\end{equation}
For UDGP formulations, the quantification of the sum is on $i\le n$. 

\subsection{The \textsf{quartic} formulation}
Minimizing the sum of squared differences of the two sides of a system of equations is the most common way to solve such a system. This probably comes from the fact that, for a linear system, there exists a closed formula (linear regression); moreover, for a linear system, the corresponding optimization problem is a ``tame'' cNLP. The application of this technique to nonlinear systems, by contrast, gives rise to nonconvex NLPs in general, which is an \textbf{NP}-hard class. When applied to the DGP system \eqref{dgp}, this is to be expected in view of the fact that the DGP itself is \textbf{NP}-hard, but it weakens the justification for choosing the minimization of the sum of squared differences to reformulate Eq.~\eqref{dgp}. This formulation can be traced back to \cite{takane_77} (and perhaps even earlier). It was tested computationally as a DGP formulation starting from \cite{lln1} and in many more papers after it.
\begin{equation}
  \min\limits_{x\in\mathbb{R}^K} \sum\limits_{\{u,v\}\in E} \big(\|x_u-x_v\|_2^2-d_{uv}^2\big)^2.
  \label{quartic}
\end{equation}
The \textsf{quartic} formulation \eqref{quartic} is unconstrained. It minimizes a nonconvex multivariate polynomial of degree 4  (hence the name \textit{quartic}). It can be solved globally with a global NLP solver, or locally with a local NLP solver starting from a given imprecise realization $x'$ of the weighted graph defined by $\{d_{uv}\;|\;\{u,v\}\in E\}$. 

\subsubsection{The unassigned \textsf{quartic}}
\label{s:unass}
We represent the unknown assignment $\alpha:[m]\to[n]\times [n]$ by means of binary variables $y_{ij\ell}$ such that $y_{ij\ell}=1$ iff $\alpha(\ell)=(i,j)$. The assignment properties are that (i) it is a function:
\begin{equation}
  \forall \ell\le m\quad \sum\limits_{i<j\le n} y_{ij\ell}\le 1,
  \label{ass1}
\end{equation}
and (ii) it is injective:
\begin{equation}
  \forall i<j\le n\quad \sum\limits_{\ell\le m} y_{ij\ell}= 1.
  \label{ass2}
\end{equation}
Since these assignment constraints will be repeated often in the following, we summarize them in the set $\mathcal{A}=\{y\in\{0,1\}^{n^2m} \;|\; \mbox{Eq.~\eqref{ass1}-\eqref{ass2}}\}$, and then use the implicit constraint $y\in\mathcal{A}$ to refer to the (explicit) constraints Eq.~\eqref{ass1}-\eqref{ass2} together with the implicit constraint $y\in\{0,1\}^{n^2m}$.

We now modify Eq.~\eqref{quartic} so that each term in the objective is added to the sum only if the corresponding $y$ variable is set to $1$:
\begin{equation}
  \min\limits_{x\in\mathbb{R}^{nK}\atop y\in\mathcal{A}} \sum\limits_{\ell\le m\atop i<j\le n} y_{ij\ell} \big(\|x_i-x_j\|_2^2-\delta_{\ell}^2\big)^2.
  \label{uquartic}
\end{equation}
The \textsf{quartic} UDGP formulation Eq.~\eqref{uquartic} is a constrained MINLP involving polynomial functions of degree 5, which first appeared in \cite{udgp}. It can be solved using a global MINLP solver. Because of the presence of the binary variables $y$, local solutions may be achieved by giving a resource constraint (CPU time, iterations, number of nodes) to the global solver. Using global solvers with resource constraints is referred to as ``using a global solver locally''.

There is an interesting, and unexpected continuous exact reformulation of Eq.~\eqref{uquartic}, yielding a nonconvex NLP which can be solved locally or globally using a local or global NLP solver. This is based on the observation that, under certain conditions, maximizing a sum of squared variables defined over $[0,1]$ yields a binary vector.
\begin{equation}
  \left.\begin{array}{rrcl}
    \min\limits_{x\in\mathbb{R}^{nK}\atop y\in\mathcal{A},t\in\mathbb{R}} & t-\sum\limits_{\ell\le m\atop i<j\le m} y^2_{ij\ell} && \\
    & \sum\limits_{\ell\le m\atop i<j\le n} y_{ij\ell} \big(\|x_i-x_j\|_2^2-\delta_{\ell}^2\big)^2 &=& t.
  \end{array}\right\}
  \label{uquartic_cont}
\end{equation}
The proof that this reformulation is indeed exact is given in \cite[Thm.~2]{udgp}. This formulation is enticing insofar as it would allow one to solve it locally using local NLP solvers, which take a starting point and improve it (which is usually much faster than using a global solver locally on a mixed-integer formulation). However, the theorem that guarantees that the continuous $y$ variables have binary values only applies at global optima. We also note that, at global optima, the objective function value is not zero but $-m$. Thus, it might be harder to claim that NO instances with only slighly imprecise distances are recognizable by a slight variation on the objective function: in Eq.~\eqref{uquartic_cont} the ``error'' given by the objective value represents distance errors as well as assignment errors (since a non-binary $y$ value cannot be interpreted as an assignment).

\subsection{The \textsf{system} formulation}
\label{s:system}
The \textsf{system} formulation is almost equivalent to the \textsf{quartic}: it simply linearizes the differences under the square:
\begin{equation}
  \left.\begin{array}{rrcl}
    \min\limits_{x\in\mathbb{R}^{nK}\atop s\in\mathbb{R}^m} & \sum\limits_{\{u,v\}\in E} s^2_{uv} && \\
    \forall \{u,v\}\in E & \|x_u-x_v\|_2^2 &=& d_{uv}^2 + s_{uv}.
  \end{array}\right\} 
  \label{system}
\end{equation}
The \textsf{system} formulation is a nonconvex QCQP which can be solved locally or globally by local or global QCQP or NLP solvers.

A variant of the \textsf{system} formulation is based on the $\ell_1$-norm to penalize the error, instead of the $\ell_2$-norm:
\begin{equation}
  \left.\begin{array}{rrcl}
    \min\limits_{x\in\mathbb{R}^{nK}\atop s\in\mathbb{R}^{2m}} & \sum\limits_{\{u,v\}\in E} (s^+_{uv} &+& s^-_{uv}) \\
    \forall \{u,v\}\in E & \|x_u-x_v\|_2^2 &=& d_{uv}^2 + s^+_{uv} - s^-_{uv}.
  \end{array}\right\} 
  \label{system_ell1}
\end{equation}
This is actually a Quadratically Constrained Program (QCP) since the objective is linear. The QCP class is a subsclass of QCQP (since all linear forms are also trivial quadratic forms). As regards DGP and UDGP formulations, there is no substantial practical difference between the two classes, as there are no specific QCP solvers that cannot solve QCQPs too. 

\subsubsection{The unassigned \textsf{system} formulation}
\label{s:ass_syst}
We refer to Sect.~\ref{s:unass} for the definition of the assignment constraint set $\mathcal{A}$ and the binary variables $y$. The unassigned version of Eq.~\eqref{system} is
\begin{equation}
  \left.\begin{array}{r@{\hspace*{0.4em}}rcl}
    \min\limits_{x\in\mathbb{R}^{nK}\atop s\in\mathbb{R}^m,y\in\mathcal{A}} &  \sum\limits_{\ell\le m\atop i<j\le n} s_{\ell}^2 & & \\
    \forall \ell\le m, i<j\le n & -s_{uv}\!-\!M(1\!-\!y_{ij\ell}) \le &\|x_i\!-\!x_j\|_2^2\!-\!\delta_\ell^2 &\le s_{uv}\!+\!M(1\!-\!y_{ij\ell}),
  \end{array}\right\}
  \label{usystem}
\end{equation}
where $M$ is a ``big-M'' constant that must be an upper bound to the diameter of any realization satisfying the DGP. A very slack bound $M=(\sum_{\ell\le m}\delta_\ell)^2$ is provided in \cite[Prop.~2.2]{cordone}. 

The unassigned version of the $\ell_1$-norm variant in Eq.~\eqref{system_ell1} is:
\begin{equation}
  \left.\begin{array}{r@{\hspace*{0.4em}}rcl}
    \min\limits_{x\in\mathbb{R}^{nK}\atop s\in\mathbb{R}^{2m},y\in\mathcal{A}} &  \sum\limits_{\ell\le m\atop i<j\le n} (s_{\ell}^+ + s_\ell^-) && \\
    \forall \ell\le m, i<j\le n & -s^-_{uv}\!-\!M(1\!-\!y_{ij\ell}) \le &\|x_i\!-\!x_j\|_2^2\!-\!\delta_\ell^2 &\le s_{uv}^+\!+\!M(1\!-\!y_{ij\ell}).
  \end{array}\right\}
  \label{usystem_ell1}
\end{equation}
In Eq.~\eqref{usystem}-\eqref{usystem_ell1} the $y$ variables are used to activate or deactivate the constraints according to whether the distance $\delta_\ell$ is assigned to edge $\{u,v\}$ or not. We note that the term on the objective need not be multiplied by $y$ when $y=0$ since this is taken care of by the optimization direction.

The advantage of unassigned \textsf{system} formulations w.r.t~other UDGP formulations is that they are MIQCQPs instead of MINLPs, which allows one to deploy a larger set of solvers upon them. In fact, Eq.~\eqref{usystem_ell1} is a Mixed-Integer QCP (MIQCP). 

\subsection{The \textsf{push-and-pull} formulation}
\label{s:pushpull}
The \textsf{push-and-pull} formulation of the DGP \cite{zoo,mwu} is as follows:
\begin{equation}
  \left.\begin{array}{rrcl}
    \max\limits_{x\in\mathbb{R}^{nK}} & \sum\limits_{\{u,v\}\in E} \|x_u-x_v\|_2^2 && \\
    \forall \{u,v\}\in E & \|x_u-x_v\|_2^2 &\le& d_{uv}^2.
  \end{array}\right\}
  \label{pushpull}
\end{equation}
It is a nonconvex QCQP with a concave objective (maximization of a convex function) and convex quadratic constraints. It can be solved locally or globally by local or global QCQP or NLP solvers. Its name is given by the suggestion that the constraints push the realization points together, while the objective pulls them apart.

There is also a \textsf{pull-and-push} formulation that inverts the objective direction and the constraint senses:
\begin{equation}
  \left.\begin{array}{rrcl}
    \min\limits_{x\in\mathbb{R}^{nK}} & \sum\limits_{\{u,v\}\in E} \|x_u-x_v\|_2^2 && \\
    \forall \{u,v\}\in E & \|x_u-x_v\|_2^2 &\ge& d_{uv}^2,
  \end{array}\right\}
  \label{pullpush}
\end{equation}
which is not quite as useful as Eq.~\eqref{pushpull} since it is usually (practically) easier, for many local NLP solvers, to decrease a difficult objective than to satisfy difficult constraints. Nonetheless we shall see later that even Eq.~\eqref{pullpush} has a use. 

It is not immediately obvious that Eq.~\eqref{pushpull} is an exact reformulation of Eq.~\eqref{dgp}. A proof of this fact is given in \cite[Prop.~2.8]{mwu} (the proof for Eq.~\eqref{pullpush} is analogous).

\subsubsection{The unassigned \textsf{push-and-pull}}
We refer to Sect.~\ref{s:unass} for the definition of the assignment constraint set $\mathcal{A}$ and the binary variables $y$. The unassigned version of Eq.~\eqref{pushpull} is
\begin{equation}
  \left.\begin{array}{rrcl}
    \max\limits_{x\in\mathbb{R}^{nK}\atop y\in\mathcal{A}} & \sum\limits_{\ell\le m\atop i<j\le n} y_{ij\ell}\|x_i-x_j\|_2^2 && \\
    \forall \ell\le m, i<j\le n & \|x_i-x_j\|_2^2 &\le& \delta_\ell^2 + M(1-y_{ij\ell}).
  \end{array}\right\}
  \label{upushpull}
\end{equation}
As in the unassigned \textsf{system} formulations, a value of $M$ is given in \cite[Prop.~2.2]{cordone}. In Eq.~\eqref{upushpull} the $y$ variables only count assigned indices in the objective, and deactivate constraints for non-assigned indices.

Eq.~\eqref{upushpull} is not quadratic, since the objective is a cubic polynomial. It is therefore a MINLP, which can be solved with a global MINLP solver, possibly used locally. 

\subsection{The \textsf{cycle} formulation}
The \textsf{cycle} formulation for the DGP is presented in \cite{dgp-ojmo}. In its native form, it decomposes the DGP into two phases: a constrained optimization problem and the solution of a linear system. The optimization problem constraint is quantified over a basis $\mathcal{B}$ \cite{fcbmmor} of the cycle space of the input graph $G$:
\begin{equation}
  \left.\begin{array}{rrcl}
  \min\limits_{z\in[-\mathbf{d},\mathbf{d}]^{mK}} & \sum\limits_{\{u,v\}\in E} (\|z_{uv}\|_2^2 &-& d_{uv}^2)^2 \\
  \forall C\in\mathcal{B}(G) & \sum\limits_{\{u,v\}\in C} z_{uv} &=& 0,
  \end{array}\right\}
  \label{cycle1}
\end{equation}
where $\mathbf{d}_{uv}=(d_{uv},\ldots,d_{uv})\in\mathbb{R}^K$ and $z_{uv}=(z_{uv1},\ldots,z_{uvK})$ for each $\{u,v\}\in E$. The linear system is:
\begin{equation}
  \forall \{u,v\}\in E, k\le K\quad x_{uk}-x_{vk}=z_{uvk}.
  \label{cycle2}
\end{equation}
Once Eq.~\eqref{cycle1} is solved, Eq.~\eqref{cycle2} is a linear system of $mK$ equations and $nK$ unknowns, which can be solved (if $z$ is an optimum of Eq.~\eqref{cycle1}) to retrieve the realization $x\in\mathbb{R}^{nK}$.

The proof in \cite{dgp-ojmo} leading to the correctness of this decomposition is long but elementary. However, the fact that the problem is decomposed in a hard part (Eq.~\eqref{cycle1}) and an easy part (Eq.~\eqref{cycle2}) does not improve computational performances by much. We therefore propose the following exact reformulation, which integrates Eq.~\eqref{cycle2} into Eq.~\eqref{cycle1}: this considerably shortens the correctness proof.
\begin{equation}
  \left.\begin{array}{rrclr}
  \min\limits_{x\in\mathbb{R}^{nK}\atop z\in[-\mathbf{d},\mathbf{d}]^{mK}} & \sum\limits_{\{u,v\}\in E} (\|z_{uv}\|_2^2&-&d_{uv}^2)^2  & \\
  \forall C\in\mathcal{B}(G) & \sum\limits_{\{u,v\}\in C} z_{uv} &=& 0 & (\dag)\\
  \forall \{u,v\}\in E & x_u-x_v &=& z_{uv}. & (\ddag)
  \end{array}\right\}
  \label{cycle}
\end{equation}

\begin{proposition}
  Eq.~\eqref{cycle} is an exact reformulation of Eq.~\eqref{dgp}.
  \label{prop:cycle}
\end{proposition}
\begin{proof}
  Let $x$ by any solution of Eq.~\eqref{dgp}. By the \textit{linearization constraints} ($\ddag$) we know that the objective function value of Eq.~\eqref{cycle} is zero. Moreover, for any $k\le K$ and any cycle $C=\{1,\ldots,c\}$ (wlog) in the graph we have:
  \begin{eqnarray*}
     (x_{1k}-x_{2k})+(x_{2k}-x_{3k})+\cdots+(x_{ck}-x_{1k}) &=&\\
    = x_{1k}-(x_{2k}-x_{2k})-(x_{3k}-x_{3k})+\cdots-x_{1k} &=& 0.
  \end{eqnarray*}
  Since every cycle can be generated linearly from the cycle basis $\mathcal{B}(G)$, the \textit{cycle constraints} ($\dag$) are also satisfied. Conversely, let $x'$ be a global optimum of Eq.~\eqref{cycle}. By replacing $z_{uv}$ with $x_u-x_v$ in the objective (which we can do since $x'$ satisfies the linearization constraints ($\ddag$) in Eq.~\eqref{cycle}), we see that $x'$ is also a global optimum of Eq.~\eqref{quartic}, which is therefore a solution of Eq.~\eqref{dgp}. 
\end{proof}

In fact, the proof of Prop.~\ref{prop:cycle} also holds without the cycle constraints, which implies that the following formulation is also correct:
\begin{equation}
  \left.\begin{array}{rrcl}
  \min\limits_{x\in\mathbb{R}^{nK}\atop z\in[-\mathbf{d},\mathbf{d}]^{mK}} & \sum\limits_{\{u,v\}\in E} (\|z_{uv}\|_2^2&-&d_{uv}^2)^2  \\
  \forall \{u,v\}\in E & x_u-x_v &=& z_{uv}.
  \end{array}\right\}
  \label{cyclesimple}
\end{equation}
We also note that Eq.~\eqref{cyclesimple} is a trivial reformulation of Eq.~\eqref{quartic} by linearization of the terms $x_u-x_v$ by the variables $z_{uv}$. The point of the \textsf{cycle} formulation is that the cycle constraints ($\dag$) tighten any relaxation of Eq.~\eqref{cyclesimple}, leading to better performances with global NLP solvers based on spatial Branch-and-Bound (sBB) \cite{couenne}. The performance of local NLP solvers deployed on Eq.~\eqref{cyclesimple} is impacted less clearly by the cycle constraints: for this reason, Eq.~\eqref{cyclesimple} deserves more in-depth study w.r.t.~local NLP optimization.

We note that Eq.~\eqref{cycle1}, \eqref{cycle}, and \eqref{cyclesimple} are all NLPs involving polynomials of degree 4. They can easily be reformulated based on the \textsf{system} (resp.~\textsf{push-and-pull}) formulation, yielding nonconvex QCQPs with nonconvex (resp.~convex) quadratic constraints. We write these reformulations for \eqref{cyclesimple}: the cycle constraints ($\dag$) may be added to Eq.~\eqref{cycsystem}-\eqref{cycsystem_ell1} to yield QCQP reformulations analogous to \eqref{cycle}:
\begin{center}
  \begin{minipage}{0.52\textwidth}
    \begin{equation}
      \left.\begin{array}{rrcl}
        \min\limits_{x\in\mathbb{R}^{nK}\!, s\in\mathbb{R}^{m}\atop z\in[-\mathbf{d},\mathbf{d}]^{mK}} & \sum\limits_{\{u,v\}\in E}\!\!s^2_{uv} && \\
        \forall \{u,v\}\in E & \|z_{uv}\|_2^2 &=& d_{uv}^2\!+\!s_{uv} \\
        \forall \{u,v\}\in E & x_u\!-\!x_v &=& z_{uv} 
      \end{array}\right\}
      \label{cycsystem}
    \end{equation}
  \end{minipage}
  \hfill
  \begin{minipage}{0.46\textwidth}
    \begin{equation}
      \left.\begin{array}{rrcl}
        \max\limits_{x\in\mathbb{R}^{nK}\atop z\in[-\mathbf{d},\mathbf{d}]^{mK}} & \sum\limits_{\{u,v\}\in E}\!\!\|&z_{uv}&\|_2^2  \\
        \forall \{u,v\}\in E & \|z_{uv}\|_2^2 &\le& d_{uv}^2 \\
        \forall \{u,v\}\in E & x_u\!-\!x_v &=& z_{uv}.
      \end{array}\right\}
      \label{cycpushpull}
    \end{equation}
  \end{minipage}
\end{center}
Eq.~\eqref{cycsystem}-\eqref{cycpushpull} correspond to the $\ell_2$-norm error. Another exact reformulation of Eq.~\eqref{cycsystem} can be written for the $\ell_1$-norm, similarly to Eq.~\eqref{system_ell1}: 
\begin{equation}
  \left.\begin{array}{rrcl}
    \min\limits_{x\in\mathbb{R}^{nK}\!,\, s\in\mathbb{R}^{2m}\atop z\in[-\mathbf{d},\mathbf{d}]^{mK}} & \sum\limits_{\{u,v\}\in E} (s^+_{uv} &+&s^-_{uv}) \\
    \forall \{u,v\}\in E & \|z_{uv}\|_2^2 &=& d_{uv}^2 + s^+_{uv} - s^-_{uv} \\
    \forall \{u,v\}\in E & x_u-x_v &=& z_{uv}.
  \end{array}\right\}
  \label{cycsystem_ell1}
\end{equation}

\subsubsection{The unassigned \textsf{cycle} formulation}
The \textsf{cycle} formulations that include the cycle constraints ($\dag$) involve the knowledge of the underlying graph, which is not available as part of the input in the UDGP. Therefore, we cannot derive unassigned formulations from Eq.~\eqref{cycle1} and \eqref{cycle}. We can derive the unassigned versions of the simplified \textsf{cycle} formulations without cycle constraints ($\dag$), namely Eq.~\eqref{cyclesimple}-\eqref{cycsystem_ell1}. We limit our treatment to the one derived from Eq.~\eqref{cycsystem_ell1} because it yields a MIQCQP instead of a general MINLP, as explained in Sect.~\ref{s:ass_syst}.
\begin{equation}
  \left.\begin{array}{r@{\quad}rcl}
    \min\limits_{x\in\mathbb{R}^{nK}\atop s\in\mathbb{R}^{2m},y\in\mathcal{A}} &  \sum\limits_{\ell\le m\atop i<j\le n} (s_{\ell}^+ + s_\ell^-) && \\
    \forall \ell\le m, i<j\le n & -s^-_{uv}-M(1\!-\!y_{ij\ell}) &\le& \|z_{ij}\|_2^2-\delta_\ell^2 \le s_{uv}^++M(1\!-\!y_{ij\ell}) \\
    \forall i<j\le n & x_i - x_j &=& z_{ij}.
  \end{array}\right\}
  \label{ucycsystem_ell1}
\end{equation}
Eq.~\eqref{ucycsystem_ell1} is a MIQCQP that can be solved using a global MIQCQP solver (possibly used locally).

\section{Matrix relaxations and approximations}
\label{s:matform}
In Sect.~\ref{s:form} we presented eleven MP formulations for the DGP and six for the UDGP. Among the DGP formulations, seven are (MI)QC(Q)P. Among the UDGP ones, three are (MI)QC(Q)P. The interest of limiting the polynomial degree to $2$ is that one can directly derive SDP relaxations of the original formulation \cite{biswasphd,ye,krislocksiam}. From these, one can then derive further linear relaxations and approximations by means of Diagonally Dominant Programming (DDP) \cite{isco16}, i.e.~linear programming over the primal and dual cones of diagonally dominant (DD) matrices, which is in fact a subclass of LP --- for which there exist extremely fast solvers. The application of SDP/DDP to the UDGP yields mixed-integer versions of the SDP/DDP formulations of the DGP. Since DDP $\subset$ LP, we obtain Mixed-Integer SDP (MISDP) relaxations, and MILP relaxations and approximations for the UDGP. 

We shall limit the application of DDP reformulations to those formulations of Sect.~\ref{s:form} where the quadratic terms only involve the realization variables $x$. These are: the $\ell_1$-norm \textsf{system} formulations Eq.~\eqref{system_ell1} and \eqref{usystem_ell1}, and the \textsf{push-and-pull}/\textsf{pull-and-push} formulations Eq.~\eqref{pushpull}-\eqref{pullpush}.

The basic reformulation steps to obtain (MI)SDP relaxations and (MI)LP relaxations and approximations is the same, as it applies term-wise to the expression $\|x_u-x_v\|_2^2$ and to the implicit constraints.

\subsection{Constructing the SDP relaxation}
\label{s:sdp}
Consider any DGP formulation in Sect.~\ref{s:form}, and in particular the term $\|x_u-x_v\|_2^2$ for any $\{u,v\}\in E$. We have:
\begin{eqnarray}
  \|x_u-x_v\|_2^2 &=& \|x_u\|_2^2 + \|x_v\|_2^2 - 2\langle x_u,x_v\rangle \nonumber \\
  &=& \langle x_u,x_u\rangle + \langle x_v,x_v\rangle - 2\langle x_u,x_v\rangle \nonumber \\
  &=& X_{uu} + X_{vv} - 2X_{uv}\label{sdplin}
\end{eqnarray}
by linearization of any term $\langle x_t,x_w$ (for $t,w\in V$) with the additional variable $X_{tw}$. There may be up to $n^2$ linearization variables organized in a symmetric $n\times n$ matrix $X$. We can now replace the nonlinear term $\|x_u-x_v\|_2^2$ by the linear term $X_{uu} + X_{vv} - 2X_{uv}$ in any DGP formulation, then add the defining constraint matrix
\begin{equation}
  X=x\transpose{x}. \label{Xxx}
\end{equation}
For any of the above DGP formulations of Sect.~\ref{s:form}, the symbolic procedure just described provides an exact reformulation. We note that such reformulations are not very convenient to solve, as satisfying Eq.~\eqref{Xxx} is a difficult task for most solvers.

Now we rewrite Eq.~\eqref{Xxx} as $X-x\transpose{x}=0$, then we relax this to $X-x\transpose{x}\succeq 0$, which, using the Schur complement, reads:
\begin{equation}
  \left(\begin{array}{cc}
    \mathbf{1}_K & \transpose{x} \\
    x & X
  \end{array}\right)\succeq 0. \label{Xxxsdp}
\end{equation}
If every occurrence of $x$ was eliminated by the linearization process, we can simplify Eq.~\eqref{Xxxsdp} to
\begin{equation}
  X\succeq 0. \label{Xsdp}
\end{equation}
SDP formulations are tractable cNLPs the nonlinearity of which is exclusively in the implicit constraint $X\succeq 0$. SDPs can be solved using an SDP solver, which runs in polynomial time up to any desired $\epsilon>0$ precision. The issue is practical, though: there are currently no SDP solvers to address SDPs with millions of variables and constraints, which is the current situation for LP solvers. 

\subsection{The DDP restriction}
\label{s:ddp}
The practical inadequacy of current SDP solvers motivates the search for inner and outer polyhedral cones to approximate the SDP cone. In this section we look at the cone of DD matrices. By Gershgorin's theorem \cite{gershgorin}, every DD matrix is also positive semidefinite (PSD), while the converse does not hold. Thus, if we replace the PSD constraint $X\succeq 0$ with ``$X$ is a DD matrix'' we obtain an inner approximation of an SDP formulation. There remain two questions: (i) how can we describe the DD cone explicitly, and (ii) how does the solution of a DDP help us solve the DGP?

The first question is easiest: an $n\times n$ symmetric matrix $X$ is DD if it satisfies:
\begin{equation}
  \forall i\le n \quad \sum\limits_{j\le n\atop j\not=i} |X_{ij}|\le X_{ii}. \label{dd}
\end{equation}
Eq.~\eqref{dd} is a piecewise-linear (hence nonlinear) constraint. But a linear description exists for it \cite[Thm.~3.9]{ahmadimajumdar}, based on linearizing the nonlinear term $|X_{ij}|$ by the components of a new matrix variable $T$:
\begin{eqnarray}
  \forall i\le n\quad \sum\limits_{j\le n\atop j\not=i} T_{ij}&\le& X_{ii} \label{dd1}\\
  -T\le X &\le& T. \label{dd2}
\end{eqnarray}
We let $\DD=\{X\in\mathbf{S}_n;|\;\exists T\;\mbox{Eq.~}\eqref{dd1}-\eqref{dd2}\}$, where $\mathbf{S}_n$ is the set of all $n\times n$ symmetric matrices, be the linear description of the DD cone. The DDP corresponding to a given SDP can then be derived by replacing $X\succeq 0$ with $X\in\DD$.

The second question depends on the output. Since $\DD\subsetneq \mathbf{S}_n^+=\{X\in\mathbf{S}_n\;|\;X\succeq 0\}$, there may be feasible SDPs where the corresponding DDP is infeasible: we can sometimes help this by relaxing some of the explicit constraints. However, if the DDP is feasible, we obtain a PSD solution matrix $X'$, which provides an interesting solution, since it can be factored (we shall see the significance of this below). On the other hand, although the DDP is a restriction of the corresponding SDP, we start from an SDP that is a relaxation of the original DGP formulation: we cannot directly infer any relationship between the original objective function value and the optimal objective value of the DDP.

DDP formulations are LPs, which can be solved with any LP solver.

\subsection{The dual DDP relaxation}
\label{s:dualddp}
If $C$ is a cone, its dual cone $C^\ast$ is defined as $\{\psi\;|\;\forall \phi\in C\;\langle \phi,\psi\rangle\ge 0\}$. It turns out that the dual DD cone $\dDD$ of $\DD$ is finitely generated by the matrices:
\[ E_{ij}^{\pm}=(e_i\pm e_j)\transpose{(e_i\pm e_j)} \]
for every $i,j\le n$, where $e_1,\ldots,e_n$ is the standard basis of $\mathbb{R}^n$ \cite{barker2}. In other words, we have
\begin{equation}
  \dDD = \{X\in\mathbf{S}_n\;|\;\forall i,j\le n\;\trace{X E_{ij}^\pm}\ge 0\}. \label{dualDD1}
\end{equation}
Equivalently, since each $E_{ij}^{\pm}$ is defined as the gram matrix of trivial linear combinations of basis vectors, we consider $\Delta=\{e_i\;|\;i\le n\}\cup \{e_i\pm e_j\;|\;i<j\le n\}$, and describe $\dDD$ as follows:
\begin{equation}
  \dDD = \{ X\in\mathbf{S}_n \;|\;\forall v\in\Delta\quad \transpose{v} Xv\ge 0 \}.
  \label{dualDD2}
\end{equation}

By Eq.~\eqref{dualDD1}, we can easily compute an explicit form for the constraints $\trace{X E_{ij}^\pm}\ge 0$. We know that $E_{ii}=\diag{e_i}$ for all $i\le n$, and that:
\begin{itemize}
\item $E_{ij}^+$ has the single nonzero minor $\left(\begin{array}{cc} 1_{ii}&1_{ij}\\1_{ij}&1_{jj}\end{array}\right)$;
\item $E_{ij}^-$ has the single nonzero minor $\left(\begin{array}{cc} 1_{ii}&-1_{ij}\\-1_{ij}&1_{jj}\end{array}\right)$.  
\end{itemize}
By inspection we have:
\begin{eqnarray*}
  \forall i\le n \quad \trace{X E_{ii}} &=& X_{ii} \label{dD1} \\
  \forall i,j\le n \quad \trace{X E_{ij}^+} &=& X_{ii}+X_{jj}+2X_{ij} \label{dD2} \\
  \forall i,j\le n \quad \trace{X E_{ij}^-} &=& X_{ii}+X_{jj}-2X_{ij}. \label{dD3}
\end{eqnarray*}
Therefore, we can define $\dDD$ by explicit constraints as follows:
\begin{equation}
  \dDD = \{ X\in\mathbf{S}_n \;|\; \diag{X}\ge 0\;\land\;\forall i<j\le n\;\;X_{ii}+X_{jj}\pm 2X_{ij}\ge 0\}. \label{dualDD3}
\end{equation}

The representation of $\dDD$ in Eq.~\eqref{dualDD2} helps us prove that the dual DDP cone is an outer approximation of the PSD cone: by Eq.~\eqref{dualDD2} we have that $X\in\dDD$ if $\transpose{v}Xv\ge 0$ for $v\in\Delta\subset\mathbb{R}^n$. Since PSD matrices are all and those for which $\transpose{v}Xv$ for all $v\in\mathbb{R}^n$, the inclusion $\mathbf{S}_n^+\subsetneq\dDD$ follows.

Dual DDP formulations belong to the LP class: they can therefore be solved by any LP solver.

\subsection{Matrix reformulations of the DGP and UDGP}
\label{s:matrixform}
Based on the sets $\mathbf{S}_n^+$, $\DD$, and $\dDD$, we can define SDP relaxations, DDP restrictions, and dual DDP relaxations for the QCQP and MIQCQP formulations in Sect.~\ref{s:form} where the quadratic terms only involve the $x$ variables.

\subsubsection{From the \textsf{system} formulation}
From the DGP \textsf{system} formulation in Eq.~\eqref{system_ell1}, for all $\mathbf{X}\in\{\mathbf{S}^+_n,\DD,\dDD\}$ we derive the following DGP reformulations:
\begin{equation}
  \left.\begin{array}{rrcl}
    \min\limits_{X\in\mathbf{X}\atop s\in\mathbb{R}^{2m}} & \sum\limits_{\{u,v\}\in E} (s^+_{uv} &+& s^-_{uv}) \\
    \forall \{u,v\}\in E & X_{uu}+X_{vv}-2X_{uv} &=& d_{uv}^2 + s^+_{uv} - s^-_{uv}.
  \end{array}\right\} 
  \label{Xsystem_ell1}
\end{equation}
We remark that Eq.~\eqref{Xsystem_ell1} (as well as all the formulations in this section) describes \textit{three} different formulations depending on the symbol $\mathbf{X}$ that ranges in the PSD cone $\mathbf{S}_n^+$, the DD cone $\DD$, the dual DD cone $\dDD$: the first is an SDP relaxation, the second an inner LP approximation of the SDP, and the third an outer LP relaxation of the SDP. Thus, the formulations in Sect.~\ref{s:matrixform} are actually \textit{meta-formulations}: they become formulations only after fixing the meaning of the symbol $\mathbf{X}$.

From the UDGP \textsf{system} formulation in Eq.~\eqref{usystem_ell1}, for all $\mathbf{X}\in\{\mathbf{S}^+_n,\DD,\dDD\}$ we derive the following UDGP reformulations:
\begin{equation}
  \left.\begin{array}{r@{\hspace*{0.4em}}rcl}
    \min\limits_{X\in\mathbf{X},y\in\mathcal{A}\atop s\in\mathbb{R}^{2m}} &  \sum\limits_{\ell\le m\atop i<j\le n} (s_{\ell}^+ + s_\ell^-) && \\
    \forall {\small \left\{\begin{array}{l}\ell\le m\\ i<j\le n\end{array}\right.} & -s^-_{uv}\!-\!M(1\!-\!y_{ij\ell}) \le &X_{ii}\!+\!X_{jj}\!-\!2X_{ij}\!-\!\delta_\ell^2 &\le s_{uv}^+\!+\!M(1\!-\!y_{ij\ell}).
  \end{array}\right\}
  \label{Xusystem_ell1}
\end{equation}
We recall that $\mathcal{A}$ describes the binary variables $y$ and corresponding assignment constraints (see Sect.~\ref{s:unass}). The three formulations described in Eq.~\eqref{Xusystem_ell1} are a MISDP (for $\mathbf{X}=\mathbf{S}_n^+$) and two MILPs (otherwise). 

\subsubsection{From \textsf{push-and-pull} formulations}
From the DGP \textsf{push-and-pull} formulation in Eq.~\eqref{pushpull}, for all $\mathbf{X}\in\{\mathbf{S}^+_n,\DD,\dDD\}$ we derive the following DGP reformulations:
\begin{equation}
  \left.\begin{array}{rrcl}
    \max\limits_{X\in\mathbf{X}} & \sum\limits_{\{u,v\}\in E} (X_{uu}+X_{vv}-2X_{uv}) && \\
    \forall \{u,v\}\in E & X_{uu}+X_{vv}-2X_{uv} &\le& d_{uv}^2.
  \end{array}\right\}
  \label{Xpushpull}
\end{equation}
For the related \textsf{pull-and-push} formulation in Eq.~\eqref{pullpush} we derive:
\begin{equation}
  \left.\begin{array}{rrcl}
    \min\limits_{X\in\mathbf{X}} & \sum\limits_{\{u,v\}\in E} (X_{uu}+X_{vv}-2X_{uv}) && \\
    \forall \{u,v\}\in E & X_{uu}+X_{vv}-2X_{uv} &\ge& d_{uv}^2.
  \end{array}\right\}
  \label{Xpullpush}
\end{equation}
Eq.~\eqref{Xpullpush} with $\mathbf{X}=\DD$ is only motivation for the original formulation Eq.~\eqref{pullpush} (the constraints of which are concave, and therefore hard to satisfy): because of the potential feasibility issues of solving DDPs (i.e.~the DDP might be infeasible even if the original SDP is feasible), Eq.~\eqref{Xpushpull}-\eqref{Xpullpush} enlarge the feasible region both $\le,\ge$ constraint senses: at least one of them must be feasible.

When $\mathbf{X}=\mathbf{S}_n$, we also propose a related formulation mentioned by Yinyu Ye in one of his course slides:
\begin{equation}
  \left.\begin{array}{rrcl}
    \max\limits_{X\succeq 0} & \trace{X} && \\
    \forall \{u,v\}\in E & X_{uu}+X_{vv}-2X_{uv} &=& d_{uv}^2.
  \end{array}\right\}
  \label{yecourse}
\end{equation}
The constraints of Eq.~\eqref{yecourse} are an SDP relaxation of Eq.~\eqref{dgp}. If the original DGP instance is NO because it has no realization in $\mathbb{R}^K$, then Eq.~\eqref{yecourse} is feasible. On the other hand, if the original DGP instance is NO because the given distances cannot be realized in any dimension, then Eq.~\eqref{yecourse} is infeasible. Obviously, if the DGP instance is YES then Eq.~\eqref{yecourse} is also feasible.

For feasible cases of Eq.~\eqref{yecourse}, using equations instead of inequalities (as in Eq.~\eqref{Xpushpull}-\eqref{Xpullpush}) produces a tighter relaxation. The objective function heuristically attempts to reduce the rank of the solution, as
\[\trace{X}=\trace{P\Lambda \transpose{P}}=\trace{P\transpose{P}\Lambda}=\trace{\Lambda}=\sum_{u\in V}\lambda_u,\]
where $P\Lambda\transpose{P}$ is an eigendecomposition of $X$, and minimizing the sum of eigenvalues should help decrease the rank of $X$ (we shall see why this is convenient in Sect.~\ref{s:rank}).

We do not derive UDGP versions from any of the formulations in this section, since it would yield a product of variables ($y$ and $X$) in the objective function, resulting in nonlinear programs with cone constraints (which are impractical to solve). 

\subsection{DGP post-processing}
\label{s:rank}
In this section we assume that the problem being solved is a DGP. From matrix formulations we do not obtain a realization $x'\in\mathbb{R}^{nK}$ as output, but a symmetric matrix $X'\in\mathbb{R}^{n\times n}$. If $X'$ is PSD, as would happen for SDP and DDP, then $X'$ is a Gram matrix of a realization, so it can be written as $X'=\xi\transpose{\xi}$, where $\xi\in\mathbb{R}^{n\times n}$: in general, $\xi$ can be interpreted as a realization in $\mathbb{R}^n$ instead of in $\mathbb{R}^K$. If $X$ is the solution of a dual DDP, then $\xi\in\mathbb{C}^n$ in general. In both cases, in order to find the ``closest'' realization in $K$ dimensions, we must reduce the rank of the realization points (in $\mathbb{R}^n$ or $\mathbb{C}^n$ to obtain a realization matrix in $x'\in\mathbb{R}^{nK}$.

In our past work we have considered two dimensionality reduction methodologies: Principal Component Analysis (PCA) \cite{hotelling,wikipedia_pca} and Barvinok's naive algorithm extended to $K$ dimensions \cite{barvinok2,barvinok_orl}. PCA can be applied to the rows of $\xi$ (the point vectors) to reduce them to the $K$ principal components (or possibly even fewer if $X'$, as the solution of a dual DDP, fails to be a PSD matrix). Barvinok's naive algorithm can only be applied to solutions of SDPs and DDPs (with dual DDPs one may simply hope for the best). PCA produces the $K$-dimensional realization $x'$ closest to the solution of the SDP or DDP. Barvinok's naive algorithm produces, with arbitrarily high probability, a $K$-dimensional realization $x'$ that is ``reasonably close to'' (or just ``not too far from'') a realization of the original DGP. Comparative computational experiments between these two rank reduction methods can be found in \cite{barvinok_orl}.

Once a matrix $x'\in\mathbb{R}^{nK}$ has been computed, it can be \textit{refined}, i.e.~its realization error can be reduced, by using $x'$ as a starting point on any one of the DGP formulations of NLP/QCQP type in Sect.~\ref{s:form} solved by a local NLP/QCQP solver. This yields an approximate realization $x^\ast$ of the DGP.

\subsubsection{A solution process for DGPs}
\label{s:process}
In summary, we propose the following process for solving DGP instances too difficult or too large to be dealt with by global NLP/QCQP solvers:
\begin{enumerate}
\item Solve an SDP/DDP/dual DDP reformulation of the DGP instance (this should be reasonably fast), obtain an $n\times n$ symmetric matrix solution $X$, and factor it as $\xi\transpose{\xi}$.
\item Reduce the rank of $\xi$ to an $n\times K$ realization matrix $x'$ using various dimensionality reduction methods.
\item Improve the quality of the realization $x'$ by using it as a starting point in a local NLP solver, which will yield a good-quality realization $x^\ast$ of the given DGP instance.
\end{enumerate}

\subsection{UDGP post-processing}
\label{s:udgp:post}
Solving UDGPs poses the problem of graph reconstruction, as discussed in Sect.~\ref{s:theory}. Solving UDGPs globally is only possible for tiny instances, in general. Thus, the MISDP/MILP matrix reformulations of the UDGP in Eq.~\eqref{Xusystem_ell1} are the only practically viable possibility to handle medium to large-sized UDGP instances. High-quality global MILP solvers can be configured to find all (or many) solutions during the search. In general, one will find solutions $(X',y')$ where $X'$ is the matrix solution and $y'$ encodes the assignment $\alpha$.

The first post-processing task to carry out is the reconstruction of the DGP graph $G_\alpha$ (see Sect.~\ref{s:theory}). The second post-processing task is to work out a realization $x\in\mathbb{R}^{nK}$ of the reconstructed graph. There are two possibilities: either one considers the matrix solution $X'$, or one discards it. Paired with the graph $G_\alpha$, the solution $X'$ can be used as described in Sect.~\ref{s:rank}, i.e.~rank reduction followed by refinement. If $X'$ is discarded, the graph $G_\alpha$ defines a DGP, which can be solved using the process given in Sect.~\ref{s:process}. The second possibility was adopted in \cite{udgp_buckminster}, since the first gave poor quality realizations. 

\subsection{Two remarks over concave constraints}
In the context of the (assigned) DGP and matrix formulation, we ignore pairs $u,v$ of vertices that are not edges in the graph. A popular way to treat them is to add a ``greater than or equal'' constraints with respect to some upper distance threshold $\bar{d}$, whenever one is known, i.e.
\begin{equation}
\forall \{u,v\}\not\in E \quad \|x_u-x_v\|_2^2 \ge \bar{d}
\label{gecon}
\end{equation}
in the case of vector formulations, or
\begin{equation}
X_{uu}+X_{vv}-2X_{uv} \ge \bar{d}^2
\label{gelin}
\end{equation}
in the case of matrix formulations. This might help avoid the typical ``overclustering'' effect of realization points around the origin for weakly connected vertices. We do not consider these constraints for two reasons: the first, and foremost, is that it assumes that \textit{all} distances shorter than $\bar{d}$ have been measured, which may not be the case for proteins (and other molecules) --- and even a single wrong constraint of this type is likely to change the resulting structure considerably. The second reason is that Eq.~\eqref{gecon} is a concave constraint, which is hard to enforce for many solvers (the linearized version in Eq.~\eqref{gelin}, by contrast, is simply a linear constraint). 

It is well known that strict constraints cannot appear in MP formulations, since they yield open sets: and optima over open sets may not exist. Typically, with an upper bound $\bar{d}$ on distance values, mathematicians will want to impose
\begin{equation*}
\forall \{u,v\}\in E\quad \|x_u-x_v\|_2^2 \le \bar{d}
\end{equation*}
and 
\begin{equation*}
\forall \{u,v\}\not\in E\quad \|x_u-x_v\|_2^2 > \bar{d},
\label{strict}
\end{equation*}
which might make optima of any MP formulation including Eq.~\eqref{strict} in its contraints non-existent. One possible reformulation of Eq.~\eqref{strict} is to trade openness for unboundedness. We introduce an auxiliary variable $t\ge 0$ in the formulation, and rewrite Eq.~\eqref{strict} as:
\begin{equation}
\forall \{u,v\}\not\in E\quad \|x_u-x_v\|_2^2 \ge \bar{d} + e^{-t}.
\label{strictref}
\end{equation}
Unboundedness in $t$ is only possible if $\|x_u-x_v\|_2^2$ is forced to be exactly equal to $\bar{d}$, which can only happen if the bound $\bar{d}$ is too small. While Eq.~\eqref{strictref} is nonconvex in general, its linearized version
\begin{equation}
\forall \{u,v\}\not\in E\quad X_{uu}+X_{vv}-2X_{uv} \ge \bar{d} + e^{-t}
\label{strictreflin}
\end{equation}
is a convex constraint, which is representable using an exponential cone. If the rest of the formulation is also a conic program (such as e.g.~SDP or DDP formulations), the whole problem can be solved efficiently using a conic programming solver.

\section{Computational evaluation}
\label{s:results}
In this section we attempt to answer the following questions.
\begin{enumerate}
\item Up to what size can we solve DGP/UDGP instances to global optimality in an acceptable time on a laptop?
\item For DGP/UDGP instances of various types, is it better to run a simple stochastic matheuristics (see below) around an exact formulation solved locally, or use the matrix formulation based solution process?
\item Is MP practically useful for solving realistic DGP/UDGP instances?
\end{enumerate}

We shall answer the first question by finding size thresholds beyond which the exponential nature of global optimization algorithms on our MP formulations becomes limiting. 

The second question deserves an explanation. A \textit{matheuristic} is a heuristic algorithm (i.e.~that does not provide exactness guarantees) based on a MP formulation. A \textit{stochastic} algorithm uses an element of randomness during its execution. The simplest stochastic matheuristic is MultiStart (MS) \cite{schoen2002,sobolopt}, shown in Alg.~\ref{a:ms}.
\begin{algorithm}[!ht]
\begin{algorithmic}
  \STATE initialize $\bar{x}$ to \texttt{NaN}
  \WHILE{resource limit not reached}
  \STATE sample random starting point $x'$
  \STATE solve a MP formulation locally from $x'$, obtain $\bar{x}$
  \IF{$\bar{x}$ improves on previous optimum}
  \STATE update $x^\ast\leftarrow \bar{x}$
  \ENDIF
  \STATE return $x^\ast$
  \ENDWHILE
\end{algorithmic}
\caption{MultiStart}
\label{a:ms}
\end{algorithm}
We shall answer the second question by means of a comparison between MS and the process described in Sect.~\ref{s:process} (for DGP) and \ref{s:udgp:post} (for UDGP).

We propose to answer the third question by attempting to reconstruct protein shapes from distance data (both with and without the graph). 

Finally, we note that centroid constraints (Eq.~\eqref{centroid}) have been added to all of our formulations. Although the results in \cite{barvinok_orl} report a computational advantage slightly in favour of Barvinok's naive algorithm, we have chosen to use PCA as a dimensional reduction method on matrix formulation solutions, simply because it is better known. 

\subsection{Instances}
\label{s:inst}
DGP instances are organized in three families.
\begin{itemize}
\item A set $\mathcal{R}$ of 28 random biconnected graphs of given sparsity and size, based on a vertex set of points in the Euclidean plane $\mathbb{R}^2$, with edges generated by means of an Erd\H{o}s-Renyi process over a starting Hamiltonian cycle and weighted by Euclidean distance between the corresponding points, to be realized in $K=2$ (all of these instances are YES by construction). The set $\mathcal{R}$ is used to answer the first question, i.e.~to what instance size can we solve DGPs and UDGPs to guaranteed optimality?
\item A collection $\mathcal{G}$ of 309 graphs of different types and sizes, some weighted some not (i.e., with unit weight), to be realized in $K=2$ (most of the randomly weighted versions of these instances are generically NO; some of the others are YES). The set $\mathcal{G}$ is used to answer the second question, i.e.~is it better to solve DGP/UDGP formulations locally withing a MS algorithm (Alg.~\ref{a:ms}), or use matrix formulation based solution process (Sect.~\ref{s:matrixform})?
\item A set $\mathcal{P}$ of protein graphs simulating NOESY experiments with known covalent bonds and angles: in other words a set of disk graphs (of radius $5.5${\AA}) on a vertex set of points in $\mathbb{R}^3$ (all of these instances are YES by definition). We use this set to answer the third question: are MP formulations actually useful in practice in regard to the DGP/UDGP?
\end{itemize}

\subsubsection{Random euclidean graphs}
\label{s:inst:eucl2d}
More precisely, the set $\mathcal{R}$ consists of 28 random biconnected graph generated with the Erd\H{o}s-Renyi model for each vertex set size $n\in\{5,8,10,12,15, 18, 20\}$, and for each sparsity parameter $p\in\{0.4,0.8, 0.9, 0.95\}$.

\subsubsection{Different graph types}
\label{s:inst:gph}
The set $\mathcal{G}$ is composed as follows: 
\begin{enumerate}
   \ifspringer\else\setlength{\parskip}{-0.2em}\fi
\item 18 almost $k$-regular graphs on $n$ vertices (9 randomly weighted and 9 unweighted),
\item 18 random graphs on $n$ vertices with edge generation probability $p$ (9 randomly weighted and 9 unweighted);
\item 18 bipartite graphs on $n+n$ vertices with edge generation probability $p$ (9 randomly weighted and 9 unweighted),
\item 18 tripartite graphs on $n+n+n$ vertices with edge generation probability $p$ (9 weighted and 9 unweighted),
\item 10 square meshes with $n^2$ vertices (5 weighted and 5 unweighted),
\item 10 torus meshes with $n^2$ vertices as a folded-up square mesh (5 weighted and 5 unweighted),
\item 10 triangular meshes with $n$ vertices per side (5 weighted and 5 unweighted),
\item 126 clustered graphs with $k$ clusters on $n$ vertices with intra-cluster edge generation probability $p$ and inter-cluster edge generation probability $q$ (63 weighted and 63 unweighted),
\item 6 power law graphs on $n$ vertices where the degree of vertex $i$ is $\lceil n\alpha i^{-\tau}\rceil$ with $\alpha\in(0,1),\tau>0$ (3 weighted and 3 unweighted),
\item 18 chain of $k$-cliques on $n$ vertices (9 weighted and 9 unweighted),
\item 10 chain of triangles on $n$ vertices (5 weighted and 5 unweighted),
\item 18 DMDGP \cite{dgp-sirev} on $n$ vertices with $k$ contiguous adjacent predecessors (9 weighted and 9 unweighted),
\item 5 Beeker-Glusa graphs \cite{dgpinnp}: chains of triangles with specific edge costs (weighted only),
\item 6 local graphs on $n$ vertices with edge threshold $t$: vertices are $n$ points in the plane, edges exist if Euclidean distance between two points shorter than $t$ (weighted only),
\item 18 norm graphs on $n$ vertices chosen as points in the plane with edge generation probability $p$, edges are weighted by $\ell_1$ and $\ell_\infty$ distances between points (weighted only);
\end{enumerate}

\subsubsection{Protein instances}
\label{s:inst:prot}
The set $\mathcal{P}$ is composed of protein graphs constructed from the Protein Data Bank (PDB) \cite{pdb} in such a way as to roughly mimick the output of a NOESY experiment on an NMR machine. We selected a set of proteins (and pieces thereof) that cover a reasonable spectrum size, from small to reasonably large (see Table \ref{t:protsz}). For each of these we extracted the first available realization in the PDB and computed all of the inter-atomic distances, then we discarded those with length larger than $5.5${\AA}. The foremost difference between these instances and those actually obtained from NOESY experiments followed by distance assignment processes is that there are no mis-assigned edges in the protein graphs in $\mathcal{P}$.
\begin{table}[!ht]
  \begin{center}
    \begin{tabular}{l|rr}
      Name & $|V|$ & $|E|$ \\ \hline
      \textsf{tiny} & 38 & 335 \\
      \textsf{1guu-1} & 150 & 959 \\
      \textsf{1guu-4000} & 150 & 968 \\
      \textsf{C0030pkl} & 198 & 3247 \\
      \textsf{1PPT} & 303 & 3102 \\
      \textsf{1guu} & 427 & 955 \\ 
      \textsf{100d} & 491 & 5741 \\
      \textsf{3al1} & 681 & 17417 \\
      \textsf{1hpv} & 1633 & 18512 \\
      \textsf{il2} & 2098 & 45251 \\
      \textsf{1tii} & 5691 & 69800 
    \end{tabular}
  \end{center}
  \caption{Vertex and edge set sizes of protein instances in the class $\mathcal{P}$.}
  \label{t:protsz}
\end{table}
We note that these graphs are sometimes disconnected. For example, \textsf{tiny} consists of a connected component of 37 atoms and a single disconnected atom (an isolated vertex in the graph); and \textsf{1guu} has 277 isolated vertices (in fact it actually has only 150 connected atoms, like its kin instances \textsf{1guu-1} and \textsf{1guu-4000}). We chose to keep such occurrences because graphs occurring from applications are often atypical with respect to the usual assumptions of connectedness: and benchmarks on these instances are supposed to verify the practical usefulness of our formulations. 

\subsubsection{UDGP versions of $\mathcal{R}$, $\mathcal{G}$, $\mathcal{P}$}
\label{s:inst:udgp}
All of our UDGP instances were derived from DGP ones by discarding the graph structure: we only keep $K,n$, and the list $L$ of distance values from the graph edges. 

\subsection{Hardware and software}
\label{s:inst:hwsw}
All tests have been carried out on an Intel architecture server with: 2 Intel Xeon Platinum 8362 CPUs at 2.80GHz, each with 32 cores with hyperthreading, for a total of 128 cores; and 2TB RAM. The software system consists of a set of of Python 3 \cite{python3} scripts, bash scripts, and AMPL \cite{ampl} code.

This system calls LP, MILP, SDP, (local and global) NLP, MISDP, and MINLP solvers as follows: CPLEX 22.1.1 \cite{cplex221} for LP, Gurobi 10.0.1 \cite{gurobi} for MILP, nonconvex NLP, and MINLP,
IPOPT 3.4.11 \cite{ipopt} for solving NLP locally, SCS 3.2.3 \cite{scs} for SDP, Pajarito \cite{pajarito} for MISDP (called from a Julia \cite{julia} script, and using Gurobi and Mosek \cite{mosek10} as subsolver). The MILP/MINLP solver is deployed with a CPU time limit of 1800s. The LP, SDP, and MISDP solvers are used without specific configurations. The local NLP solver is used within a MS algorithm limited to 5 iterations unless specified otherwise.

We remark that there is an option for providing IPOPT with a termination based on CPU time limit, but the verification for this type of termination is only carried out in a certain outer phase of the algorithm. This means that IPOPT can (and often does) exceed the given CPU time limit by arbitrary amounts. We therefore decided to refrain from imposing a time limit on IPOPT in our computational experiments. This explains the fact that our reported CPU times may dramatically exceed the default CPU time limit. 

\subsection{DGP tests and results}
\label{s:dgpres}
\label{s:res:forms}

We have tested the following formulations:
\begin{enumerate}
   \ifspringer\else\setlength{\parskip}{-0.2em}\fi
\item \textsf{cycle} (Eq.~\eqref{cycle}),
\item \textsf{cyclesimple} (Eq.~\eqref{cyclesimple}),
\item \textsf{cycpushpull} (Eq.~\eqref{cycpushpull} with cycle constraints ($\dag$)),
\item \textsf{cycsimplepushpull} (Eq.~\eqref{cycpushpull}),
\item \textsf{cycsimplesys1} (Eq.~\eqref{cycsystem_ell1}),
\item \textsf{cycsimplesys2} (Eq.~\eqref{cycsystem}),
\item \textsf{cycsys1} (Eq.~\eqref{cycsystem_ell1} with ($\dag$)),
\item \textsf{cycsys2} (Eq.~\eqref{cycsystem} with ($\dag$)),
\item \textsf{pushpull} (Eq.~\eqref{pushpull}),
\item \textsf{pullpush} (Eq.~\eqref{pullpush}),
\item \textsf{quartic} (Eq.~\eqref{quartic}),
\item \textsf{system1} (Eq.~\eqref{system_ell1}),
\item \textsf{system2} (Eq.~\eqref{system}),
\end{enumerate}
both by themselves, and used as refinement steps to:
\begin{enumerate}
\item an SDP matrix formulation, i.e.~a variant of Eq.~\eqref{yecourse} with a modified objective
  \[\sum_{\{u,v\}\in E}(X_{uu}+X_{vv}-2X_{uv}) + 0.1\,\trace{X},\]
which was heuristically found to perform slightly better on protein instances\footnote{With the constraints of Eq.~\eqref{yecourse}, the term $\sum_{\{u,v\}\in E}(X_{uu}+X_{vv}-2X_{uv})$ in the objective is actually equal to the constant $\sum_{uv}d_{uv}^2$, and so it should be irrelevant. But not every SDP or local NLP solver always ensures feasibility at every step: this depends on the reformulations and algorithms it implements. We think that this fact might give this formulation the slight empirical advantage we observed in previously conducted experiments.}: we note that Eq.~\eqref{yecourse} is infeasible on NO instances of the DGP where the graph cannot be realized in any dimension;
\item the corresponding DDP and dualDDP polyhedral approximations that replace the PSD cone $\mathbf{S}_n^+$ with the DDP and dualDDP cones $\mathbf{D}_n,\mathbf{D}^\ast_n$; we also relaxed the equations $X_{uu}+X_{vv}-2X_{uv}=d_{uv}^2$ to $\ge$-inequalities in the DDP to prevent an excessive number of infeasibilities.
\end{enumerate}

We present our computational results grouped in various ways:
\begin{itemize}
\item by approximate vertex set size (graphs grouped by $|V|$ closest to multiples of 10);
\item by approximate edge set size (graphs grouped by $|E|$ closest to multiples of 50);
\item by approximate edge density $\frac{|E|}{V(V-1)/2}$ (graphs grouped by density values closest to multiples of $1/10$);
\item by graph type (only for the class $\mathcal{G}$), where randomly weighted graph types have their name prefixed by `W');
\item by formulation type.
\end{itemize}

The accuracy of our DGP results was described by \textit{mean distance error} (\textsf{mde}), \textit{largest distance error} (\textsf{lde}), and the algorithmic performance by seconds of CPU time. For a realization $x\in\mathbb{R}^{nK}$ of a DGP instance $(K,G)$ where $G=(V,E,d)$, we have:
\begin{eqnarray}
  \mathsf{mde}(x) &=& \sum\limits_{\{u,v\}\in E} \big|\,\|x_u-x_v\|_2^2-d_{uv}^2\,\big| \label{mde}\\
  \mathsf{lde}(x) &=& \max\limits_{\{u,v\}\in E} \big|\,\|x_u-x_v\|_2^2-d_{uv}^2\,\big|.\label{lde}
\end{eqnarray}

\subsubsection{The Euclidean graph collection $\mathcal{R}$}
\label{s:dgp:euclgph}
The point of this graph collection (Sect.~\ref{s:inst:eucl2d}) is to provide a testbed of small graphs (all of which are YES instances of the DGP) for answering our first question concerning the DGP: how far can we go up in size and yet obtain a DGP realization with an algorithmic guarantee that the realization is precise, at least up to an $\epsilon>0$ tolerance and in a reasonable amount of time? We consider an optimality tolerance of $10^{-6}$, and a ``resonable amount of time'' to mean 1800s of CPU time. We solve these instances with the Gurobi solver.

We only consider the exact formulations \textsf{cycle}, \textsf{cyclesimple}, \textsf{cycpushpull}, \textsf{cycsimplepushpull}, \textsf{cycsimplesys1}, \textsf{cycsimplesys2}, \textsf{cycsys1}, \textsf{cycsys2}, \textsf{pullpush}, \textsf{pushpull}, \textsf{quartic}, \textsf{system1}, \textsf{system2} (see Sect.~\ref{s:res:forms}). We do not consider any of the matrix formulations (SDP, DDP, dual DDP) because none of them is exact. 

We present average results grouped by approximate vertex cardinality, and its corresponding bar plot figure, in Table \ref{t:Rvtx}. The results are shown terms of \textsf{mde}, \textsf{lde}, CPU time.
\begin{table}[!ht]
  \begin{center}
    \begin{minipage}{0.48\textwidth}
      \begin{tabular}{l|rrr}
        \textbf{$\approx |V|$} & \textsf{mde} & \textsf{lde} & CPU \\ \hline
        10 & 0.0001 & 0.0003 & 192.63 \\
        20 & 0.3884 & 1.8223 & 1356.60
      \end{tabular}
    \end{minipage}
    \begin{minipage}{0.5\textwidth}
      \includegraphics[width=\textwidth,height=0.13\textheight]{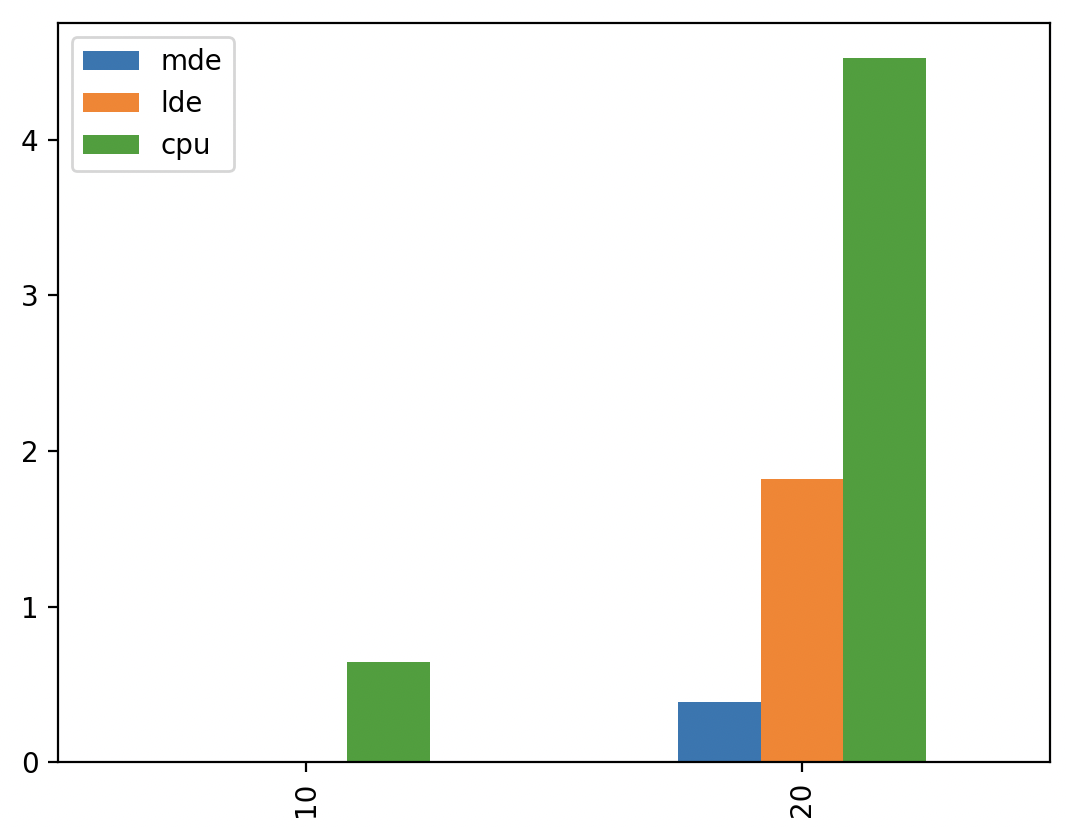}
    \end{minipage}  
  \end{center}
  \caption{Average results on approximate vertex cardinality and the corresponding bar plot for the graph class $\mathcal{R}$ (the CPU time column was scaled by 1/300).}
  \label{t:Rvtx}
\end{table}
These data show that guaranteed globally optimal solutions may only be found in 1800s of CPU time up to $|V|=15$.

In Table \ref{t:Redge} (left), we present average results grouped by formulation type. The corresponding bar plot figure is shown in Table \ref{t:Redge} (right).
\begin{table}[!ht]
  \begin{center}
    \begin{minipage}{0.48\textwidth}
      \begin{tabular}{l|rrr}
        \textbf{$\approx |E|$} & \textsf{mde} & \textsf{lde} & CPU \\ \hline
        50 & 0.0085 & 0.0678 & 294.37 \\
        100 & 0.1711 & 0.9917 & 1162.36 \\
        150 & 0.3331 & 1.6304 & 1594.27 \\
        200 & 1.1886 & 4.9158 & 1794.13
      \end{tabular}
    \end{minipage}
    \begin{minipage}{0.5\textwidth}
      \includegraphics[width=\textwidth,height=0.15\textheight]{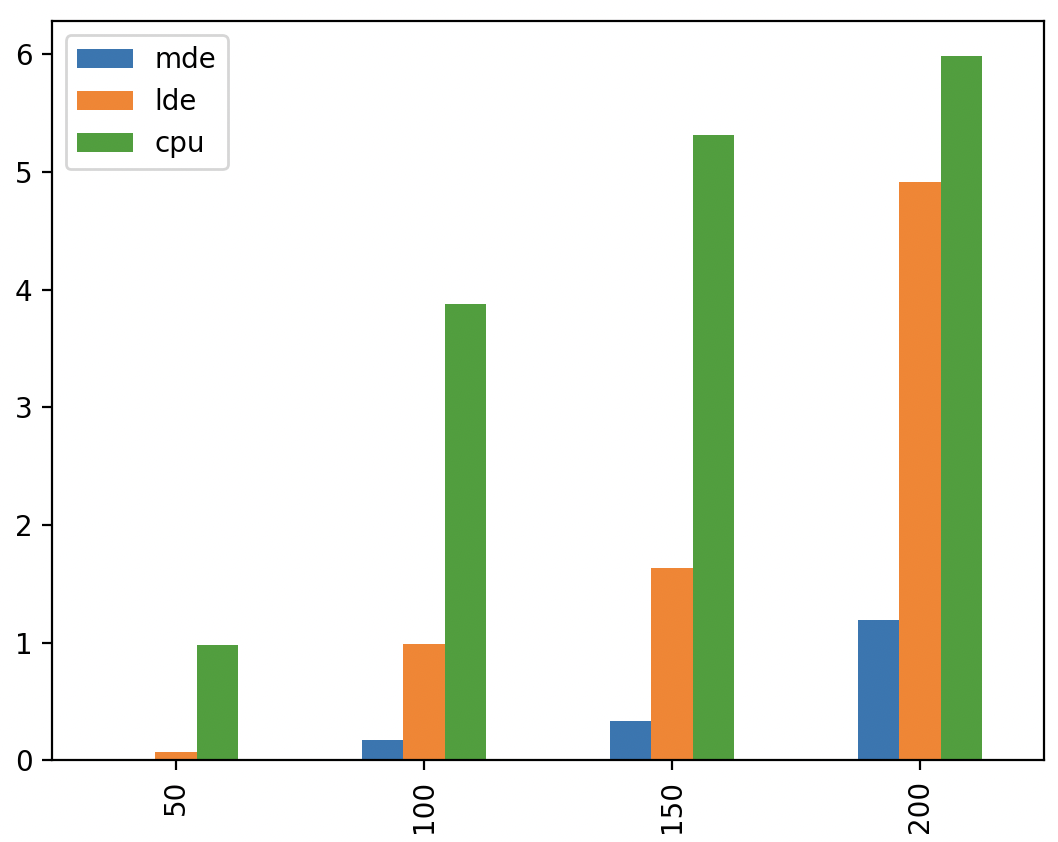}
    \end{minipage}  
  \end{center}
  \caption{Average results on approximate edge cardinality and the corresponding bar plot for the graph class $\mathcal{R}$ (the CPU time column was scaled by 1/300).}
  \label{t:Redge}
\end{table}
These data show that guaranteed optimal solutions in 1800s are more likely to happen with as few edges as possible: very likely with 50 edges, and very unlikely with 200.

In Table \ref{t:Rdens} (left), we present average results grouped by edge density. The corresponding bar plot figure is shown in Table \ref{t:Rdens} (right).
\begin{table}[!ht]
  \begin{center}
    \begin{minipage}{0.48\textwidth}
      \begin{tabular}{l|rrr}
        \textbf{$\approx$ density} & \textsf{mde} & \textsf{lde} & CPU \\ \hline
        0.4 & 0.3945 & 2.3032 & 1532.25 \\
        0.5 & 0.3024 & 1.8276 & 1160.63 \\
        0.6 & 0.0001 & 0.0004 & 373.80 \\
        0.7 & 0.0000 & 0.0001 & 1062.62 \\
        0.8 & 0.3705 & 1.7659 & 1233.24 \\
        0.9 & 0.2539 & 0.9882 & 880.90 \\
        1.0 & 0.1654 & 0.7392 & 609.35 
      \end{tabular}
    \end{minipage}
    \begin{minipage}{0.5\textwidth}
      \includegraphics[width=\textwidth,height=0.2\textheight]{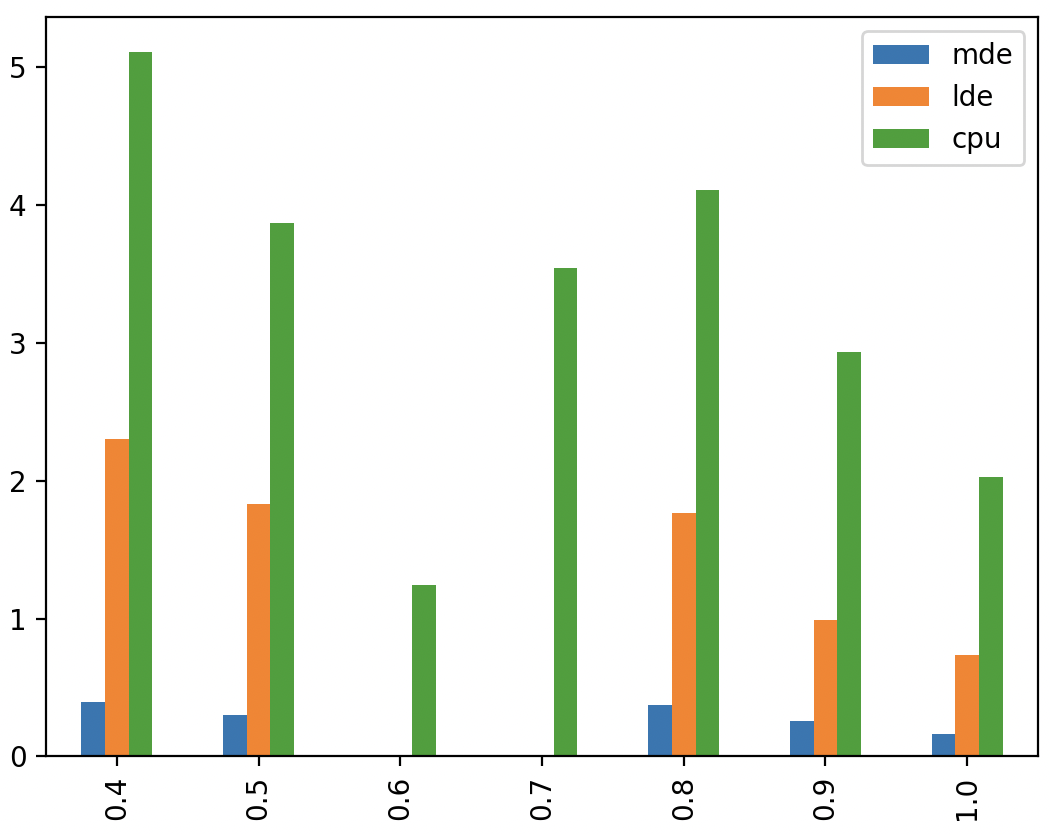}
    \end{minipage}  
  \end{center}
  \caption{Average results on approximate edge density and the corresponding bar plot for the graph class $\mathcal{R}$ (the CPU time column was scaled by 1/300).}
  \label{t:Rdens}
\end{table}
Guaranteed optimal solutions in 1800s are most likely with edge densities around $0.6$-$0.7$.

In Table \ref{t:Rftp} (left), we present average results grouped by formulation type. The corresponding bar plot figure is shown in Table \ref{t:Rftp} (right).
\begin{table}[!ht]
  \begin{center}
    \begin{minipage}{0.48\textwidth}
      \begin{tabular}{l|rrr}
      \textbf{formulation} & \textsf{mde} & \textsf{lde} & CPU \\ \hline
      \textsf{cycle} & 0.0046 & 0.0197 & 806.88 \\
      \textsf{cyclesimple} & 0.1074 & 0.4928 & 1042.94 \\
      \textsf{cycpushpull} & 0.0007 & 0.0037 & 1256.00 \\
      \textsf{cycsimplepushpull} & 0.0013 & 0.0067 & 1451.29 \\
      \textsf{cycsimplesys1} & 0.0000 & 0.0002 & 276.27 \\
      \textsf{cycsimplesys2} & 0.0861 & 0.3885 & 930.46 \\
      \textsf{cycsys1} & 0.0000 & 0.0001 & 453.94 \\
      \textsf{cycsys2} & 0.0447 & 0.3032 & 771.03 \\
      \textsf{pullpush} & 2.2092 & 9.7720 & 668.97 \\
      \textsf{pushpull} & 0.0001 & 0.0002 & 1353.12 \\
      \textsf{quartic} & 0.2173 & 1.3483 & 808.64 \\
      \textsf{system1} & 0.0000 & 0.0002 & 490.75 \\
      \textsf{system2} & 0.2141 & 1.2033 & 840.52
      \end{tabular}
    \end{minipage}
    \begin{minipage}{0.5\textwidth}
      \includegraphics[width=\textwidth,height=0.3\textheight]{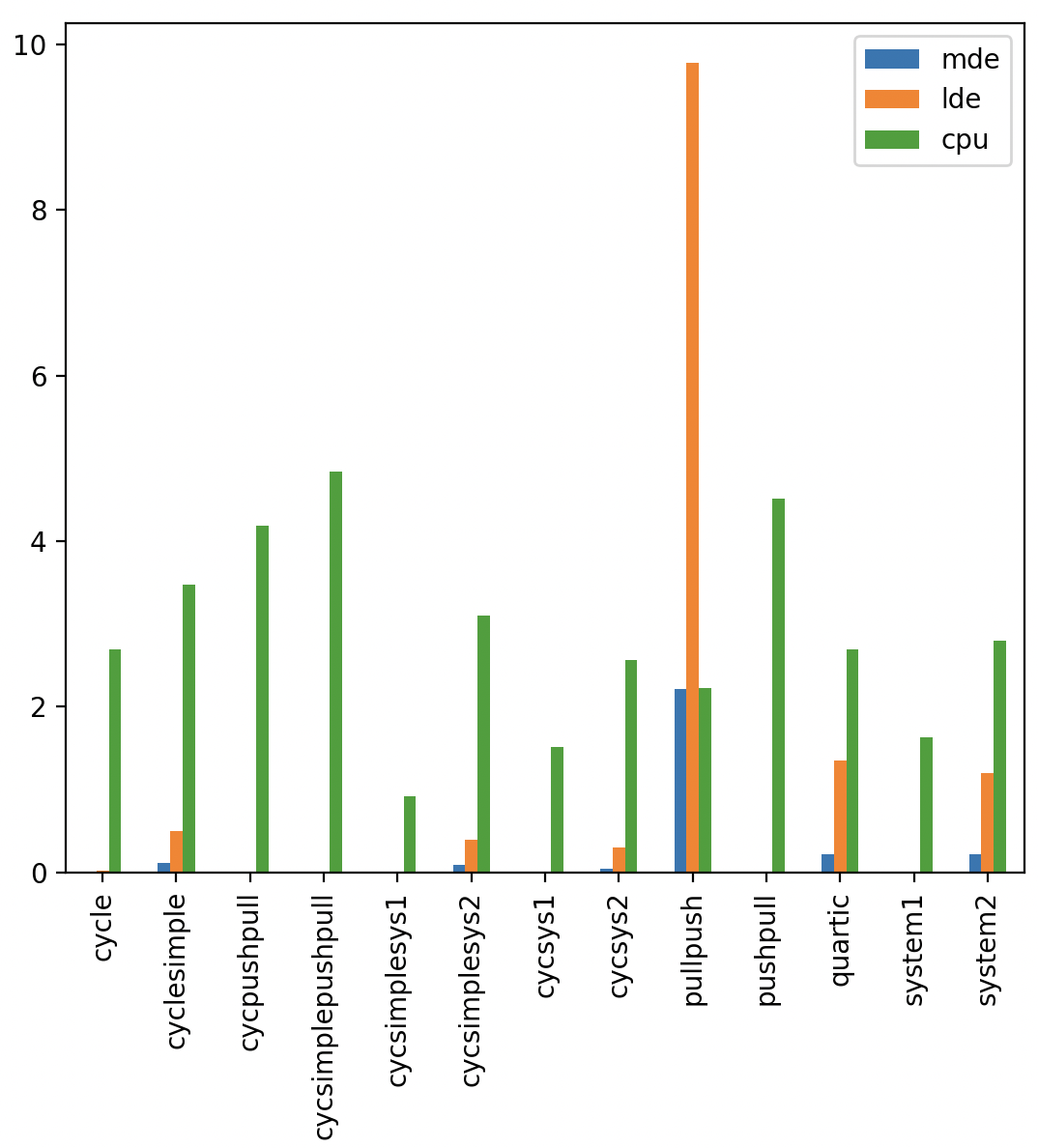}
    \end{minipage}  
  \end{center}
  \caption{Average results on formulation types for the graph class $\mathcal{R}$ (the CPU time column was scaled by 1/300).}
  \label{t:Rftp}
\end{table}
The formulations that were able to yield guaranteed optimal solutions in 1800s to precision $10^{-4}$ on average were: \textsf{cycsimplesys1}, \textsf{pushpull}, \textsf{system1}. The formulations \textsf{cycpushpull} and \textsf{cycsimplepushpull} went up to precision $10^{-3}$ on average. We note that \textsf{pullpush} was the worst-performing formulation by far (this is consistent with the observation below Eq.~\eqref{pullpush}).
\begin{table}[!ht]
  \begin{center}
    \begin{minipage}{0.22\textwidth}
      \begin{tabular}{l|r}
        $\approx |V|$ & CPU \\ \hline
        10 & 82.91 \\
        20 & 276.23 \\
        \multicolumn{2}{c}{} \\
        \textbf{$\approx |E|$} & CPU \\ \hline
        50 & 108.30 \\
        100 & 181.33 \\
        150 & 438.50 \\
        200 & 1400.09 \\
        \multicolumn{2}{c}{}  \\
        \textbf{$\approx$ dens.} & CPU \\ \hline
        0.4 & 22.96 \\
        0.5 & 322.11 \\
        0.6 & 33.85 \\
        0.7 & 197.63 \\
        0.8 & 337.40 \\
        0.9 & 188.36 \\
        1.0 & 58.97 
      \end{tabular}      
    \end{minipage}
    \hspace*{-1.4em}\begin{minipage}{0.25\textwidth}
      \includegraphics[width=\textwidth,height=0.07\textheight]{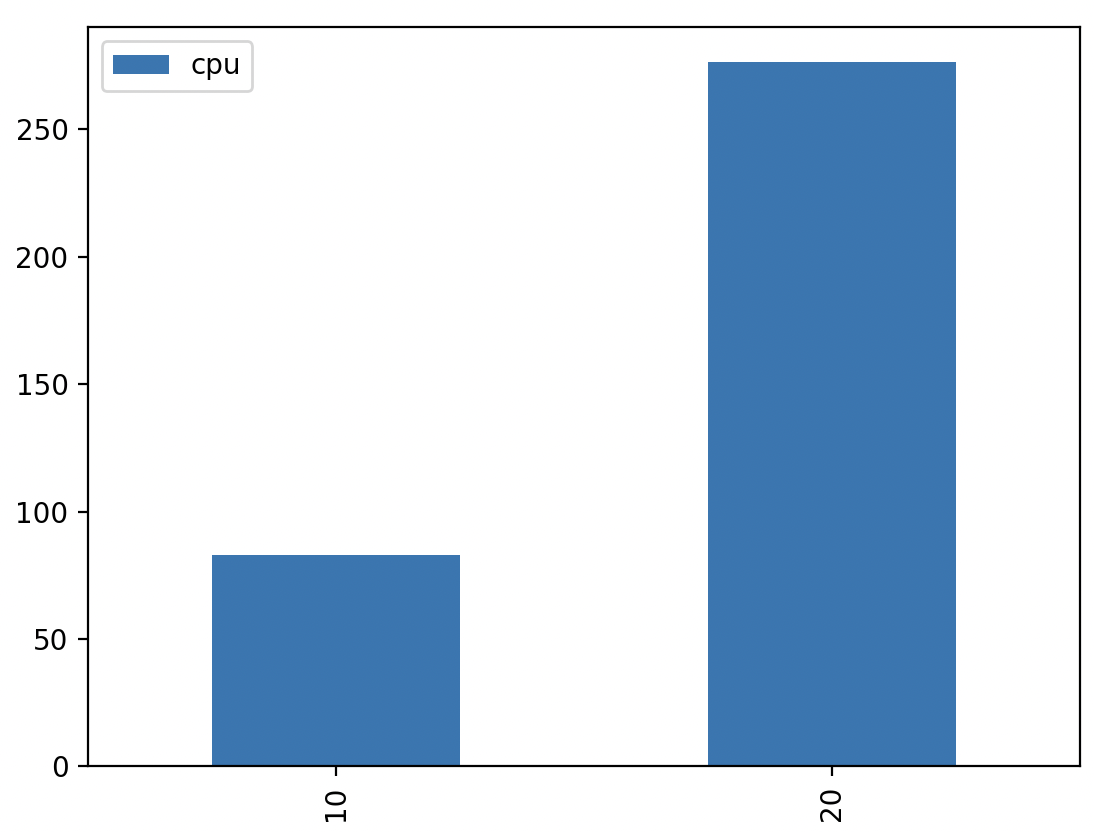}
      \includegraphics[width=\textwidth,height=0.13\textheight]{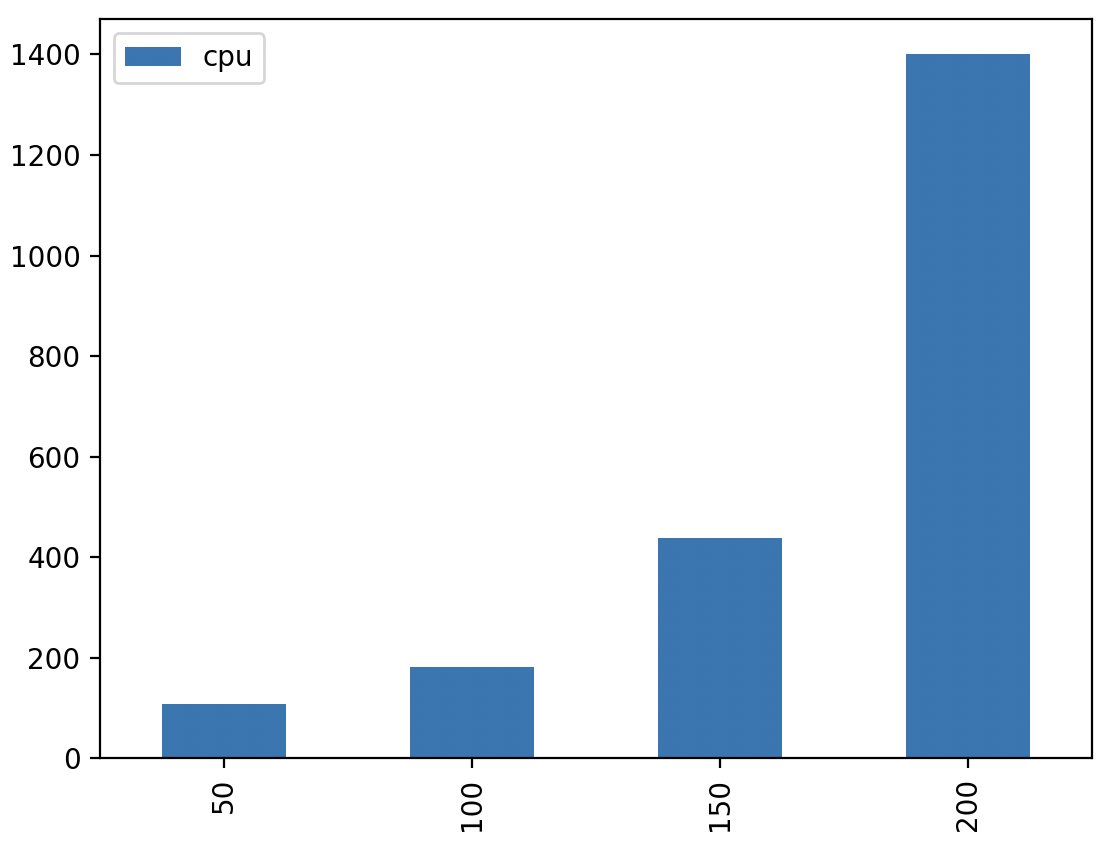}
      \includegraphics[width=\textwidth,height=0.13\textheight]{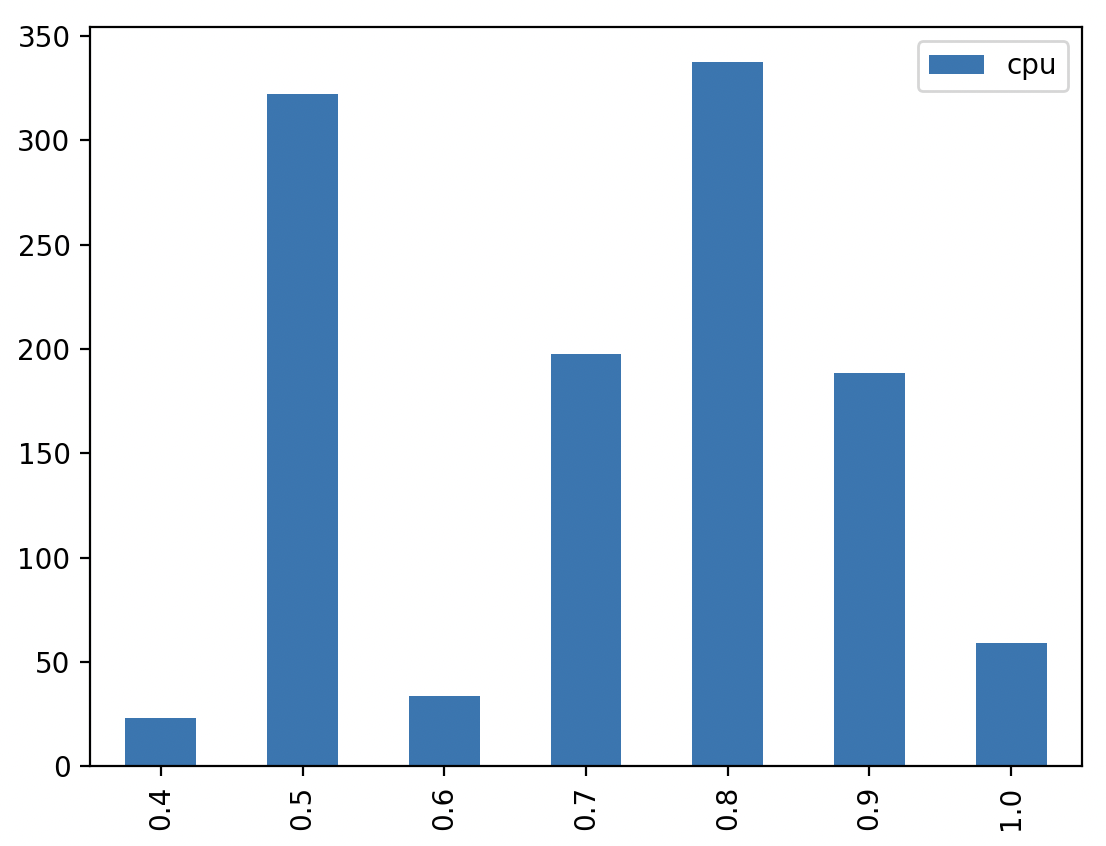}
    \end{minipage}    
    \begin{minipage}{0.25\textwidth}
      \begin{tabular}{l|r}
        \textbf{formulation} & CPU \\ \hline
        \textsf{cycle} & 163.99 \\
        \textsf{cyclesimple} & 30.93 \\
        \textsf{cycpushpull} & 398.56 \\
        \textsf{cycsimple-} & \\ \hspace*{1em}\textsf{pushpull} & 369.86 \\
        \textsf{cycsimplesys1} & 92.64 \\
        \textsf{cycsimplesys2} & 59.92 \\
        \textsf{cycsys1} & 83.10 \\
        \textsf{cycsys2} & 199.12 \\
        \textsf{pullpush} & 137.02 \\
        \textsf{pushpull} & 403.17 \\
        \textsf{quartic} & 64.25 \\
        \textsf{system1} & 53.70 \\
        \textsf{system2} & 120.11
      \end{tabular}
    \end{minipage}
    \begin{minipage}{0.25\textwidth}
      \hspace*{-1em}\includegraphics[width=1.2\textwidth,height=0.22\textheight]{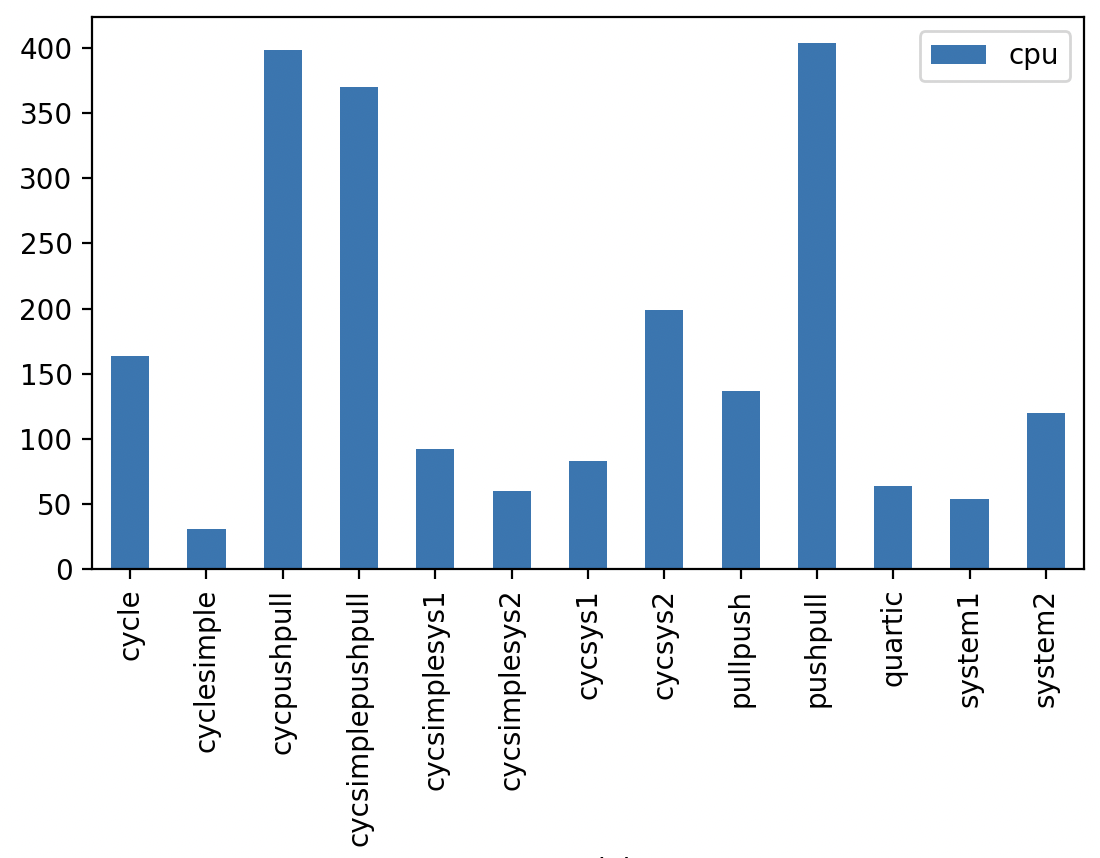}
    \end{minipage}  
  \end{center}
  \caption{Average CPU time for naturally terminating (instance,formulation) pairs in $\mathcal{R}$.}
  \label{t:R_cpu}
\end{table}

We now restrict the analysis to the cases where the global nonconvex NLP solver reached termination within the allotted time limit (1800s), which happened in 206 (instance, formulation) pairs. The average \textsf{mde} for these cases is $0.66\times 10^{-6}$, the average \textsf{lde} is $0.0002$, the average CPU time is $\approx 140$s. All of these pairs have \textsf{mde} and \textsf{lde} between $0$ and $0.0006$, which we do not report, as they are essentially zero, and indicate a correct realization. The average CPU times grouped by vertex/edge cardinality, graph density, and formulation type are given in Table \ref{t:R_cpu}. While the picture does not change too much for $|V|$, $|E|$, and edge density, the best performing formulations are different in this test with respect to the previous test: the best performing formulation was \textsf{cyclesimple}, followed by \textsf{system1}, \textsf{quartic}, and \textsf{cycsimplesys2}.

\subsubsection{The 309-graph collection $\mathcal{G}$}
\label{s:dgp:gph}
We use this graph collection (Sect.~\ref{s:inst:gph}) to establish whether it is preferable to solve DGP instances with local NLP solvers on non-matrix DGP formulations within a MS algorithm, or to use the matrix formulation solution process. The test is such that we need only group results by formulation type.

In Table \ref{t:Gftp}, we present average results grouped by formulation type. The corresponding bar plot figure is shown in Fig.~\ref{f:Gftp}.
\begin{figure}[!ht]
  \begin{center}
    \includegraphics[width=\textwidth]{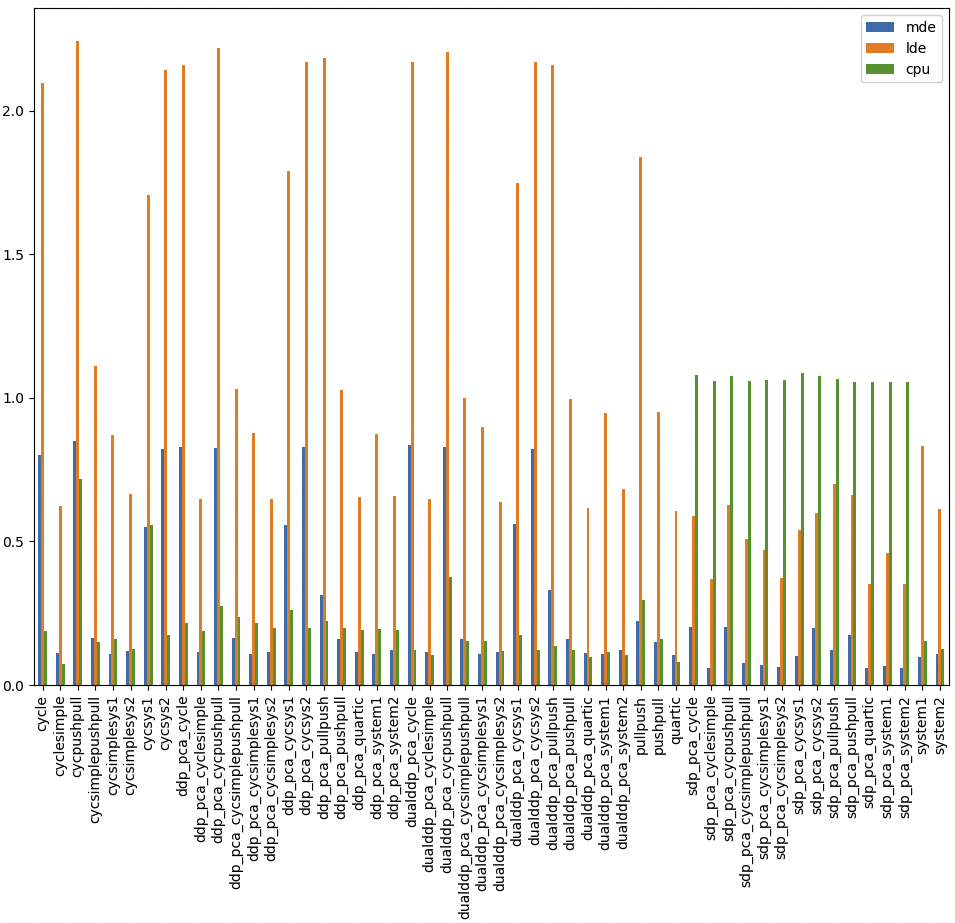}
  \end{center}
  \caption{The bar plot for the graph class $\mathcal{G}$ corresponding to Table \ref{t:Gftp}.}
  \label{f:Gftp}
\end{figure}
The best performing formulations in terms of solution quality are \textsf{sdp\_pca\_}$\chi$ with \[\chi\in\{\mathsf{quartic}, \mathsf{cyclesimple}, \mathsf{system2}\};\] in terms of CPU time we have \textsf{cyclesimple}, \textsf{quartic}, \textsf{system2}, and \textsf{dualddp\_pca\_}$\chi$ with $\chi$ as above. An important remark is that the best formulations in terms of CPU time are also very good in terms of solution quality. Among these, six are matrix formulations of SDP and dualDDP types, and three are non-matrix.
\begin{table}[!ht]
  \begin{center}
    \begin{tabular}{l|rrr}
      \textbf{formulation} & \textsf{mde} & \textsf{lde} & CPU \\ \hline
      \textsf{cycle} & 0.8001 & 2.0958 & 0.19 \\ [-0.2em]
      \textsf{cyclesimple} & 0.1128 & 0.6248 & 0.07 \\ [-0.2em]
      \textsf{cycpushpull} & 0.8492 & 2.2438 & 0.72 \\ [-0.2em]
      \textsf{cycsimplepushpull} & 0.1635 & 1.1107 & 0.15 \\ [-0.2em]
      \textsf{cycsimplesys1} & 0.1088 & 0.8713 & 0.16 \\ [-0.2em]
      \textsf{cycsimplesys2} & 0.1190 & 0.6669 & 0.13 \\ [-0.2em]
      \textsf{cycsys1} & 0.5517 & 1.7059 & 0.56 \\ [-0.2em]
      \textsf{cycsys2} & 0.8206 & 2.1401 & 0.18 \\ [-0.2em]
      \textsf{ddp\_pca\_cycle} & 0.8278 & 2.1597 & 0.22 \\ [-0.2em]
      \textsf{ddp\_pca\_cyclesimple} & 0.1157 & 0.6470 & 0.19 \\ [-0.2em]
      \textsf{ddp\_pca\_cycpushpull} & 0.8269 & 2.2188 & 0.28 \\ [-0.2em]
      \textsf{ddp\_pca\_cycsimplepushpull} & 0.1633 & 1.0303 & 0.24 \\ [-0.2em]
      \textsf{ddp\_pca\_cycsimplesys1} & 0.1090 & 0.8762 & 0.21 \\ [-0.2em]
      \textsf{ddp\_pca\_cycsimplesys2} & 0.1154 & 0.6478 & 0.20 \\ [-0.2em]
      \textsf{ddp\_pca\_cycsys1} & 0.5578 & 1.7894 & 0.26 \\ [-0.2em]
      \textsf{ddp\_pca\_cycsys2} & 0.8273 & 2.1707 & 0.20 \\ [-0.2em]
      \textsf{ddp\_pca\_pullpush} & 0.3141 & 2.1826 & 0.22 \\ [-0.2em]
      \textsf{ddp\_pca\_pushpull} & 0.1621 & 1.0275 & 0.20 \\ [-0.2em]
      \textsf{ddp\_pca\_quartic} & 0.1143 & 0.6536 & 0.19 \\ [-0.2em]
      \textsf{ddp\_pca\_system1} & 0.1084 & 0.8753 & 0.20 \\ [-0.2em]
      \textsf{ddp\_pca\_system2} & 0.1207 & 0.6601 & 0.19 \\ [-0.2em]
      \textsf{dualddp\_pca\_cycle} & 0.8348 & 2.1682 & 0.12 \\ [-0.2em]
      \textsf{dualddp\_pca\_cyclesimple} & 0.1141 & 0.6480 & 0.10 \\ [-0.2em]
      \textsf{dualddp\_pca\_cycpushpull} & 0.8297 & 2.2055 & 0.38 \\ [-0.2em]
      \textsf{dualddp\_pca\_cycsimplepushpull} &0.1614&1.0001 & 0.15 \\ [-0.2em]
      \textsf{dualddp\_pca\_cycsimplesys1} & 0.1083 & 0.8970 & 0.15 \\ [-0.2em]
      \textsf{dualddp\_pca\_cycsimplesys2} & 0.1136 & 0.6361 & 0.12 \\ [-0.2em]
      \textsf{dualddp\_pca\_cycsys1} & 0.5596 & 1.7468 & 0.17 \\ [-0.2em]
      \textsf{dualddp\_pca\_cycsys2} & 0.8235 & 2.1685 & 0.12 \\ [-0.2em]
      \textsf{dualddp\_pca\_pullpush} & 0.3314 & 2.1606 & 0.14 \\ [-0.2em]
      \textsf{dualddp\_pca\_pushpull} & 0.1616 & 0.9969 & 0.12 \\ [-0.2em]
      \textsf{dualddp\_pca\_quartic} & 0.1104 & 0.6177 & 0.10 \\ [-0.2em]
      \textsf{dualddp\_pca\_system1} & 0.1083 & 0.9467 & 0.12 \\ [-0.2em]
      \textsf{dualddp\_pca\_system2} & 0.1206 & 0.6822 & 0.10 \\ [-0.2em]
      \textsf{pullpush} & 0.2235 & 1.8404 & 0.29 \\ [-0.2em]
      \textsf{pushpull} & 0.1516 & 0.9501 & 0.16 \\ [-0.2em]
      \textsf{quartic} & 0.1042 & 0.6053 & 0.08 \\ [-0.2em]
      \textsf{sdp\_pca\_cycle} & 0.2006 & 0.5890 & 1.08 \\ [-0.2em]
      \textsf{sdp\_pca\_cyclesimple} & 0.0607 & 0.3683 & 1.06 \\ [-0.2em]
      \textsf{sdp\_pca\_cycpushpull} & 0.2027 & 0.6268 & 1.08 \\ [-0.2em]
      \textsf{sdp\_pca\_cycsimplepushpull} & 0.0773 & 0.5090 & 1.06 \\ [-0.2em]
      \textsf{sdp\_pca\_cycsimplesys1} & 0.0705 & 0.4720 & 1.06 \\ [-0.2em]
      \textsf{sdp\_pca\_cycsimplesys2} & 0.0625 & 0.3718 & 1.06 \\ [-0.2em]
      \textsf{sdp\_pca\_cycsys1} & 0.1021 & 0.5414 & 1.09 \\ [-0.2em]
      \textsf{sdp\_pca\_cycsys2} & 0.2004 & 0.6002 & 1.08 \\ [-0.2em]
      \textsf{sdp\_pca\_pullpush} & 0.1215 & 0.6993 & 1.07 \\ [-0.2em]
      \textsf{sdp\_pca\_pushpull} & 0.1744 & 0.6633 & 1.06 \\ [-0.2em]
      \textsf{sdp\_pca\_quartic} & 0.0593 & 0.3524 & 1.06 \\ [-0.2em]
      \textsf{sdp\_pca\_system1} & 0.0663 & 0.4607 & 1.06 \\ [-0.2em]
      \textsf{sdp\_pca\_system2} & 0.0581 & 0.3515 & 1.06 \\ [-0.2em]
      \textsf{system1} & 0.0986 & 0.8314 & 0.15 \\ [-0.2em]
      \textsf{system2} & 0.1077 & 0.6118 & 0.13 
      \end{tabular}
  \end{center}
  \caption{Average results on formulation types for the graph class $\mathcal{G}$.}
  \label{t:Gftp}
\end{table}
Our answer to the second question in relation to DGP is therefore that the matrix formulation process is generally better, but the refinement step is crucially important.

Since the instance family $\mathscr{G}$ contains many graph types, we also present average results grouped by formulation type in Table \ref{t:Ggtp} (left), and the corresponding bar plot figure in Table \ref{t:Ggtp} (right).
\begin{table}[!ht]
  \begin{center}
    \begin{minipage}{0.40\textwidth}
      \begin{tabular}{l|rrr}
        \textbf{graph type} & \textsf{mde} & \textsf{lde} & CPU \\ \hline
        Walmostreg & 0.0916 & 0.4283 & 0.15 \\
        Wbipartite & 0.0008 & 0.0460 & 1.31 \\
        Wcliquechain & 0.2607 & 0.9838 & 0.15 \\
        Wcluster & 0.2021 & 0.7913 & 0.19 \\
        Wdmdgp & 0.2680 & 1.0036 & 0.15 \\
        Wmesh & 0.0106 & 0.1382 & 0.13 \\
        Wpowerlaw & 0.0571 & 0.4401 & 0.11 \\
        Wrandom & 0.0755 & 0.4857 & 0.14 \\
        Wtorus & 0.0501 & 0.3761 & 0.10 \\
        Wtriangle & 0.1551 & 0.6157 & 0.06 \\
        Wtrichain & 0.0363 & 0.1330 & 0.04 \\
        Wtripartite & 0.0086 & 0.1551 & 4.23 \\
        almostreg & 0.0410 & 0.2685 & 0.14 \\
        beeker\_glusa & 0.0000 & 0.0000 & 0.04 \\
        bipartite & 0.0000 & 0.0000 & 0.88 \\
        cliquechain & 0.2419 & 1.0672 & 0.17 \\
        cluster & 0.1713 & 0.7952 & 0.18 \\
        dmdgp & 0.2512 & 1.0721 & 0.17 \\
        local & 0.1436 & 0.6596 & 0.26 \\
        mesh & 0.0000 & 0.0000 & 0.09 \\
        norm & 3.3560 & 12.0053 & 0.57 \\
        powerlaw & 0.0086 & 0.2692 & 0.18 \\
        random & 0.0464 & 0.2732 & 0.13 \\
        torus & 0.0070 & 0.0659 & 0.08 \\
        triangle & 0.0853 & 0.4461 & 0.06 \\
        trichain & 0.0034 & 0.0269 & 0.04 \\
        tripartite & 0.0000 & 0.0000 & 2.10 
      \end{tabular}
    \end{minipage}
    \begin{minipage}{0.58\textwidth}
      \includegraphics[width=\textwidth,height=0.5\textheight]{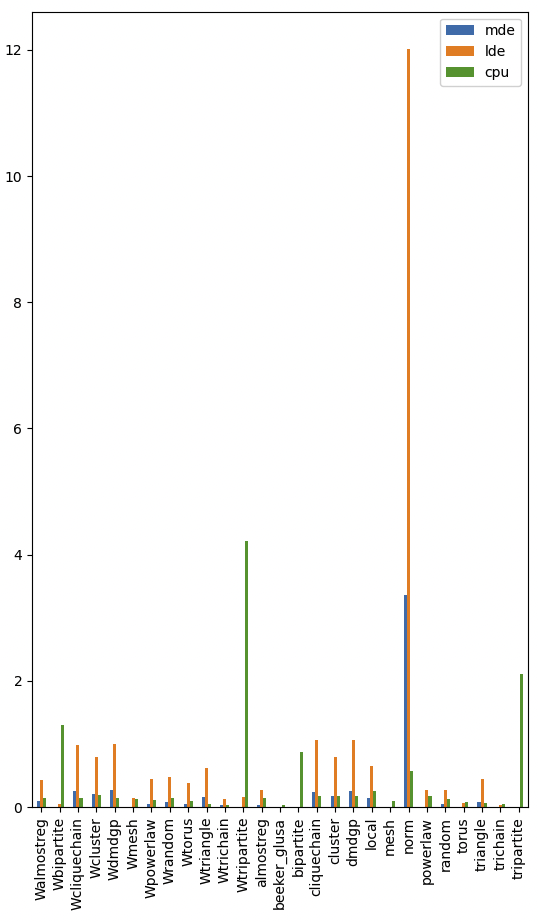}
    \end{minipage}  
  \end{center}
  \caption{Average results on graph types and the corresponding bar plot for the graph class $\mathcal{G}$.}
  \label{t:Ggtp}
\end{table}

\subsubsection{The protein graph collection $\mathcal{P}$}
\label{s:dgp:prot}
We attempt to reconstruct the shape of proteins from a partial set of inter-atomic distances with their adjacencies (Sect.~\ref{s:inst:prot}) by using the methods that appear more promising from earlier DGP experiments (Sect.~\ref{s:dgp:euclgph}-\ref{s:dgp:gph}), namely: the non-matrix formulations \textsf{cyclesimple}, \textsf{quartic}, \textsf{system2} and the matrix formulations based on SDP and dual DDP followed by PCA, with refinement step carried out using the non-matrix formulations above. Since SDP solvers fail with larger instances (from \textsf{3al} upwards, see Table \ref{s:inst:prot}), we replaced SDP by DDP in such cases. For non-matrix formulations we used IPOPT within the MS algorithm (Alg.~\ref{a:ms}) with a 10 iterations limit.

We first present results about the best formulation per protein instance. By ``best'' we mean the best trade-off between \textsf{mde} and \textsf{lde}. When no solution dominates the others over both measures the choice was made by the authors of this survey, based on their past experience. 
\begin{table}[!ht]
  \begin{center}
    \begin{tabular}{l|rrrr}
      \textbf{instance} & formulation & \textsf{mde} & \textsf{mde} & CPU \\ \hline \hline
      \textsf{tiny} & \textsf{dualddp\_cyclesimple} & 0.00 & 0.00 & 0.19 \\
      \textsf{1guu-1} & \textsf{dualddp\_system2} & 0.06 & 1.05 & 0.83 \\
      \textsf{1guu-400} & \textsf{system2} & 0.08 & 0.97 & 2.70 \\
      \textsf{C0030pkl} & \textsf{dualddp\_cyclesimple} & 0.04 & 1.32 & 5.68 \\
      \textsf{1PPT} & \textsf{quartic} & 0.26 & 2.58 & 8.48 \\
      \textsf{1guu} & \textsf{sdp\_system2} & 0.05 & 0.87 & 8476.58 \\
      \textsf{100d} & \textsf{cyclesimple} & 0.33 & 3.06 & 13.50 \\ \hline
      \textsf{3al1} & \textsf{system2} & 0.04 & 2.28 & 498.07 \\
      \textsf{1hpv} & \textsf{ddp\_quartic} & 0.40 & 3.62 & 214.49 \\
      \textsf{il2} & \textsf{ddp\_quartic} & 0.03 & 4.36 & 1262.47 \\
      \textsf{1tii} & \textsf{dualddp\_quartic} & 0.43 & 4.10 & 2928.25 
    \end{tabular}
  \end{center}
  \caption{Best formulations per instance of graph class $\mathcal{P}$. The SDP formulation was replaced by the DDP formulation in the lower half due to excessive size.}
  \label{t:dgp:prot:best}
\end{table}

The results grouped by formulation type (independently of the instance) are given in Table \ref{t:dgp:prot:form}. 
\begin{table}[!ht]
  \begin{center}
    \begin{tabular}{lr}
      \begin{minipage}{0.55\textwidth}
        \begin{tabular}{l|rrr}
          \textbf{formulation} & \textsf{mde} & \textsf{lde} & CPU \\ \hline
          \textsf{cyclesimple} & 0.2343 & 2.9293 & 532.43 \\
          \textsf{quartic} & 0.1710 & 2.5200 & 2795.65 \\
          \textsf{system2} & 0.1518 & 2.4399 & 6512.35 \\
          \textsf{sdp\_cyclesimple} & 0.1777 & 2.1073 & 2593.07 \\
          \textsf{sdp\_quartic} & 0.2339 & 2.1764 & 2738.82 \\
          \textsf{sdp\_system2} & 0.1740 & 2.1560 & 2958.58 \\
          \textsf{ddp\_cyclesimple} & 0.2859 & 3.9973 & 1280.88 \\
          \textsf{ddp\_quartic} & 0.2460 & 4.1536 & 1024.87 \\
          \textsf{ddp\_system2} & 0.2800 & 4.2591 & 2036.27 \\
          \textsf{dualddp\_cyclesimple} & 0.2013 & 2.5855 & 437.30 \\
          \textsf{dualddp\_quartic} & 0.2273 & 2.7961 & 370.53 \\
          \textsf{dualddp\_system2} & 0.2334 & 2.7907 & 634.17         
        \end{tabular} 
      \end{minipage} & 
      \begin{minipage}{0.4\textwidth}
        \hspace*{-2em}\includegraphics[width=1.2\textwidth]{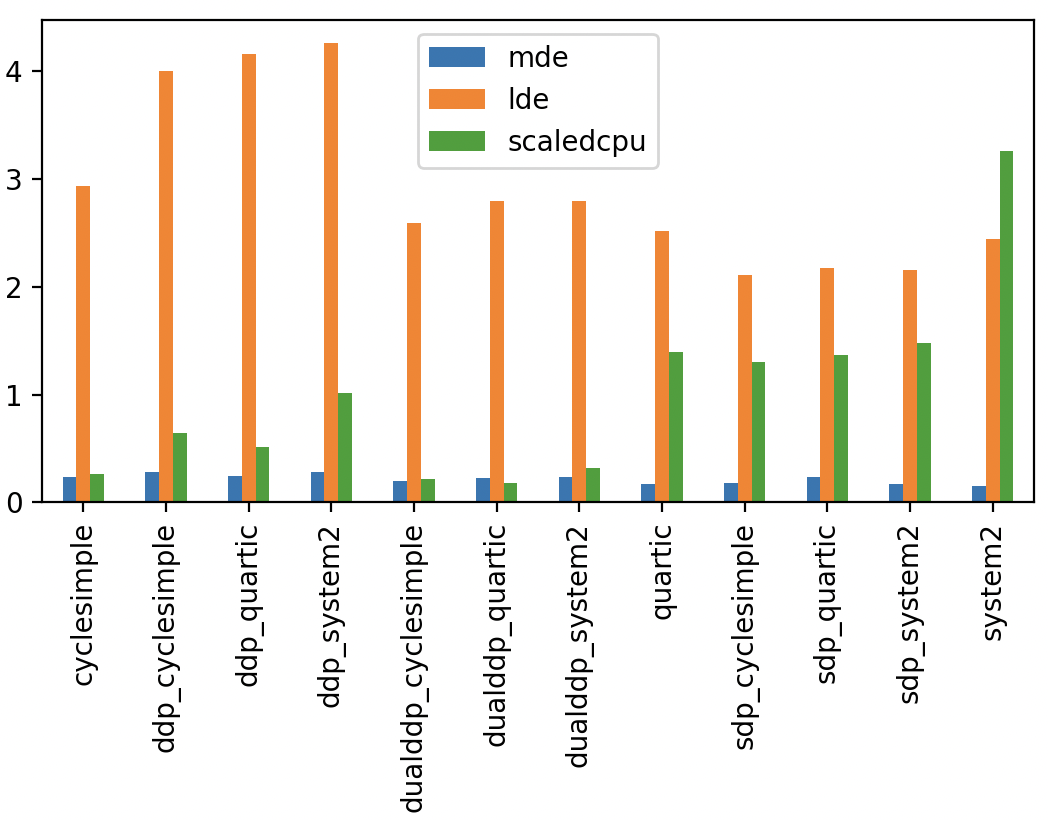}
      \end{minipage}
    \end{tabular}
  \end{center}
  \caption{Average results (over all instances) on formulation types for the graph class $\mathcal{P}$.}
  \label{t:dgp:prot:form}
\end{table}
The best formulation as regards solution quality (\textsf{mde}, \textsf{lde}) is \textsf{system2}, which is also the worst for CPU time (but CPU time is not a crucial measure for the purpose of finding the structure of proteins). The formulation \textsf{quartic} is the closest competitor. We note that, however, the variance of \textsf{mde}, \textsf{lde} over all of the tested formulations is small. The fastest formulation is \textsf{dualddp\_pca\_quartic}, with \textsf{ddp\_pca\_quartic} the closest competitor. The CPU time variance is small for all formulations but \textsf{quartic} and \textsf{system2}, which took considerably longer to solve. The best trade-off overall between quality and CPU time is \textsf{ddp\_pca\_system2}. 

\subsection{UDGP tests and results}
\label{s:udgpres}
\label{s:res:uforms}

We have tested the following exact formulations on the set of UDGP instances obtained as explained in Sect.~\ref{s:inst:udgp}:
\begin{enumerate}
   \ifspringer\else\setlength{\parskip}{-0.2em}\fi
\item \textsf{ucycsimplesys1} (Eq.~\eqref{ucycsystem_ell1}),
\item \textsf{upushpull} (Eq.~\eqref{upushpull}),
\item \textsf{uquartic} (Eq.~\eqref{uquartic}),
\item \textsf{uquarticcont} (Eq.~\eqref{uquartic_cont}),
\item \textsf{usystem1} (Eq.~\eqref{usystem_ell1}),
\item \textsf{usystem2} (Eq.~\eqref{usystem}),
\end{enumerate}
and the approximate matrix reformulations:
\begin{enumerate}
   \ifspringer\else\setlength{\parskip}{-0.2em}\fi
\item \textsf{misdp\_usystem1} (Eq.~\eqref{Xusystem_ell1} with $\mathbf{X}=\mathbf{S}_n^+$),
\item \textsf{middp\_usystem1} (Eq.~\eqref{Xusystem_ell1} with $\mathbf{X}=\mathbf{D}_n$),
\item \textsf{midualddp\_usystem1} (Eq.~\eqref{Xusystem_ell1} with $\mathbf{X}=\mathbf{D}^\ast_n$).
\end{enumerate}

We present our computational results grouped in the same way as for the DGP results (Sect.~\ref{s:dgpres}). The quality of the realization is measured by \textsf{mde} and \textsf{lde}, as in Sect.~\ref{s:dgpres}.

While there is no intrinsic measure of the assignment $\alpha$, all our UDGP instances are generated from a DGP one (by losing the graph and keeping the distance values), so we can evaluate the difference between the graph $G_\alpha$ reconstructed from $\alpha$ and the original graph $G$ that gave rise to the UDGP instance. This measure, called \textsf{gphsim}, is based on an evaluation of label-independent topological similarity of graphs if $G,G_\alpha$ have the same number of nodes, and on a comparison of the (zero-padded) spectra of the Laplacian matrices \cite{cvetkovic} of $G,G_\alpha$ otherwise (we let $\mathsf{spectrumLaplacian}(G)$ be the vector of eigenvalues of the Laplacian matrix of $G$); it makes use of the normalized adjacency matrix $\nadj{G}$ of a graph $G$, which is its adjacency matrix scaled by its matrix norm. The measure \textsf{gphsim} is computed as in Alg.~\ref{alg:gphsim}. It has values in $[-1,1]$, with $\mathtt{gphsim}=1$ if $G,G_\alpha$ are isomorphic graphs.
\begin{algorithm}[!ht]
 \begin{algorithmic}[1]
  \IF{$|V(G)|=|V(G_\alpha)|$}
    \STATE $\mathtt{gphsim} = 0$
    \IF{$G,G_\alpha$ have matching degree sequences}
      \STATE $\mathtt{gphsim} \leftarrow \mathtt{gphsim} + 1$
      \IF{$G,G_\alpha$ have matching triangle sequences}
        \STATE $\mathtt{gphsim} \leftarrow \mathtt{gphsim} + 1$
        \IF{$G,G_\alpha$ have matching clique sequences}
          \STATE $\mathtt{gphsim} \leftarrow \mathtt{gphsim} + 1$
          \IF{$G,G_\alpha$ are isomorphic}
            \STATE $\mathtt{gphsim} \leftarrow \mathtt{gphsim} + 1$
          \ENDIF
        \ENDIF
      \ENDIF        
    \ENDIF
    \STATE $\mathtt{gphsim} \leftarrow \mathtt{gphsim}/4$
    \IF{$\mathtt{gphsim} < 1$}
      \STATE $\mathtt{gphsim}\leftarrow\frac{1}{2}\mathtt{gphsim} + \frac{1}{2}\trace{\nadj{G}\,\nadj{H}}$
    \ENDIF
  \ELSE
    \STATE \gray{\# assume $|V(G)|<|V(G_\alpha)|$ without loss of generality}
    \STATE $s_G=\mathsf{spectrumLaplacian}(G)$ \gray{\quad \# eigenvalues in decreasing order}
    \STATE pad $s_G$ with $V(G_\alpha)-V(G)$ tailing zeros 
    \STATE $s_{G_\alpha}=\mathsf{spectrumLaplacian}(G_\alpha)$
    \STATE normalize $s_G,s_H$
    \STATE $\mathtt{gphsim} = \langle s_G,s_{G_\alpha}\rangle$
  \ENDIF
 \end{algorithmic}
 \caption{The graph similarity measure $\mathtt{gphsim}$}
 \label{alg:gphsim}
\end{algorithm}

Accordingly, for UDGP results we report \textsf{mde}, \textsf{lde} and \textsf{gphsim}. Only the first two measures attest to the success of the solution algorithm (when close to zero), while a \textsf{gphsim} measure significantly lower than $1$ might simply attest to many different graphs compatible with the given distance values.

\subsubsection{The Euclidean graph collection $\mathcal{R}$}
\label{s:udgp:euclgph}
As mentioned in Sect.~\ref{s:dgp:euclgph} for DGP instances, we use this benchmark to establish to what size we can hope to solve UDGP instances to guaranteed optimality. Tolerance and time are as in the DGP case ($10^{-6}$ and 1800s, see Sect.~\ref{s:dgp:euclgph}). As in the DGP case, given that these tests aim at establishing the size for which we may be able to solve such problems to optimality, we only tested the exact formulations \textsf{ucycsimplesys1}, \textsf{upushpull}, \textsf{uquartic}, \textsf{uquarticcont}, \textsf{usystem1}, \textsf{usystem2}.

\begin{table}[!ht]
  \begin{center}
    \begin{tabular}{lr}
      \begin{minipage}{0.4\textwidth}
        \begin{tabular}{l|rrr}
          $\approx |V|$ & \textsf{mde} & \textsf{lde} & CPU \\ \hline
           5 & 0.1533 & 0.6015  & 1037.35 \\
           8 & 2.2117 & 5.6512  & 1059.07 \\
          10 & 1.6241 & 5.8019  & 1263.91 \\
          12 & 3.3561 & 9.8532  &  669.65 \\
          15 & 4.2756 & 13.3832 & 1000.90 \\
          18 & 4.1293 & 12.4107 & 1297.91 \\
          20 & 7.0485 & 19.2524 & 1443.99 
        \end{tabular} 
      \end{minipage} & 
      \begin{minipage}{0.58\textwidth}
        \includegraphics[width=\textwidth,height=0.2\textheight]{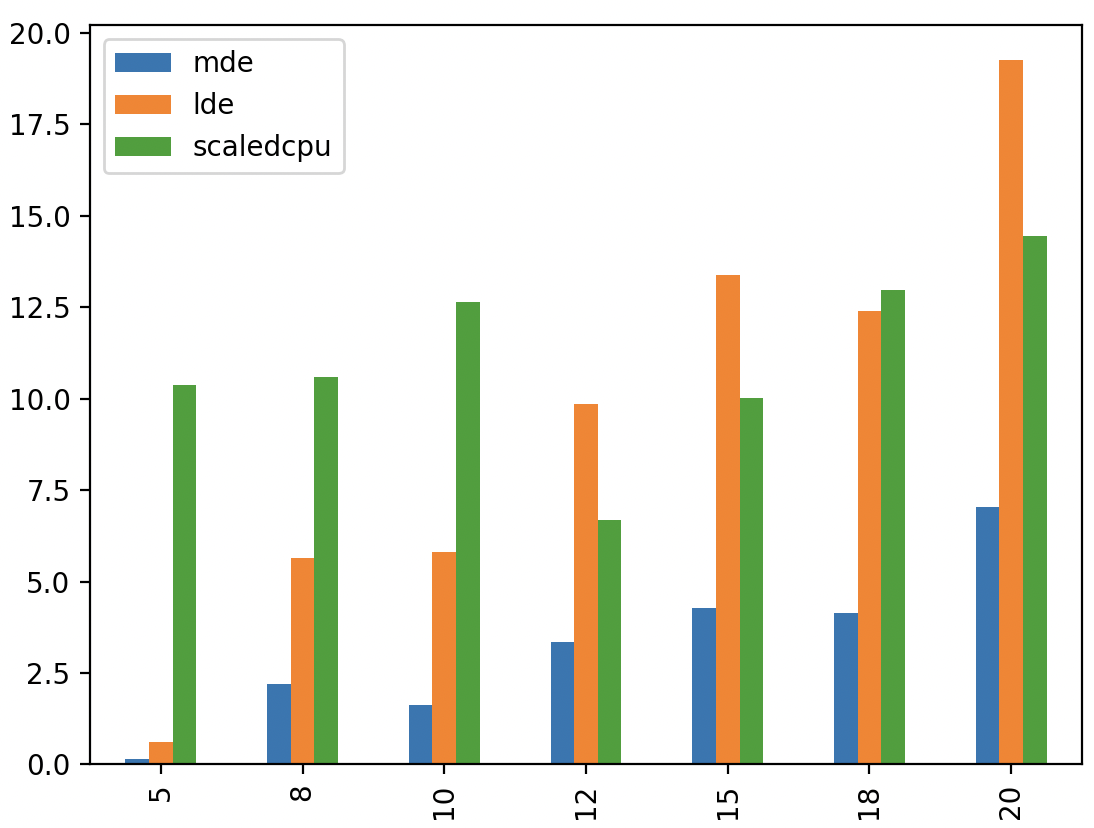}
      \end{minipage}
    \end{tabular}
  \end{center}
  \caption{Average results on $|V|$ for the graph class $\mathcal{R}$ on UDGP.}
  \label{t:udgp:Gvtx}
\end{table}

\begin{table}[!ht]
  \begin{center}
    \begin{tabular}{lr}
      \begin{minipage}{0.4\textwidth}
        \begin{tabular}{l|rrr}
          $\approx |E|$ & \textsf{mde} & \textsf{lde} & CPU \\ \hline
          50 & 1.6256 & 4.8607 & 1044.06 \\
          100 & 4.3829 & 13.6181 & 1000.95 \\
          150 & 4.5666 & 13.9273 & 1317.59 \\
          200 & 6.6131 & 16.7588 & 1406.54
        \end{tabular} 
      \end{minipage} & 
      \begin{minipage}{0.58\textwidth}
        \includegraphics[width=\textwidth,height=0.2\textheight]{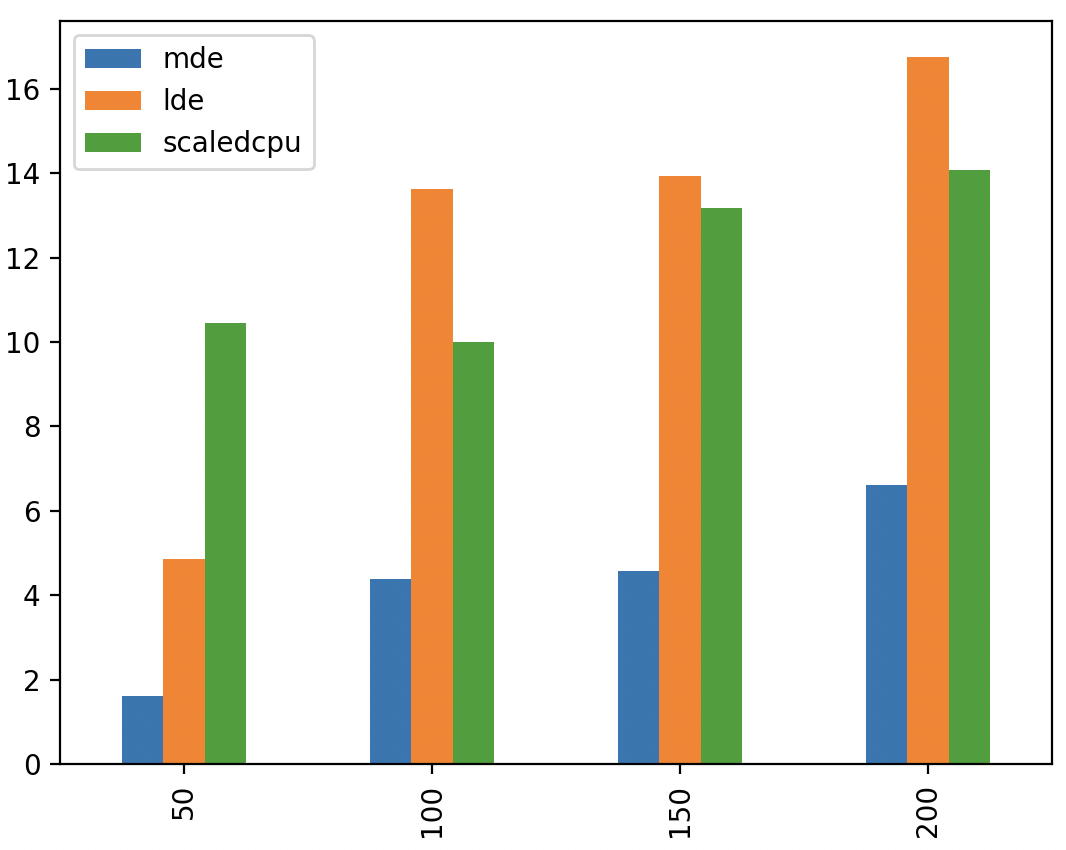}
      \end{minipage}
    \end{tabular}
  \end{center}
  \caption{Average results on $|E|$ for the graph class $\mathcal{R}$ on UDGP.}
  \label{t:udgp:Gedges}
\end{table}

\begin{table}[!ht]
  \begin{center}
    \begin{tabular}{lr}
      \begin{minipage}{0.4\textwidth}
        \begin{tabular}{l|rrr}
          $\approx$ \textbf{density} & \textsf{mde} & \textsf{lde} & CPU \\ \hline
          0.4 & 2.4903 & 8.0731 & 987.70 \\
          0.5 & 4.9176 & 15.1659 & 901.58 \\
          0.6 & 0.5023 & 2.1405 & 1088.01 \\
          0.7 & 4.3175 & 14.0233 & 790.11 \\
          0.8 & 3.4730 & 10.0472 & 1306.70 \\
          0.9 & 3.4126 & 9.6798 & 1231.16 \\
          1.0 & 2.8122 & 8.0143 & 1060.42
        \end{tabular} 
      \end{minipage} & 
      \begin{minipage}{0.58\textwidth}
        \includegraphics[width=\textwidth,height=0.2\textheight]{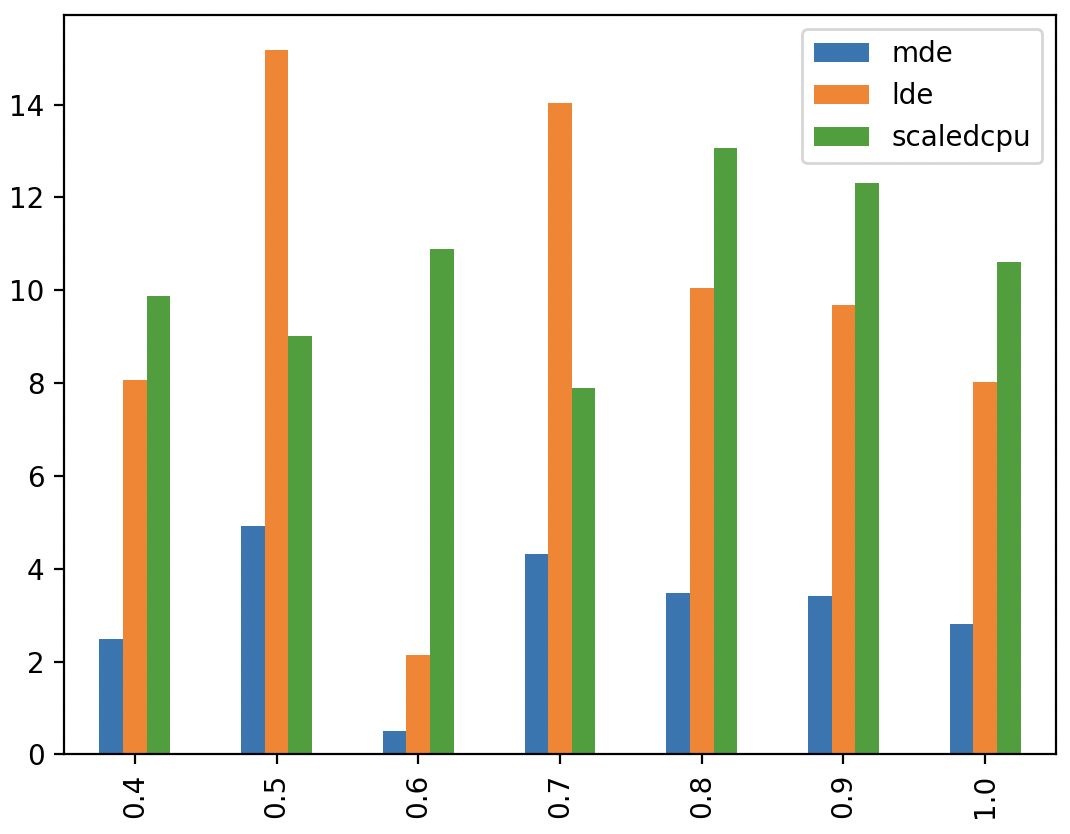}
      \end{minipage}
    \end{tabular}
  \end{center}
  \caption{Average results on edge density for the graph class $\mathcal{R}$ on UDGP.}
  \label{t:udgp:graphdens}
\end{table}

The results are shown in Tables \ref{t:udgp:Gvtx}-\ref{t:udgp:Gftp}. The only story they tell is that the realization errors are always large, even when the reconstructed graph $G_\alpha$ is isomorphic to the original graph $G$ (Table \ref{t:udgp:gphsim}).

\begin{table}[!ht]
  \begin{center}
    \begin{tabular}{lr}
      \begin{minipage}{0.4\textwidth}
        \begin{tabular}{l|rrr}
          $\approx$ \texttt{gphsim} & \textsf{mde} & \textsf{lde} & CPU \\ \hline
          0.1 & 0.4060 & 3.0809 & 1804.01 \\
          0.2 & 3.0043 & 9.5707 & 755.69 \\
          0.3 & 3.2969 & 10.5904 & 1105.91 \\
          0.4 & 4.2320 & 12.1688 & 1278.52 \\
          1.0 & 2.1472 & 5.7603 & 1045.87
        \end{tabular} 
      \end{minipage} & 
      \begin{minipage}{0.58\textwidth}
        \includegraphics[width=\textwidth,height=0.2\textheight]{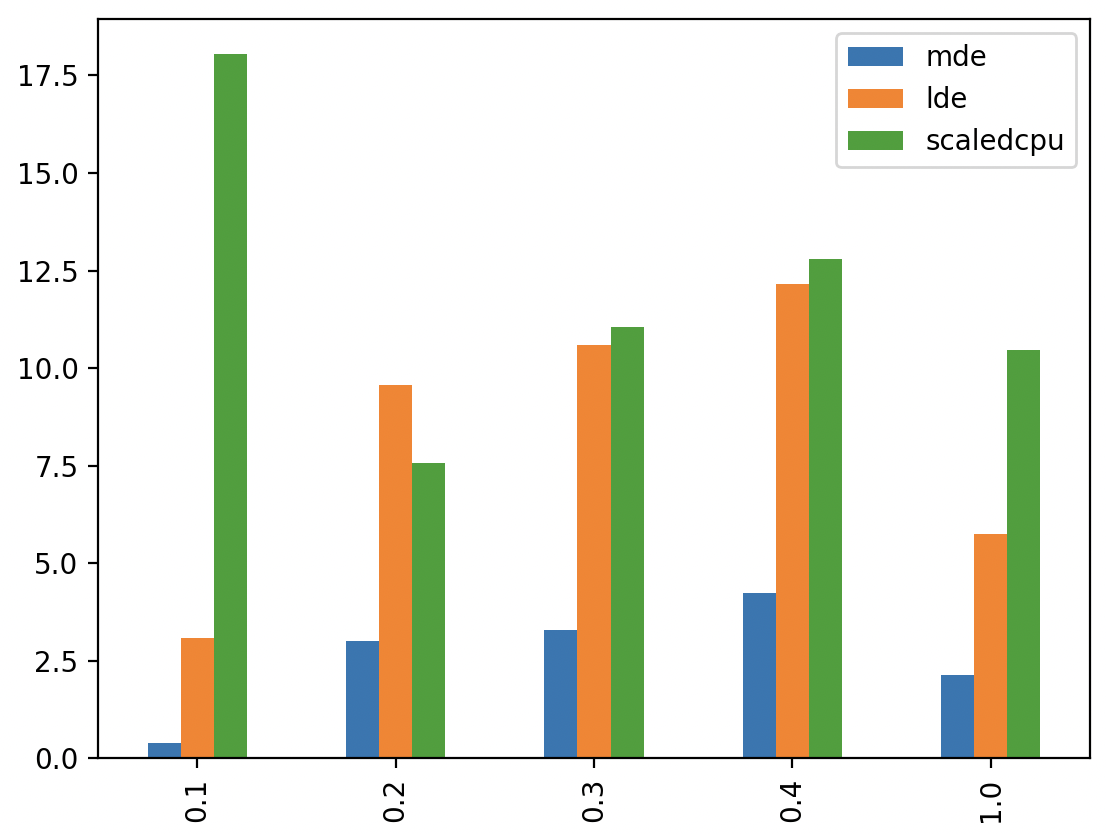}
      \end{minipage}
    \end{tabular}
  \end{center}
  \caption{Average results on graph similarity (\texttt{gphsim}) for the graph class $\mathcal{R}$ on UDGP.}
  \label{t:udgp:gphsim}
\end{table}

A more in-depth look at the results reveals that the global MINLP solver terminated naturally on 40 (instance, MINLP formulation) pairs out of the 152 pairs tested. Of these, only 6 yielded \textsf{mde} and \textsf{lde} measures within $O(10^{-3})$ (involving the smallest instance sizes and mainly the \textsf{upushpull} formulation) and only 3 within $O(10^{-5})$: (\textsf{euclid-5\_0.9},\textsf{upushpull}), (\textsf{euclid-8\_0.4},\textsf{upushpull}), (\textsf{euclid-8\_0.4},\textsf{system1}).

\begin{table}[!ht]
  \begin{center}
    \begin{tabular}{lr}
      \begin{minipage}{0.45\textwidth}
        \begin{tabular}{l|rrr}
          \textbf{formulation} & \textsf{mde} & \textsf{lde} & CPU \\ \hline
          \textsf{ucycsimplesys1} & 1.1592 & 5.0889 & 1620.06 \\
          \textsf{upushpull} & 4.7066 & 12.3713 & 993.11 \\
          \textsf{uquartic} & 5.5397 & 13.5655 & 1195.73 \\
          \textsf{uquarticcont} & 0.2852 & 1.4838 & 59.50 \\
          \textsf{usystem1} & 2.9467 & 11.0602 & 1368.15 \\
          \textsf{usystem2} & 4.3407 & 12.0299 & 1301.94
        \end{tabular} 
      \end{minipage} & 
      \begin{minipage}{0.54\textwidth}
        \includegraphics[width=\textwidth,height=0.2\textheight]{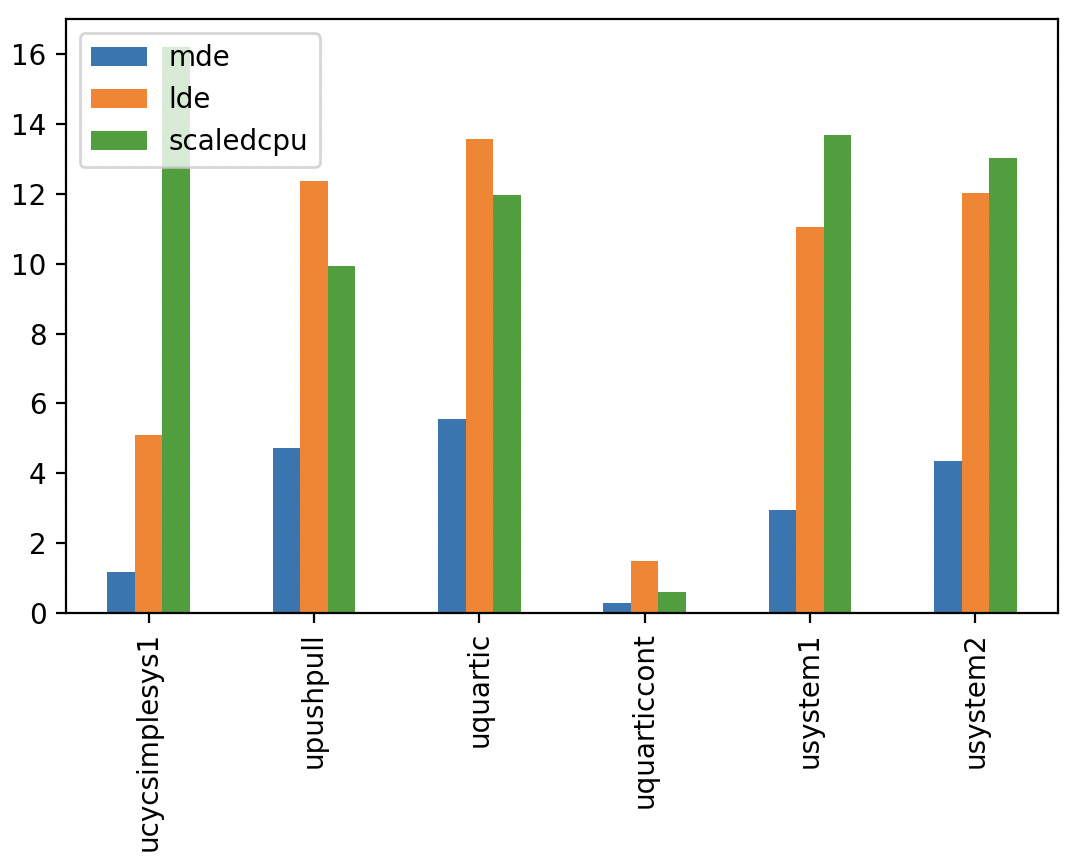}
      \end{minipage}
    \end{tabular}
  \end{center}
  \caption{Average results on formulation type for the graph class $\mathcal{R}$ on UDGP.}
  \label{t:udgp:Gftp}
\end{table}

In particular, the fact that a global MINLP solver terminates naturally on 40 (instances, formulation) pairs, but only on 6 does it find an optimum that looks close enough to a global optimum, denotes a high level of degeneration and a general lack of constraint qualification conditions in the local optimization algorithms implemented by the local NLP subsolver(s) within the global MINLP solver \cite{waechter,leoinbook}. We therefore conclude that the current state of affairs with MINLP formulations and their global solvers is not mature enough for the UDGP. The only usable formulation is continuous nonconvex NLP \textsf{uquarticcont}. 

\subsubsection{The 309-graph collection $\mathcal{G}$}
\label{s:udgp:gph}
As in the DGP case (see Sect.~\ref{s:dgp:gph}), we use this benchmark to verify whether it is better to use the matrix formulation process or not. This test, however, involves some differences with respect to the DGP case.
\begin{enumerate}
\item We established in Sect.~\ref{s:udgp:euclgph} that non-matrix MINLP formulations of the UDGP are ill-behaved to the point of being next to useless, even when the solver achieves a natural termination to what should be a global optimum (but which most of the times is not). This leaves \textsf{uquarticcont} as the only non-matrix formulation. We solve it using a local NLP solver within the MS algorithm (Alg.~\ref{a:ms}), since a global nonconvex NLP solver would be too time-consuming.
\item Local optima of MINLP may be as hard to find (both theoretically and practically) as global optima, which prevents the use of ``refinement'' in practice. Again, this leaves \textsf{uquarticcont} (a nonconvex NLP) as the only possible alternative in the area of non-matrix formulations of the UDGP: since \textsf{quarticcont} is continuous, local optima can be found rapidly. The downside is that the $y$ variables are likely to attain non-integral values at local optima $(X',y')$, which prevents the construction of the graph $G_\alpha$.\label{point2}
\item As mentioned in Sect.~\ref{s:udgp:post}, we have two possible matrix formulation processes. The first uses the matrix solution $X'$ followed by rank reduction and refinement over the reconstructed graph $G_\alpha$. The second discards $X'$ and solves the DGP instance $(K,G_\alpha)$. In the first case, refinement only involves a single call to a local NLP solver from the starting point $x'$ obtained using rank reduction. In the second case we can use any NLP solver with any DGP formulation. Given point \ref{point2}.~above, we focus on the second process: we use matrix formulations only in order to reconstruct $G_\alpha$, and then solve the resulting DGP instance.
\end{enumerate}

The test we run therefore consists in comparing:
\begin{enumerate}[A.]
\item results on \textsf{uquarticcont} solved using a a local NLP solver (IPOPT) within MS, and
\item results from matrix formulations yielding a DGP instance $(K,G_\alpha)$ solved using the \textsf{quartic} formulation (chosen as the best non-matrix formulation in Sect.~\ref{s:dgp:gph}) by means of the local NLP solver IPOPT within MS (5 iterations) acting on the \textsf{quartic} formulation. 
\end{enumerate}
This comparison is reported in Table \ref{t:udgp:gph:form}, with average performance measures aggregated by formulation type. We denote the mixed-integer SDP (MISDP) formulation used in the UDGP by \textsf{umisdp}, the mixed-integer DDP (MIDDP) matrix formulation by \textsf{umiddp}, and the mixed-integer dual DDP matrix formulation by \textsf{umidualddp}. 
\begin{table}[!ht]
  \begin{center}
    \begin{tabular}{lr}
      \begin{minipage}{0.50\textwidth}
        \begin{tabular}{l|rrr}
          \textbf{formulation} & \textsf{mde} & \textsf{lde} & CPU \\ \hline
          \textsf{umiddp} & 0.0360 & 0.2149 & 690.43 \\
          \textsf{umidualddp} & 0.1387 & 0.6687 & 1791.72 \\
          \textsf{umisdp} & 0.0914 & 0.4381 & 1862.70 \\
          \textsf{uquarticcont} & 0.0889 & 0.3045 & 17416.50 
        \end{tabular} 
      \end{minipage} & 
      \begin{minipage}{0.48\textwidth}
        \includegraphics[width=1\textwidth]{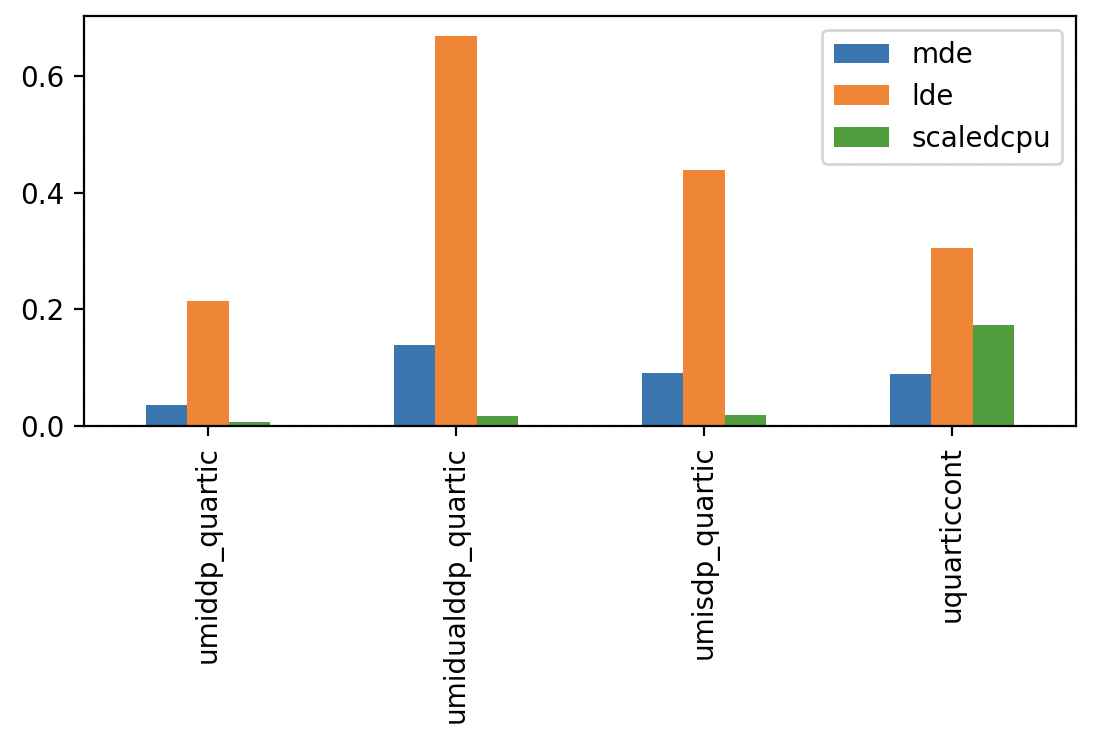}
      \end{minipage}
    \end{tabular}
  \end{center}
  \caption{Average results on formulation type for the graph class $\mathcal{G}$ on UDGP.}
  \label{t:udgp:gph:form}
\end{table}
The best performance in both solution quality and CPU time is given by \textsf{umiddp\_quartic}.

We also look at other result aggregations: by vertex cardinality (Table \ref{t:udgp:gph:vtx}), by edge cardinality (Table \ref{t:udgp:gph:edge}), and by graph similarity (Table \ref{t:udgp:gph:gphsim}). We recall that the graph similarity score \textsf{gphsim} measures the success of the graph reconstruction step from distance values in UDGPs (see Alg.~\ref{alg:gphsim}).

\begin{table}[!ht]
  \begin{center}
    \begin{tabular}{lr}
      \begin{minipage}{0.45\textwidth}
        \begin{tabular}{l|rrr}
          $\approx |V|$ & \textsf{mde} & \textsf{lde} & CPU \\ \hline
          3 & 0.0000 & 0.0000 & 8.82 \\
          4 & 0.0000 & 0.0000 & 8.84 \\
          5 & 0.0070 & 0.0140 & 8.85 \\
          6 & 0.0098 & 0.0471 & 9.53 \\
          7 & 0.0137 & 0.0325 & 9.11 \\
          9 & 0.0074 & 0.0304 & 11.54 \\
          10 & 0.0093 & 0.0447 & 10.26 \\
          15 & 0.0097 & 0.0671 & 462.55 \\
          16 & 0.0034 & 0.0333 & 238.49 \\
          20 & 0.0431 & 0.2903 & 395.46 \\
          21 & 0.0111 & 0.0784 & 486.29 \\
          25 & 0.0045 & 0.0339 & 331.86 \\
          28 & 0.0065 & 0.0605 & 681.86 \\
          35 & 0.1135 & 0.5495 & 1851.65 \\
          36 & 0.0032 & 0.0318 & 358.64 \\
          40 & 0.0000 & 0.0006 & 43.73 \\
          49 & 0.0049 & 0.0622 & 1886.85 \\
          50 & 0.1943 & 0.7452 & 7485.38 \\
          60 & 0.0005 & 0.0156 & 370.49 \\
          70 & 0.0012 & 0.0209 & 1420.22 \\
          100 & 0.0093 & 0.1128 & 17656.13 \\
          105 & 0.0114 & 0.1419 & 18530.49 \\
          150 & 0.0757 & 0.3191 & 227135.50
          \end{tabular} 
      \end{minipage} & 
      \begin{minipage}{0.53\textwidth}
        \includegraphics[width=\textwidth,height=0.43\textheight]{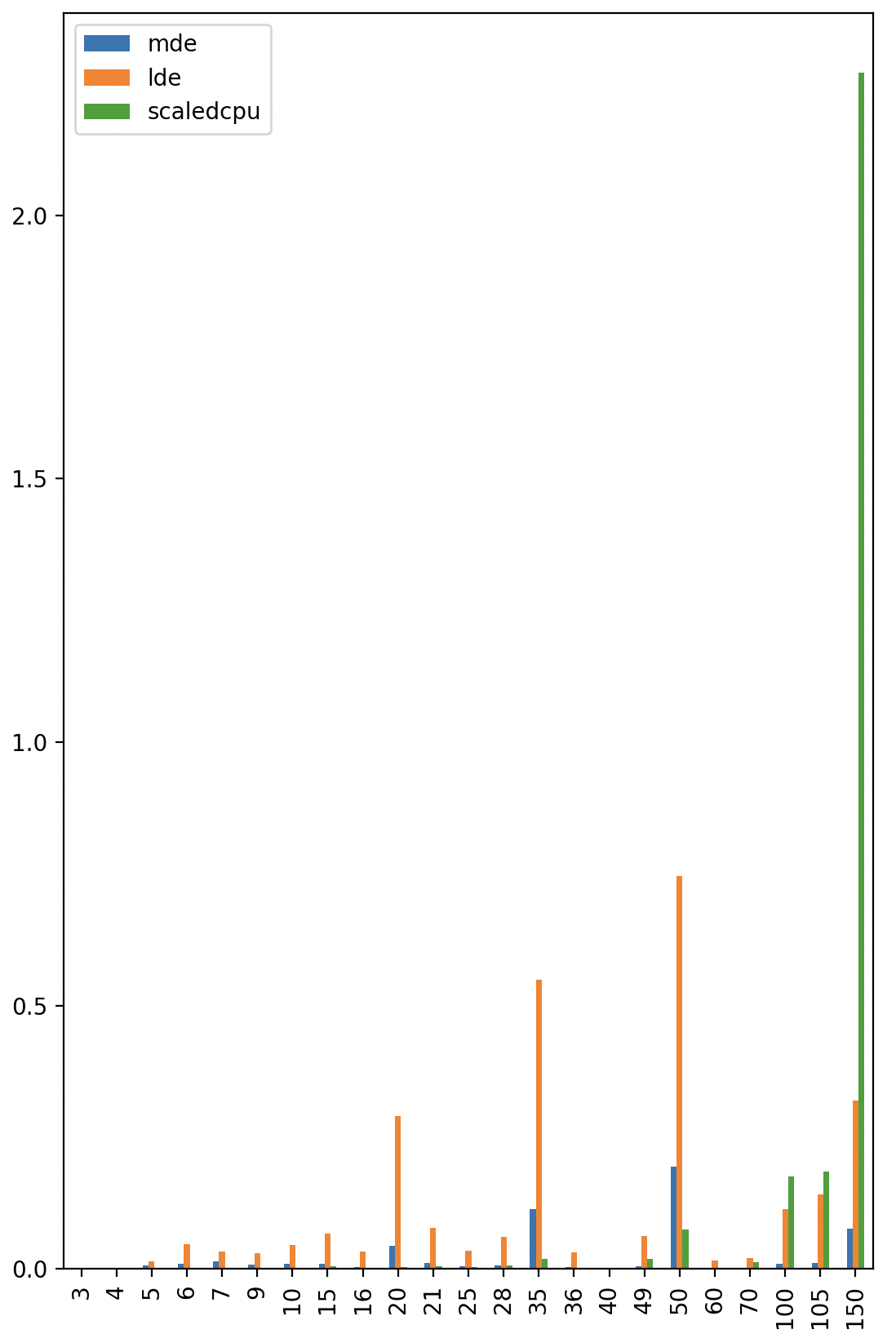}
      \end{minipage}
    \end{tabular}
  \end{center}
  \caption{Average results on $|V|$ for the graph class $\mathcal{G}$ on UDGP.}
  \label{t:udgp:gph:vtx}
\end{table}

\begin{table}[!ht]
  \begin{center}
    \begin{tabular}{lr}
      \begin{minipage}{0.45\textwidth}
        \begin{tabular}{l|rrr}
          $\approx |E|$ & \textsf{mde} & \textsf{lde} & CPU \\ \hline
          50 & 0.0055 & 0.0416 & 150.87 \\
          100 & 0.0108 & 0.0937 & 991.03 \\
          150 & 0.0279 & 0.2748 & 6839.26 \\
          200 & 0.0985 & 0.6401 & 3502.28 \\
          250 & 0.0238 & 0.1488 & 34620.31 \\
          300 & 0.0481 & 0.2441 & 26749.03 \\
          350 & 1.2046 & 6.4760 & 5869.55 \\
          400 & 0.0885 & 0.3493 & 12939.00 \\
          450 & 0.0449 & 0.1929 & 16695.82 \\
          500 & 0.0762 & 0.3332 & 21461.61 \\
          550 & 1.0753 & 4.1404 & 15175.19 \\
          650 & 2.4117 & 9.8750 & 39603.21 \\
          700 & 1.3523 & 7.3373 & 39099.54 \\
          1100 & 7.4280 & 22.6607 & 23115.13 \\
          1150 & 4.7297 & 13.1730 & 23067.89
        \end{tabular} 
      \end{minipage} & 
      \begin{minipage}{0.53\textwidth}
        \includegraphics[width=\textwidth,height=0.3\textheight]{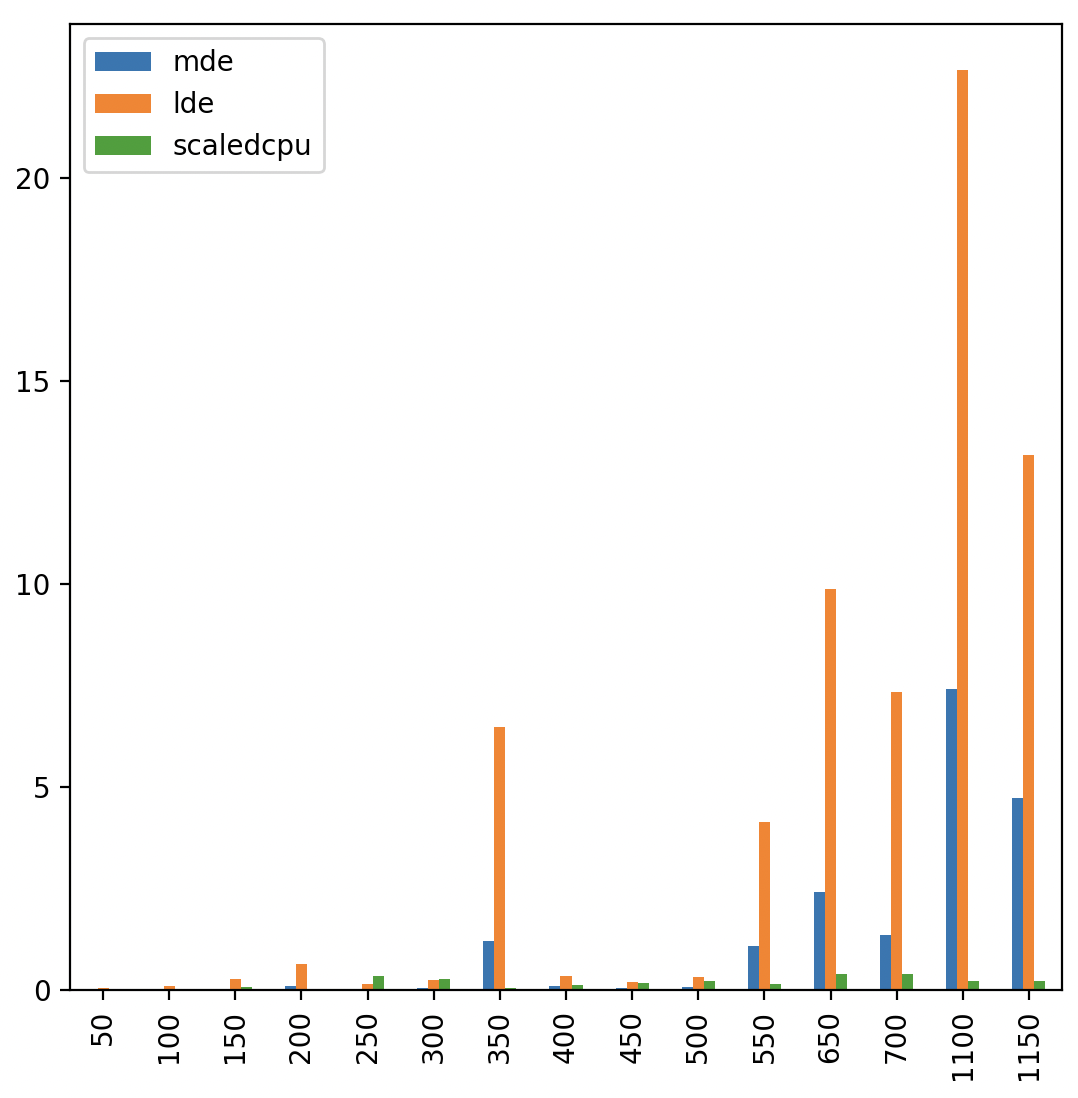}
      \end{minipage}
    \end{tabular}
  \end{center}
  \caption{Average results on $|E|$ for the graph class $\mathcal{G}$ on UDGP.}
  \label{t:udgp:gph:edge}
\end{table}

\begin{table}[!ht]
  \begin{center}
    \begin{tabular}{lr}
      \begin{minipage}{0.45\textwidth}
        \begin{tabular}{l|rrr}
          $\approx$ \texttt{gphsim} & \textsf{mde} & \textsf{lde} & CPU \\ \hline
          0.0 & 0.0062 & 0.0530 & 3741.30 \\
          0.1 & 0.0149 & 0.1102 & 4205.63 \\
          0.2 & 0.1318 & 0.6932 & 6163.91 \\
          0.3 & 0.0074 & 0.1082 & 241.70 \\
          0.4 & 2.2244 & 7.6193 & 9435.18 \\
          0.5 & 0.0025 & 0.0075 & 18.30 \\
          0.6 & 0.0070 & 0.0943 & 25506.97 \\
          0.7 & 0.0000 & 0.0000 & 121338.65 \\
          0.8 & 0.0080 & 0.0995 & 9913.92 \\
          0.9 & 0.0113 & 0.0968 & 11824.05 \\
          1.0 & 0.0459 & 0.2785 & 1437.42
        \end{tabular} 
      \end{minipage} & 
      \begin{minipage}{0.53\textwidth}
        \includegraphics[width=\textwidth,height=0.21\textheight]{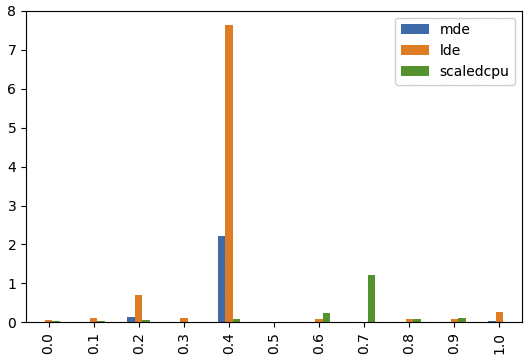}
      \end{minipage}
    \end{tabular}
  \end{center}
  \caption{Average results on graph similarity (\texttt{gphsim}) for the graph class $\mathcal{G}$ on UDGP.}
  \label{t:udgp:gph:gphsim}
\end{table}

As in Sect.~\ref{s:dgp:gph}, for $\mathcal{G}$ we also present results aggregations according to graph type in Table \ref{t:udgp:gph:gtp}, obtaining results close to those of Table \ref{t:Ggtp}.
\begin{table}[!ht]
  \begin{center}
    \begin{minipage}{0.43\textwidth}
      \begin{tabular}{l|rrr}
        \textbf{graph type} & \textsf{mde} & \textsf{lde} & CPU \\ \hline
        Walmostreg & 0.0121 & 0.1218 & 1780.42 \\
        Wbipartite & 0.0070 & 0.0895 & 9740.23 \\
        Wcliquechain & 0.0361 & 0.2257 & 3281.06 \\
        Wcluster & 0.0364 & 0.2062 & 4254.55 \\
        Wdmdgp & 0.0360 & 0.2195 & 3533.00 \\
        Wmesh & 0.0041 & 0.0558 & 729.02 \\
        Wpowerlaw & 0.0077 & 0.0914 & 708.66 \\
        Wrandom & 0.0114 & 0.1228 & 1669.00 \\
        Wtorus & 0.0176 & 0.1138 & 1275.43 \\
        Wtriangle & 0.0200 & 0.1250 & 638.01 \\
        Wtrichain & 0.0074 & 0.0212 & 8.97 \\
        Wtripartite & 0.0399 & 0.2537 & 63168.87 \\
        almostreg & 0.0000 & 0.0000 & 283.50 \\
        beeker\_glusa & 0.0110 & 0.0333 & 9.02 \\
        bipartite & 0.0000 & 0.0000 & 3006.49 \\
        cliquechain & 0.0000 & 0.0000 & 1377.71 \\
        cluster & 0.0034 & 0.0107 & 2742.12 \\
        dmdgp & 0.0000 & 0.0000 & 1909.70 \\
        local & 0.0058 & 0.0476 & 707.85 \\
        mesh & 0.0000 & 0.0000 & 113.17 \\
        norm & 1.3154 & 5.6693 & 9234.31 \\
        powerlaw & 0.0000 & 0.0000 & 116.94 \\
        random & 0.0000 & 0.0000 & 256.40 \\
        torus & 0.0000 & 0.0000 & 327.64 \\
        triangle & 0.0000 & 0.0000 & 22.26 \\
        trichain & 0.0000 & 0.0000 & 9.05 \\
        tripartite & 0.0000 & 0.0000 & 42806.12
        \end{tabular}
    \end{minipage}
    \begin{minipage}{0.55\textwidth}
      \includegraphics[width=\textwidth]{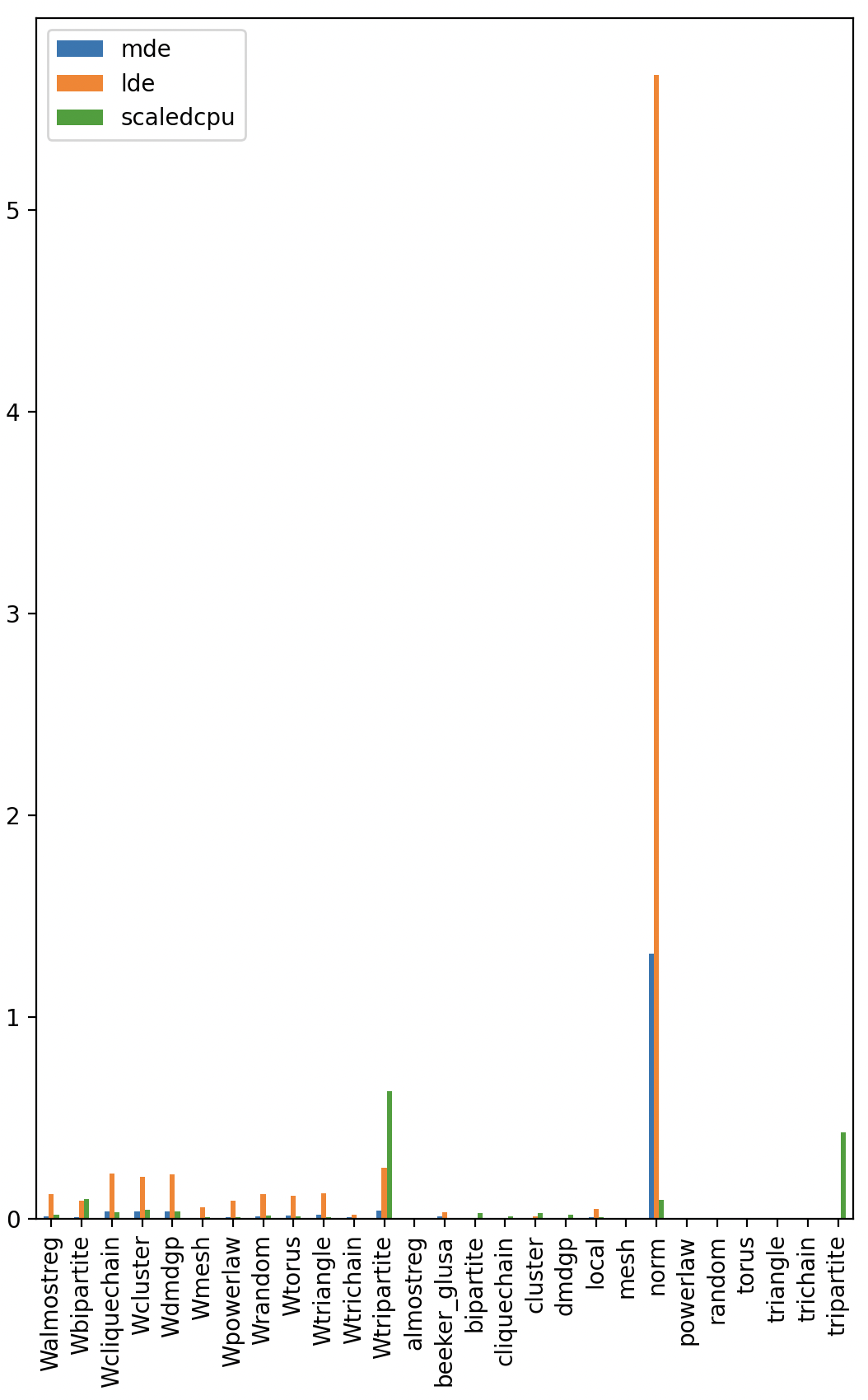}
    \end{minipage}  
  \end{center}
  \caption{Average results on graph types and the corresponding bar plot for the graph class $\mathcal{G}$ on UDGP.}
  \label{t:udgp:gph:gtp}
\end{table}

\subsubsection{The protein graph collection $\mathcal{P}$}
\label{s:udgp:prot}
We attempt to reconstruct the shape of proteins from a partial set of inter-atomic distances without their adjacencies (see Sect.~\ref{s:inst:prot} and \ref{s:inst:udgp}) by using the methods that appear more promising from earlier UDGP experiments (Sect.~\ref{s:udgp:euclgph}-\ref{s:udgp:gph}), namely: carrying out the graph reconstruction using the mixed-integer DDP and dual DDP matrix formulations, and then solving the resulting DGP instance using IPOPT within the MS algorithm (Alg.~\ref{a:ms}) with a 5 iterations limit. IPOPT was deployed on the \textsf{quartic} formulation. 

Some attempts with the usual time limit (1800s) imposed on the (matrix) MILP formulations only yielded solutions for the \textsf{tiny} instance. We therefore removed the largest protein instances from our benchmark, and only focused on a subset of six protein instances (from \textsf{tiny} up to \textsf{1guu} in Table \ref{t:protsz}). We allowed Gurobi a maximum of 12h of CPU time. Even so, with the exception of \textsf{tiny}, we only obtained solutions from the mixed-integer dual DDP formulation, which therefore becomes, \textit{de facto}, the only eligible formulation to successfully obtain approximate solutions of the UDGP on protein instances.

We remark that, in \cite{udgp}, the \textsf{uquarticcont} formulation in Eq.~\eqref{uquartic_cont} was successfully used to solve random UDGP instances (generated similarly to \cite{Lav05}) with up to 400 atoms in just over 13h of CPU time by the Baron \cite{baron} solver. But instances generated from PDB information (Sect.~\ref{s:inst:prot}) are generally more difficult --- the random instances from \cite{Lav05} have more and better distributed solutions.

Given the dearth of instances of this benchmark there is no need for aggregated and averaged results. We give full results in Table \ref{t:udgp:prot}.
\begin{table}[!ht]
  \begin{center}
    \begin{tabular}{ll|rrrr}
      \textbf{instance} & \textbf{formulation} & \texttt{gphsim} & \textsf{mde} & \textsf{lde} & CPU \\ \hline \hline
      \textsf{tiny} & \textsf{umiddp} & 0.989 & 0.000 & 0.006 & 2257.76 \\ \hline
      \textsf{tiny} & \textsf{umidualddp} & 0.216 & 0.154 & 1.962 & 58.53  \\
      \textsf{1guu-1} & \textsf{umidualddp} & 0.917 & 0.126 & 1.513 & 335825.94 \\
      \textsf{1guu-4000} & \textsf{umidualddp} & 0.929 & 6.429 & 10.300 & 4143.67 \\
      \textsf{C0030pkl} & \textsf{umidualddp} & 0.891 & 7.912 & 11.704 & 25833.85 \\
      \textsf{1PPT} & \textsf{umidualddp} & 0.034 & 10.579 & 14.236 & 20248.91 \\
      \textsf{1guu} & \textsf{umidualddp} & 0.619 & 13.307 & 16.749 & 14289.60 \\
    \end{tabular} 
  \end{center}
  \caption{Results for (part of) the graph class $\mathcal{G}$ on UDGP.}
  \label{t:udgp:prot}
\end{table}
The results in table \ref{t:udgp:prot} appear to indicate that the MIDDP formulations applied to proteins are either infeasible or have few solutions that are hard to find. Instead, dual DDP relaxations are always feasible whenever the originating SDP relaxation is feasible, and so are their mixed-integer counterparts, provided that the given distance values are compatible with an $\ell_2$ metric: this may well be the reason that the \textsf{umidualddp} matrix formulation is the only one that scores some success on protein instances.

And yet, dual DDPs (both continuous and mixed-integer) generally produce indefinite matrix solutions, which, after dimensional reduction, contain considerable error. Even if we only employ the \textsf{umidualddp} formulation to reconstruct the graph from the values of the assignment variables $y$, we would expect an indefinite matrix to carry more reconstruction error than a PSD one. And, indeed, most of the larger instances have disappointingly large \textsf{mde} and \textsf{lde} error measures. While a part of this error can be attributed to the MS algorithm configured with only 5 iterations, another is certainly due to poor graph reconstruction. That no instance displays a graph reconstruction similarity (\texttt{gphsim}) value of exactly $1.0$ (i.e., a perfect reconstruction) is to be expected: there will be many graphs compatible with the same distance values. But not all of these graphs will be realizable in $K=3$ dimensions. We can see this in the \textsf{tiny} instance, the only one that could be solved by both \textsf{umiddp} and \textsf{umidualddp}. The \textsf{umiddp} solution has a $\mathtt{gphsim}$ score that is very close to $1.0$, and almost zero \textsf{mde} and \textsf{lde} error measures. On the contrary, the \textsf{umidualddp} solution has a $\mathtt{gphsim}$ score of $0.216$ and considerably large error measures given the small instance size (38 atoms). 

The only encouraging result is \textsf{1guu-1}, the size of which is nontrivial (150 atoms). The \textsf{umidualddp} formulation provided a $\mathtt{gphsim}$ score of $0.917$, and IPOPT, taking over 92h of CPU time, was able to find a solution with tolerable error measures. The rest of the results show error measures that denote bad graph reconstructions and realizations.

\section{Conclusion}
\label{s:concl}
In this survey we have surveyed MP formulation-based methods for solving distance geometry problems, even when the input is a list of distance values instead of a weighted graph (i.e., the distances are not assigned to graph edges). The computational benchmarks established that while formulation-based methodologies are useful to solve even fairly large DGP instances (derived from protein data) in 3D, similar methodologies are not able to solve UDGP instances of the same size. 

\bibliographystyle{plain}
\bibliography{dr1}

\begin{thebibliography}{10}

\bibitem{ahmadimajumdar}
A.~Ahmadi and A.~Majumdar.
\newblock {DSOS} and {SDSOS} optimization: {M}ore tractable alternatives to sum
  of squares and semidefinite optimization.
\newblock {\em SIAM Journal on Applied Algebra and Geometry}, 3(2):193--230,
  2019.

\bibitem{alencar2}
J.~Alencar, C.~Lavor, and L.~Liberti.
\newblock Realizing euclidean distance matrices by sphere intersection.
\newblock {\em Discrete Applied Mathematics}, 256:5--10, 2019.

\bibitem{fcbmmor}
E.~Amaldi, L.~Liberti, F.~Maffioli, and N.~Maculan.
\newblock Edge-swapping algorithms for the minimum fundamental cycle basis
  problem.
\newblock {\em Mathematical Methods of Operations Research}, 69:205--223, 2009.

\bibitem{bahr}
A.~Bahr, J.~Leonard, and M.~Fallon.
\newblock Cooperative localization for autonomous underwater vehicles.
\newblock {\em International Journal of Robotics Research}, 28(6):714--728,
  2009.

\bibitem{barker2}
G.~Barker and D.~Carlson.
\newblock Cones of diagonally dominant matrices.
\newblock {\em Pacific Journal of Mathematics}, 57(1):15--32, 1975.

\bibitem{barvinok2}
A.~Barvinok.
\newblock Measure concentration in optimization.
\newblock {\em Mathematical Programming}, 79:33--53, 1997.

\bibitem{dgpinnp}
N.~Beeker, S.~Gaubert, C.~Glusa, and L.~Liberti.
\newblock Is the distance geometry problem in {NP}?
\newblock In A.~Mucherino, C.~Lavor, L.~Liberti, and N.~Maculan, editors, {\em
  Distance Geometry: Theory, Methods, and Applications}, pages 85--94.
  Springer, New York, 2013.

\bibitem{couenne}
P.~Belotti, J.~Lee, L.~Liberti, F.~Margot, and A.~W\"achter.
\newblock Branching and bounds tightening techniques for non-convex {MINLP}.
\newblock {\em Optimization Methods and Software}, 24(4):597--634, 2009.

\bibitem{berger}
B.~Berger, J.~Kleinberg, and T.~Leighton.
\newblock Reconstructing a three-dimensional model with arbitrary errors.
\newblock {\em Journal of the ACM}, 46(2):212--235, 1999.

\bibitem{pdb}
H.~Berman, J.~Westbrook, Z.~Feng, G.~Gilliland, T.~Bhat, H.~Weissig, I.N.
  Shindyalov, and P.~Bourne.
\newblock The protein data bank.
\newblock {\em Nucleic Acid Research}, 28:235--242, 2000.

\bibitem{julia}
J.~Bezanson, A.~Edelman, S.~Karpinski, and V.~Shah.
\newblock Julia: A fresh approach to numerical computing.
\newblock {\em SIAM Review}, 59(1):65--98, 2017.

\bibitem{dg-4or}
S.~Billinge, P.~Duxbury, D.~Gon\c{c}alves, C.~Lavor, and A.~Mucherino.
\newblock Assigned and unassigned distance geometry: {A}pplications to
  biological molecules and nanostructures.
\newblock {\em 4OR}, 14:337--376, 2016.

\bibitem{biswasphd}
P.~Biswas.
\newblock {\em Semidefinite programming approaches to distance geometry
  problems}.
\newblock PhD thesis, Stanford University, 2007.

\bibitem{biswasacm}
P.~Biswas, T.~Lian, T.~Wang, and Y.~Ye.
\newblock Semidefinite programming based algorithms for sensor network
  localization.
\newblock {\em ACM Transactions in Sensor Networks}, 2:188--220, 2006.

\bibitem{cordone}
M.~Bruglieri, R.~Cordone, and L.~Liberti.
\newblock Maximum feasible subsystems of distance geometry constraints.
\newblock {\em Journal of Global Optimization}, 83:29--47, 2022.

\bibitem{burer2}
S.~Burer.
\newblock On the copositive representation of binary and continuous nonconvex
  quadratic programs.
\newblock {\em Mathematical Programming A}, 120:479--495, 2009.

\bibitem{bipbip}
A.~Cassioli, B.~Bordeaux, G.~Bouvier, A.~Mucherino, R.~Alves, L.~Liberti,
  M.~Nilges, C.~Lavor, and T.~Malliavin.
\newblock An algorithm to enumerate all possible protein conformations
  verifying a set of distance constraints.
\newblock {\em BMC Bioinformatics}, 16:23--38, 2015.

\bibitem{ipopt}
COIN-OR.
\newblock {\em Introduction to IPOPT: A tutorial for downloading, installing,
  and using IPOPT}, 2006.

\bibitem{CH88}
G.~Crippen and T.~Havel.
\newblock {\em Distance Geometry and Molecular Conformation}.
\newblock Wiley, New York, 1988.

\bibitem{cvetkovic}
D.~Cvetkovi\'c, P.~Rowlinson, and S.~Simi\'c.
\newblock {\em An introduction to the theory of graph spectra}.
\newblock CUP, Cambridge, 2010.

\bibitem{oneinfnorm-lncs}
C.~D'Ambrosio and L.~Liberti.
\newblock Distance geometry in linearizable norms.
\newblock In F.~Nielsen and F.~Barbaresco, editors, {\em Geometric Science of
  Information}, volume 10589 of {\em LNCS}, pages 830--838, Berlin, 2017.
  Springer.

\bibitem{zoo}
C.~D'Ambrosio, Ky~Vu, C.~Lavor, L.~Liberti, and N.~Maculan.
\newblock New error measures and methods for realizing protein graphs from
  distance data.
\newblock {\em Discrete and Computational Geometry}, 57(2):371--418, 2017.

\bibitem{isco16}
G.~Dias and L.~Liberti.
\newblock Diagonally dominant programming in distance geometry.
\newblock In R.~Cerulli, S.~Fujishige, and R.~Mahjoub, editors, {\em
  International Symposium in Combinatorial Optimization}, volume 9849 of {\em
  LNCS}, pages 225--236, New York, 2016. Springer.

\bibitem{vetterli}
I.~Dokmani\'c, R.~Parhizkar, J.~Ranieri, and M.~Vetterli.
\newblock Euclidean distance matrices: Essential theory, algorithms and
  applications.
\newblock {\em IEEE Signal Processing Magazine}, 1053-5888:12--30, Nov. 2015.

\bibitem{udgp}
P.~Duxbury, C.~Lavor, L.~Liberti, and L.~{de Salles-Neto}.
\newblock Unassigned distance geometry and molecular conformation problems.
\newblock {\em Journal of Global Optimization}, 83:73--82, 2022.

\bibitem{ampl}
R.~Fourer and D.~Gay.
\newblock {\em The {AMPL} Book}.
\newblock Duxbury Press, Pacific Grove, 2002.

\bibitem{gershgorin}
S.~Gershgorin.
\newblock \"uber die {A}bgrenzung der {E}igenwerte einer {M}atrix.
\newblock {\em Zvesti Akademii Nauk SSSR. Otdelenie Fizicheskikh i
  Matematicheskikh Nauk}, 6:749–754, 1931.

\bibitem{idgpsurvey}
D.~Gon\c{c}alves, A.~Mucherino, C.~Lavor, and L.~Liberti.
\newblock Recent advances on the interval distance geometry problem.
\newblock {\em Journal of Global Optimization}, 69:525--545, 2017.

\bibitem{gurobi}
{Gurobi Optimization, LLC}.
\newblock {Gurobi Optimizer Reference Manual}, 2023.

\bibitem{wuthrich}
T.~Havel and K.~W\"uthrich.
\newblock An evaluation of the combined use of nuclear magnetic resonance and
  distance geometry for the determination of protein conformations in solution.
\newblock {\em Journal of Molecular Biology}, 182(2):281--294, 1985.

\bibitem{Hen95}
B.~Hendrickson.
\newblock The molecule problem: exploiting structure in global optimization.
\newblock {\em SIAM Journal on Optimization}, 5:835--857, 1995.

\bibitem{hotelling}
H.~Hotelling.
\newblock Analysis of a complex of statistical variables into principal
  components.
\newblock {\em Journal of Educational Psychology}, 24(6):417--441, 1933.

\bibitem{cplex221}
IBM.
\newblock {\em {ILOG CPLEX} 22.1 {U}ser's {M}anual}.
\newblock IBM, 2022.

\bibitem{liga}
P.~Juh\'as, D.~Cherba, P.~Duxbury, W.~Punch, and S.~Billinge.
\newblock \textit{Ab initio} determination of solid-state nanostructure.
\newblock {\em Nature}, 440(30):655--658, 2006.

\bibitem{dgpwords}
S.~Khalife, D.~Gon\c{c}alves, and Leo Liberti.
\newblock Distance geometry for word representations and applications.
\newblock {\em Journal of Computational Mathematics and Data Science},
  6:100073, 2023.

\bibitem{krislocksiam}
N.~Krislock and H.~Wolkowicz.
\newblock Explicit sensor network localization using semidefinite
  representations and facial reductions.
\newblock {\em SIAM Journal on Optimization}, 20:2679--2708, 2010.

\bibitem{sobolopt}
S.~Kucherenko and Yu. Sytsko.
\newblock Application of deterministic low-discrepancy sequences in global
  optimization.
\newblock {\em Computational Optimization and Applications}, 30(3):297--318,
  2004.

\bibitem{Lav05}
C.~Lavor.
\newblock On generating instances for the molecular distance geometry problem.
\newblock In Liberti and Maculan \cite{leonelson}, pages 405--414.

\bibitem{lln1}
C.~Lavor, L.~Liberti, and N.~Maculan.
\newblock Computational experience with the molecular distance geometry
  problem.
\newblock In J.~Pint\'er, editor, {\em Global Optimization: Scientific and
  Engineering Case Studies}, pages 213--225. Springer, Berlin, 2006.

\bibitem{dmdgpejor}
C.~Lavor, L.~Liberti, N.~Maculan, and A.~Mucherino.
\newblock Recent advances on the discretizable molecular distance geometry
  problem.
\newblock {\em European Journal of Operational Research}, 219:698--706, 2012.

\bibitem{skiena}
P.~Lemke, S.~Skiena, and W.~Smith.
\newblock Reconstructing sets from interpoint distances.
\newblock In B.~Aronov and {\it et al.}, editors, {\em Discrete and
  Computational Geometry}, volume~25 of {\em Algorithms and Combinatorics},
  pages 597--631, Berlin, 2003. Springer.

\bibitem{leoinbook}
L.~Liberti.
\newblock Writing global optimization software.
\newblock In Liberti and Maculan \cite{leonelson}, pages 211--262.

\bibitem{refmathprog}
L.~Liberti.
\newblock Reformulations in mathematical programming: Definitions and
  systematics.
\newblock {\em RAIRO-RO}, 43(1):55--86, 2009.

\bibitem{undecminlp}
L.~Liberti.
\newblock Undecidability and hardness in mixed-integer nonlinear programming.
\newblock {\em RAIRO-Operations Research}, 53:81--109, 2019.

\bibitem{dgds}
L.~Liberti.
\newblock Distance geometry and data science.
\newblock {\em TOP}, 28:271--339, 2020.

\bibitem{dg4wv}
L.~Liberti.
\newblock A new distance geometry method for constructing word and sentence
  vectors.
\newblock In {\em Companion Proceedings of the Web Conference (DL4G Workshop)},
  volume~20 of {\em WWW}, New York, 2020. ACM.

\bibitem{udgp_buckminster}
L.~Liberti.
\newblock Unassigned distance geometry and the {B}uckminsterfullerene.
\newblock AIRO-ODS 2024 Conference, Cham, accepted. Springer.

\bibitem{arschapter}
L.~Liberti, S.~Cafieri, and F.~Tarissan.
\newblock Reformulations in mathematical programming: {A} computational
  approach.
\newblock In A.~Abraham, A.-E. Hassanien, P.~Siarry, and A.~Engelbrecht,
  editors, {\em Foundations of Computational Intelligence Vol.~3}, number 203
  in Studies in Computational Intelligence, pages 153--234. Springer, Berlin,
  2009.

\bibitem{dgp-ojmo}
L.~Liberti, G.~Iommazzo, C.~Lavor, and N.~Maculan.
\newblock Cycle-based formulations in distance geometry.
\newblock {\em Open Journal of Mathematical Optimization}, 4:art.~1, 16p.,
  2023.

\bibitem{six}
L.~Liberti and C.~Lavor.
\newblock Six mathematical gems in the history of distance geometry.
\newblock {\em International Transactions in Operational Research},
  23:897--920, 2016.

\bibitem{lln5}
L.~Liberti, C.~Lavor, and N.~Maculan.
\newblock A branch-and-prune algorithm for the molecular distance geometry
  problem.
\newblock {\em International Transactions in Operational Research}, 15:1--17,
  2008.

\bibitem{dgp-sirev}
L.~Liberti, C.~Lavor, N.~Maculan, and A.~Mucherino.
\newblock Euclidean distance geometry and applications.
\newblock {\em SIAM Review}, 56(1):3--69, 2014.

\bibitem{mdgpsurvey}
L.~Liberti, C.~Lavor, A.~Mucherino, and N.~Maculan.
\newblock Molecular distance geometry methods: from continuous to discrete.
\newblock {\em International Transactions in Operational Research}, 18:33--51,
  2010.

\bibitem{leonelson}
L.~Liberti and N.~Maculan, editors.
\newblock {\em Global Optimization: from Theory to Implementation}.
\newblock Springer, Berlin, 2006.

\bibitem{barvinok_orl}
L.~Liberti and K.~Vu.
\newblock Barvinok's naive algorithm in distance geometry.
\newblock {\em Operations Research Letters}, 46:476--481, 2018.

\bibitem{pajarito}
M.~Lubin, E.~Yamangil, R.~Bent, and J.P. Vielma.
\newblock Extended formulations in mixed-integer convex programming.
\newblock In Q.~Louveaux and M.~Skutella, editors, {\em Integer Programming and
  Combinatorial Optimization (Proceedings of IPCO16)}, number 9682 in LNCS,
  pages 102--113, New York, 2016. Springer.

\bibitem{ye}
A.~{Man-Cho So} and Y.~Ye.
\newblock Theory of semidefinite programming for sensor network localization.
\newblock {\em Mathematical Programming B}, 109:367--384, 2007.

\bibitem{mwu}
L.~Mencarelli, Y.~Sahraoui, and L.~Liberti.
\newblock A multiplicative weights update algorithm for {MINLP}.
\newblock {\em EURO Journal on Computational Optimization}, 5:31--86, 2017.

\bibitem{mosek10}
Mosek ApS.
\newblock {\em The \texttt{mosek} manual, Version 10}, 2024.

\bibitem{motzkinstraus}
T.~Motzkin and E.~Straus.
\newblock Maxima for graphs and a new proof of a theorem of {T}ur\'an.
\newblock {\em Canadian Journal of Mathematics}, 17:533--540, 1965.

\bibitem{scs}
B.~O'Donoghue, E.~Chu, N.~Parikh, and S.~Boyd.
\newblock Operator splitting for conic optimization via homogeneous self-dual
  embedding.
\newblock {\em Journal of Optimization Theory and Applications},
  169(3):1042--1068, 2016.

\bibitem{patterson}
A.~Patterson.
\newblock A direct method for the determination of the components of
  interatomic distances in crystals.
\newblock {\em Zeitschrift f\"ur Kristallographie}, 90:517--542, 1935.

\bibitem{baron}
N.V. Sahinidis and M.~Tawarmalani.
\newblock {\em BARON 7.2.5: Global Optimization of Mixed-Integer Nonlinear
  Programs, {\em User's Manual}}, 2005.

\bibitem{saxe79}
J.~Saxe.
\newblock Embeddability of weighted graphs in $k$-space is strongly {NP}-hard.
\newblock {\em Proceedings of 17th Allerton Conference in Communications,
  Control and Computing}, pages 480--489, 1979.

\bibitem{schoen2002}
F.~Schoen.
\newblock Two-phase methods for global optimization.
\newblock In P.M. Pardalos and H.E. Romeijn, editors, {\em Handbook of Global
  Optimization}, volume~2, pages 151--177. Kluwer Academic Publishers,
  Dordrecht, 2002.

\bibitem{singer4}
A.~Singer.
\newblock Angular synchronization by eigenvectors and semidefinite programming.
\newblock {\em Applied and Computational Harmonic Analysis}, 30:20--36, 2011.

\bibitem{sippl}
M.~Sippl and H.~Scheraga.
\newblock Cayley-{M}enger coordinates.
\newblock {\em Proceedings of the National Academy of Sciences}, 83:2283--2287,
  1986.

\bibitem{skiena2}
S.~Skiena and G.~Sundaram.
\newblock A partial digest approach to restriction site mapping.
\newblock {\em Bulletin of Mathematical Biology}, 56(2):275--294, 1994.

\bibitem{takane_77}
Y.~Takane, F.~Young, and J.~De~Leeuw.
\newblock Nonmetric individual differences in multidimensional scaling: an
  alternating least squares method with optimal scaling features.
\newblock {\em Psychometrika}, 42:7--67, 1977.

\bibitem{python3}
G.~van Rossum and {\it et al.}
\newblock {\em Python Language Reference, version 3}.
\newblock Python Software Foundation, 2019.

\bibitem{waechter}
A.~W\"achter and L.~Biegler.
\newblock On the implementation of an interior-point filter line-search
  algorithm for large-scale nonlinear programming.
\newblock {\em Mathematical Programming}, 106(1):25--57, 2006.

\bibitem{wikipedia_pca}
Wikipedia.
\newblock Principal component analysis, 2019.
\newblock [Online; accessed 190726].

\bibitem{williams}
H.P. Williams.
\newblock {\em Model Building in Mathematical Programming}.
\newblock Wiley, Chichester, 4th edition, 1999.

\bibitem{wuthrich_nobel}
K.~W\"uthrich.
\newblock {NMR} studies of structure and function of biological macromolecules
  ({N}obel lecture).
\newblock {\em Angewandte Chemie}, 42:3340--3363, 2003.

\end{thebibliography}
\end{document}